\documentclass[11pt]{amsart}
\usepackage[margin=1in]{geometry}

\usepackage{amssymb}
\usepackage{amsthm}
\usepackage{amsmath}
\usepackage{mathrsfs}
\usepackage{amsbsy}
\usepackage{bm}
\usepackage{tikz}
\usepackage{array}
\usepackage{enumerate}
\usepackage{bbm}
\usepackage{comment}
\usepackage{mathtools}
\usepackage{tabu}
\usepackage{makecell} 
\usepackage{colortbl}
\usepackage{xcolor}
\usetikzlibrary{positioning,calc}
\usetikzlibrary{decorations.markings}
\usetikzlibrary{fit} 
\usepackage{booktabs}
\usepackage{csquotes}
\usepackage{bbold}

\DeclareFontFamily{U}{mathx}{}
\DeclareFontShape{U}{mathx}{m}{n}{<-> mathx10}{}
\DeclareSymbolFont{mathx}{U}{mathx}{m}{n}
\DeclareMathAccent{\widecheck}{0}{mathx}{"71}

\usepackage{microtype}
\usepackage{multicol}
\usepackage{paracol}
\usepackage{etoolbox}
\usepackage{caption}
\usepackage[nobreak,noadjust]{cite}
\usepackage{hyperref}
\hypersetup{colorlinks=true, citecolor=LightBlue, linkcolor=NewTurquoise,urlcolor=NewTurquoise} 

\usepackage[noabbrev,capitalise,nameinlink]{cleveref}
\crefname{conjecture}{Conjecture}{Conjectures}

\usepackage{letltxmacro} 

\usepackage{letltxmacro}
\usepackage{xparse}
\LetLtxMacro{\originalcite}{\cite}
\RenewDocumentCommand{\cite}{om}{\IfValueTF{#1}{\mbox{\originalcite[#1]{#2}}}{\mbox{\originalcite{#2}}}}

\DeclareRobustCommand{\fold}{\raisebox{0.2ex}{\resizebox{\width}{0.5\height}{\rotatebox{90}{\textsf{w}}}}}

\definecolor{LightBlue}{rgb}{0,0.8,1} 
\definecolor{Green}{rgb}{0,0.863,0}
\definecolor{DarkGreen}{rgb}{0,0.5,0}
\definecolor{MildGreen}{rgb}{0,0.784,0}
\definecolor{Turquoise}{rgb}{0,0.784,0.45}
\definecolor{NormalGreen}{rgb}{0,0.8,0}
\definecolor{LightGreen}{rgb}{0,0.922,0}
\definecolor{Magenta}{rgb}{1,0,0.6}
\definecolor{Yellow}{rgb}{0.95,0.95,0}
\definecolor{lavender}{rgb}{0.4,0,1}
\definecolor{peach}{rgb}{1,0.43,0.39} 
\definecolor{DarkBlue}{rgb}{0,0,0.549} 
\definecolor{Purple}{rgb}{0.863,0,1} 
\definecolor{Orange}{rgb}{1,0.7,0}
\definecolor{Teal}{rgb}{0,0.7,1}
\definecolor{Brown}{rgb}{0.5,0.294,0}
\definecolor{Gray}{rgb}{0.3,0.3,0.3}
\definecolor{NewBlue}{rgb}{0,0.2,0.9}
\definecolor{NewTurquoise}{rgb}{0.2,0,1}
\definecolor{s1Yellow}{RGB}{255,210,0}
\definecolor{s2Blue}{RGB}{0,210,255}
\definecolor{Cyan}{RGB}{0,255,255}

\newcommand{\Sort}{{\sf Sort}}

\newcommand{\Cat}{{\sf Cat}}

\newcommand{\wo}{w_\circ}

\newcommand{\w}{{\sf w}}

\newcommand{\Red}{\mathrm{Red}}

\definecolor{keywords}{RGB}{255,0,90}
\definecolor{comments}{RGB}{0,0,113}

\newcommand{\ZZ}{{\mathbb{Z}}}
\newcommand{\ff}{\mathsf{f}}
\newcommand{\g}{\mathsf{g}}

\newcommand{\Q}{\mathcal{Q}}

\newcommand{\inv}{{\mathsf{inv}}}

\newcommand{\hgt}{{\mathsf{ht}}}

\newcommand{\R}{\mathrm{R}} 

\newcommand{\HH}{\mathcal{H}}

\newcommand{\BIP}{Q}
\newcommand{\BIPs}{\mathbf{Q}} 

\newcommand{\JJ}{\mathfrak{J}}

\newcommand{\B}{\mathbf{B}}
\newcommand{\sw}{\mathsf{w}}

\def\bw{\mathbf{w}}

\NewDocumentCommand\CatpW{O{p}}{{\Cat_{#1}(W)}}

\NewDocumentCommand\Parkcp{O{p}}{{\mathcal{P}_{e,\mathbf{c}^{#1}}(W)}}

\newcommand{\Vreg}{V^{\mathrm{reg}}}
\newcommand{\bpi}{\vec{\pi}}

\crefname{problem}{problem}{problems}
\Crefname{problem}{Problem}{Problems}
\crefname{problem}{problem}{problems}
\Crefname{problem}{Problem}{Problems}
\newcommand{\Sal}{{\sf Sal}}

\newcommand{\Weak}{{\sf Weak}}
\newcommand{\NC}{{\sf NC}}

\newcommand{\MCh}{{\sf MC}}
\newcommand{\Bessis}{{\sf{BBW}}}

\newcommand{\SBDW}{\mathsf{SBDW}} 

\def\hgt{\mathrm{ht}}

\def\RR{\mathbb{R}}
\def\CC{\mathbb{C}}

\newcommand{\leqB}{\leq_{\mathrm{B}}}
\newcommand{\C}{\vec{\pi}}
\newcommand{\tC}{\widetilde{\boldsymbol{\pi}}}
\newcommand{\LL}{\mathrm{L}} 
\newcommand{\rw}{\mathrm{rw}} 
\newcommand{\yy}{\mathfrak{y}}
\newcommand{\yys}{{\mathfrak{y}^{\star}}}
\newcommand{\xx}{\mathfrak{x}}
\newcommand{\uuu}{\mathfrak{u}}
\newcommand{\vvv}{\mathfrak{v}}
\newcommand{\www}{\mathfrak{w}}
\newcommand{\zz}{\mathfrak{z}}
\newcommand{\BB}{\mathbb{B}}
\newcommand{\SSS}{\mathfrak{S}}

\newcommand{\DD}{\mathcal{D}}
\newcommand{\bs}{\mathbf{s}}
\newcommand{\bt}{\mathbf{t}}
\newcommand{\conv}{\mathrm{conv}}
 
\newcommand{\VV}{\mathfrak{V}}

\newcommand{\Faces}{\mathsf{Faces}} 
\newcommand{\Class}{\mathsf{CC}} 
\newcommand{\RRR}{\mathcal{R}}
\newcommand{\Span}{\mathrm{span}} 
\newcommand{\CL}{\mathcal{C}} 

\newcommand{\NCW}{\mathrm{NCWeak}}
\newcommand{\Perm}{\mathrm{Perm}}

\newcommand{\dT}{\mathbf{T}_{c}}
\newcommand{\dt}{\mathbf{t}_{c}}
\newcommand{\BA}{\B_{W,w_\circ}}
\newcommand{\BD}{\B_{W,c}}
\newcommand{\hB}{\mathrm{B}}
\newcommand{\ppp}{\Pi}
\newcommand{\leqD}{\leq_{\Delta}}
\newcommand{\spange}{\mathrm{span}_{\geq 0}}
\newcommand{\dd}{\delta}

\newcommand{\MCl}{\mathrm{MClus}^+}

\newcommand{\thuge}{{\LARGE huge}}
\newcommand{\tlarge}{{\Large large}}
\newcommand{\tmedium}{{\normalsize medium}}
\newcommand{\tsmall}{{\tiny small}}
\newtheorem{theorem}{Theorem}[section]
\newtheorem{proposition}[theorem]{Proposition}
\newtheorem{corollary}[theorem]{Corollary}
\newtheorem{conjecture}[theorem]{Conjecture}

\newtheorem{lemma}[theorem]{Lemma}

\theoremstyle{definition}
\newtheorem{definition}[theorem]{Definition}
\newtheorem{remark}[theorem]{Remark}
\newtheorem{example}[theorem]{Example}

\usepackage[textsize=tiny]{todonotes}
\newcommand{\dfn}[1]{\textcolor{NewBlue}{\emph{#1}}}
\newcommand{\defn}[1]{\textcolor{NewBlue}{\emph{#1}}}

\begin{document}

\title[]{Permutahedron Triangulations via Total Linear Stability \\ and the Dual Braid Group}
\subjclass[2010]{}

\author[]{Colin Defant}
\address[]{Department of Mathematics, Harvard University, Cambridge, MA 02138, USA}
\email{colindefant@gmail.com} 

\author[]{Melissa Sherman-Bennett}
\address[]{Department of Mathematics, University of California Davis, Davis, CA 95616, USA
}
\email{mshermanbennett@ucdavis.edu} 

\author[]{Nathan Williams}
\address[]{Department of Mathematical Sciences, University of Texas at Dallas, Richardson, TX 75080}
\email{nathan.f.williams@gmail.com}

\begin{abstract} For each finite Coxeter group $W$ and each standard Coxeter element of $W$, we construct a triangulation of the $W$-permutahedron. For particular realizations of the $W$-permutahedron, we show that this is a regular triangulation induced by a height function coming from the theory of total linear stability for Dynkin quivers. We also explore several notable combinatorial properties of these triangulations that relate the Bruhat order, the noncrossing partition lattice, and Cambrian congruences. Each triangulation gives an explicit mechanism for relating two different presentations of the corresponding braid group (the standard Artin presentation and Bessis's dual presentation). This is a step toward uniformly proving conjectural simple, explicit, and type-uniform presentations for the corresponding pure braid group. 
\end{abstract} 

\maketitle

\section{Introduction}\label{sec:intro} 

\subsection{Prelude}  
This article concerns new\footnotemark{} triangulations of Coxeter permutahedra into simplices coming from the interplay between Bruhat order and the lattice of noncrossing partitions.  Before defining these triangulations and considering their geometric and combinatorial properties, we set the stage by discussing the topological and algebraic concepts that motivate them. 

\subsubsection{$K(\pi,1)$ hyperplane complements}
The complexified complements of finite real simplicial hyperplane arrangements 
and the complements of hyperplane arrangements for finite well-generated complex reflection groups 
share the remarkable property of being $K(\pi,1)$ spaces: the fundamental group---called the \defn{pure braid group} 
 of the arrangement---is the only nontrivial homotopy group of the space.

\medskip

\begin{paracol}{2}
Building on Garside's single mathematical paper~\cite{garside1969braid} (``que je suis
souvent de tr\`es pr\`es''~\cite{deligne1972immeubles}), Deligne proved that the complexified complement of a real finite simplicial hyperplane arrangement is a $K(\pi,1)$ space~\cite{deligne1972immeubles}. This result had previously been proven for the reflection arrangements of types $A,B,D, F_4$, and $I_2(m)$ by Brieskorn~{\cite{brieskorn2006groupes}} (who also conjectured the result for all finite Coxeter groups).  According to van der Lek~\cite{lek1983homotopy}, Arnold, Pham, and Thom subsequently conjectured it to hold for hyperplane arrangements for all (infinite) Coxeter groups~\cite{charney1995k,paris2014k}.
\switchcolumn
Building on the structure of Deligne's proof~\cite{deligne1972immeubles} (``The general architecture of our
proof is borrowed from Deligne’s original approach, but the details are quite
different''~\cite{bessis2015finite}), Bessis proved that the complement of the hyperplane arrangement of a finite complex reflection group is a $K(\pi,1)$ space~\cite{bessis2015finite}. Various special cases~\cite{nakamura1983note,orlik1988discriminants} (including those of Shephard groups, which have Artin groups isomorphic to those of certain real groups) were known prior to Bessis's foundational work, and his ``dual'' approach has subsequently been generalized to affine Weyl groups~\cite{mccammond2017artin} and Coxeter groups of rank $3$~\cite{delucchi2024dual}. 
\end{paracol}

\subsubsection{The Salvetti and Bessis--Brady--Watt complexes}\label{sec:sal_bessis}
The proofs that these spaces are $K(\pi,1)$ rely on two different combinatorial models for the hyperplane complements.

\begin{paracol}{2}
The complexified hyperplane complement of a real simplicial arrangement has a combinatorial model called the \emph{Salvetti complex}.  In~\cite{salvetti1987topology}, by choosing a real basepoint for the fundamental group, Salvetti used the real chamber geometry of any real finite hyperplane arrangement---not just simplicial arrangements---to construct a CW complex that is a deformation retract of the complexified hyperplane complement.  
\switchcolumn \footnotetext{See~\Cref{rem:new} and~\Cref{sec:past_and_present}.}
The hyperplane complement for a finite well-generated complex reflection group has a combinatorial model we call the \emph{Bessis--Brady--Watt complex}.    In~\cite{bessis2015finite}, by choosing an eigenvector for a Coxeter element as the basepoint for the fundamental group, Bessis used the noncrossing partition lattice as a replacement for the chamber geometry for real arrangements (see also Brady and Brady--Watt~\cite{brady2000artin,brady2001partial,brady2002k}).
\end{paracol}
\medskip 

An arrangement that lies at the intersection of simplicial arrangements and well-generated complex reflection arrangements---that is, a finite Coxeter arrangement---has \emph{both} combinatorial models: that is, the hyperplane complement is homotopic to both the Salvetti complex and the Bessis--Brady--Watt complex.  

\subsubsection{The Artin and dual presentations}  

Associated to quotients of the Salvetti and Bessis--Brady--Watt complexes are the Artin and dual presentations for the braid group of the arrangement.

Let $W$ be a finite Coxeter group acting in its reflection representation on a real vector space $V$ of dimension $r$.  Write $\mathcal{H}$ for the set of its $N$ reflecting hyperplanes and $V^\mathrm{reg}:=V\setminus \bigcup\mathcal{H}$ for the hyperplane complement.  Choose a connected component of $V^\mathrm{reg}$ to call the \dfn{base region} $\BB$; this induces a bijection from $W$ to the set of connected components of $V^\mathrm{reg}$ given by $w \mapsto w\BB$.  Let $\Vreg_\CC$ be the complexification of $\Vreg$. The \dfn{pure braid group} of type $W$ is the fundamental group ${\bf P}_W:=\pi_1\!\left(\Vreg\right)$. The \dfn{braid group} of type $W$ is $\B_W:=\pi_1\!\left(V^\mathrm{reg}_\CC/W\right)$, the fundamental group of the quotient of $\Vreg_\CC$ by the natural action of $W$.   These fit into a natural short exact sequence
\[0\to {\bf P}_W \to {\bf B}_W \to W \to 0.\]

\medskip

\begin{paracol}{2}
The group $W$ is generated by its set $S$ of $r$ \dfn{simple reflections}, which act on $V$ via reflections through the walls of $\BB$. The \defn{length} $\ell_S(w)$ of an element $w \in W$ is the minimum number of simple reflections required to write $w$. Equivalently, $\ell_S(w)$ is the number of hyperplanes in $\HH$ separating $\BB$ from $w\BB$. 
\switchcolumn

The group $W$ is generated by its set $T$ of $N$ \dfn{reflections}, which act on $V$ via reflections through the hyperplanes in $\HH$. The \defn{absolute length} $\ell_T(w)$ of an element $w \in W$ is the minimum number of reflections required to write $w$.  Equivalently (as $W$ is real), $\ell_T(w)$ is the dimension of the image of $w-1$.
\switchcolumn*

The \dfn{long element} $w_\circ$ is the unique element of $W$ with $\ell_S(w_\circ)=N$. For $\yy\in \BB$, the \dfn{$W$-permutahedron} $\Perm_\yy=\Perm_\yy(W)\subseteq V$ is the polytope defined as the convex hull of the $W$-orbit $W\yy$. The 1-skeleton of $\Perm_\yy$ is the \defn{weak order}, the interval $\Weak(W):=[e,w_\circ]_S$ between the identity $e$ and the long element in the right Cayley graph of $W$ generated by $S$ (with edges labeled by simple reflections).

\switchcolumn

A \dfn{standard Coxeter element} is an element $c$ of $W$ obtained as a product of the simple reflections in some order. We have $\ell_T(c)=r= \ell_S(c)$. The order of $c$ is denoted $h$. The \dfn{$c$-noncrossing partition lattice} is the interval $\NC(W,c):=[e,c]_T$ between the identity $e$ and the Coxeter element $c$ in the right Cayley graph of $W$ generated by $T$ (with edges labeled by reflections).  Elements of $\NC(W,c)$ are called \defn{$c$-noncrossing partitions}.  

\switchcolumn*

Write $\mathbf{S}:=\{ \mathbf{s} : s \in S\}$ for a formal copy of the set $S$ of simple reflections. The \dfn{Artin presentation} of the braid group is
\[\BA = \langle \mathbf{S} : [\mathrm{Red}_S(w_\circ)] \rangle,\]
where $[\mathrm{Red}_S(w_\circ)]$ stands for the relations setting equal all reduced $S$-words for $w_\circ$ (replacing each $s \in S$ by its corresponding $\mathbf{s} \in \mathbf{S})$. 

Writing $\w_\circ$ for the image of $w_\circ$ in $\BA$, the \defn{full twist} ${\bf w}_\circ^2$ lies in the center of $\BA$.

\switchcolumn

Write $\dT:=\{ \dt : t \in T\}$ for a formal copy of the set $T$ of reflections. Bessis's \dfn{dual presentation} of the braid group is
\[\BD = \langle \dT : [\mathrm{Red}_T(c)] \rangle,\]
where $[\mathrm{Red}_T(c)]$ stands for the relations setting equal all reduced $T$-words for $c$ (replacing each $t \in T$ by its corresponding $\dt \in \dT)$.

Writing ${\bf c}$ for the image of $c$ in $\BD$, the \defn{full twist} ${\bf c}^h$ lies in the center of $\BD$.
\switchcolumn*
\end{paracol}

\medskip 

These two different presentations for the braid group come from the two different combinatorial models for the complexified hyperplane complement mentioned in~\Cref{sec:sal_bessis}.  In more detail:
\medskip
\begin{paracol}{2}
For a finite Coxeter group $W$, the Salvetti complex $\Sal(W)$ can be taken to be invariant with respect to the action of the group, giving a homotopy equivalence between the quotient complex and the orbit space~\cite{salvetti1994homotopy}.

The Artin presentation for $\B_{W}$ is associated to this \dfn{quotient Salvetti complex} $\Sal(W)/W$: informally, this complex is formed by gluing together oriented copies of the $W$-permutahedron $\Perm_\yy(W)$---one for each element of $W$---and then quotienting by the action of $W$.  The edges of this complex can be viewed as oriented edges of the permutahedron, so they are labeled by simple reflections, while relations come from higher-dimensional faces of the permutahedron.
\switchcolumn

The Bessis--Brady--Watt complex $\Bessis(W,c)$ carries a natural action of the complex reflection group $W$, and it is beautifully described as a simplicial complex built from the noncrossing partition lattice in~\cite[Examples 2.2 \& 2.3]{brady2018noncrossing}. 

Bessis's dual presentation for $\B_{W}$ is associated to this \dfn{quotient Bessis--Brady--Watt complex} $\Bessis(W,c)/W$: informally, this complex is formed by gluing together oriented copies of the order complex of the noncrossing partition lattice $\NC(W,c)$---one for each element of $W$---and then quotienting by the action of $W$.  The edges of this complex can be viewed as oriented edges in the noncrossing partition lattice, so they are labeled by reflections, while relations come from longer chains in the noncrossing partition lattice.  
\end{paracol} 

\medskip 

Our triangulation of the permutahedron (depending on a Coxeter element $c$) subdivides the Salvetti complex into simplices that resemble those used in the Bessis--Brady--Watt complex.  In~\cref{sec:presentations}, we discuss how it provides an explicit mechanism for deriving the relations in the Artin presentation from the dual presentation. Because the Salvetti and BBW complexes are both homotopy equivalent to $\Vreg_{\CC}$, it is natural to ask for an explicit homotopy between them---in \cref{sec:obstacles}, we discuss the obstacles we encountered when attempting to construct this homotopy from our triangulation.

\begin{remark}
Since one can identify $\mathbf{S}$ as a subset of $\dT$, it is reasonable to restrict the Cayley graph of the \emph{dual braid monoid} to the image of the \emph{positive Artin monoid} (see~\Cref{fig:nathan}).  After drawing some highly suggestive low-rank pictures, the third author was led to search the internet for the term ``triangulation of the permutahedron'' in the summer of 2021.  This search yielded some excellent slides by Brady, Delucchi, and Watt~\cite{slides} that appeared to have the same construction as these low-rank drawings and a similar construction to the one we present here for bipartite $c$.  (In this bipartite case, the height function is particularly nice to describe, as indicated in~\Cref{foot:bipartite}.)  After some number of years, it was suggested to us by several parties that we write up our version. 
\label{rem:new}
\end{remark} 

\subsubsection{Presentations of the pure braid group}
The Artin and dual presentations of the braid group are explicit, elegant, and type-uniform. By contrast, it is notoriously difficult to find ``nice'' presentations for the pure braid group.

Artin gave a normal form of pure braids, which yields a presentation of the pure braid group in type $A$~\cite{artin1947theory}.  This presentation is at least somewhat unsatisfactory---in Artin's own words~\cite{artin1947theory}:
\begin{displayquote}
``Although it has been proved that every braid can be deformed into a similar normal form the writer is convinced that any attempt to carry this out on a living person would only lead to violent protests and discrimination against
mathematics. He would therefore discourage such an experiment.''
\end{displayquote}
Extending the fiber-type technique from type $A$, Cohen gave a presentation of the pure braid group in type $B$~\cite[Theorem 1.4.3]{cohen2001monodromy} (see also Digne--Gomi~\cite{digne2001presentation,digne2015pr}).  Markushevich claimed to have found a presentation in type $D$ in~\cite{markushevich1991d}, but Falk and Randell later determined it was incorrect~\cite{falk2000homotopy}. These approaches share similar problems with Artin's presentation: they are type-specific and give complicated, unmotivated presentations. For example, Artin's presentation uses five families of relations, while the presentations of Cohen and Digne--Gomi each use nine families.

Salvetti's presentation in~\cite{salvetti1987topology} for the complexified complement of any real arrangement is unwieldy without very precise combinatorial control over all shards of the arrangement (see~\cite{reading2011noncrossing} and especially~\cite[Theorem 1.1]{defant2022pop}).  Nearly a decade ago, in unpublished notes, the third author conjectured explicit, elegant, type-uniform presentations for the pure braid group of $W$ relying on the interaction of Salvetti's presentation with noncrossing Coxeter--Catalan combinatorics.  He proved these presentations in types $A$, $B$, $H_3$, and $I_2(m)$ using the precise combinatorial control over noncrossing shards given by Coxeter--Catalan combinatorics.

In more detail, using the inclusion $\mathbf{P}_W \hookrightarrow \mathbf{B}_W$, write $\mathbb{t}_c:=\mathbf{t}^2_c$ for the square of the generator of the dual braid group corresponding to the reflection $t$, and let $\mathbb{T}_c:=\{\mathbb{t}_c : t \in T\}$.  It is not difficult~\cite{defant2022pop} to use Salvetti's construction together with Reading's theory of Coxeter-sorting words to see that $\mathbb{T}_c$ generates ${\bf P}_W$.  The \defn{full twist} $\mathbb{c}:={\bf w}_\circ^2=\mathbf{c}^h \in \mathbf{P}_W$ generates the center of ${\bf P}_W$.  
\begin{conjecture}[{N.~Williams}]\label{conj:pure}  For any finite Coxeter group $W$, the pure braid group has presentation
\[{\bf P}_W =  \langle \mathbb{T}_c : [\Red_{\mathbb{T}_c}(\mathbb{c})] \rangle,\]
where $[\Red_{{\mathbb{T}_c}}(\mathbb{c})]$ stands for the relations setting equal all reduced words in $\mathbb{T}_c$ for the full twist $\mathbb{c}$.
\end{conjecture}
We give more context for this conjecture in~\cref{subsec:pure}, including a conjectural characterization of $[\Red_{{\mathbb{T}_c}}(\mathbb{c})]$ as the EL-shelling orders of the noncrossing partition lattice (\Cref{conj:total_orders}).  We view our results in this article, together with recent work of the first and third authors~\cite{defant2022pop}, as steps toward resolving these conjectures.

\begin{figure}[htbp]
  \begin{center}{\includegraphics[height=5.745cm]{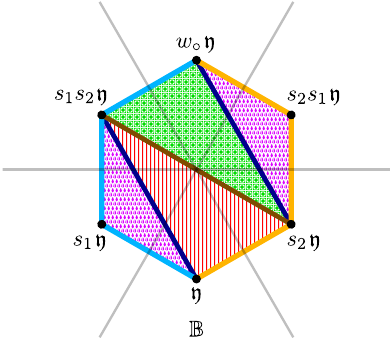}}\qquad\qquad\raisebox{-1.43cm}{\includegraphics[height=8.003cm]{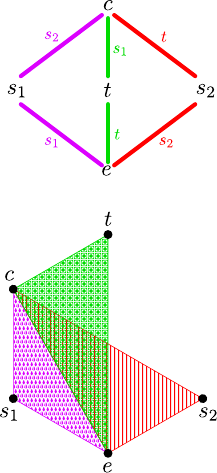}}
  \end{center}
\caption{On the top right is the $c$-noncrossing partition lattice of $\SSS_3$, where ${c=s_1s_2}$. On the bottom right is its order complex. On the left is the permutahedron $\Perm_\yy(\SSS_3)$, triangulated by $\SBDW_\yy(\SSS_3,c)$. Each simplex $\conv(w\C \yy)$ is drawn on the left in the same color as the edges in the chain $\C$ on the right. There are two ways of walking from $\yy$ up to $\wo\yy$ along the $1$-skeleton of $\Perm_\yy$; these paths are colored differently.}\label{fig:S3}
\end{figure}

\begin{figure}[htbp]
  \begin{center}\includegraphics[height=9.937cm]{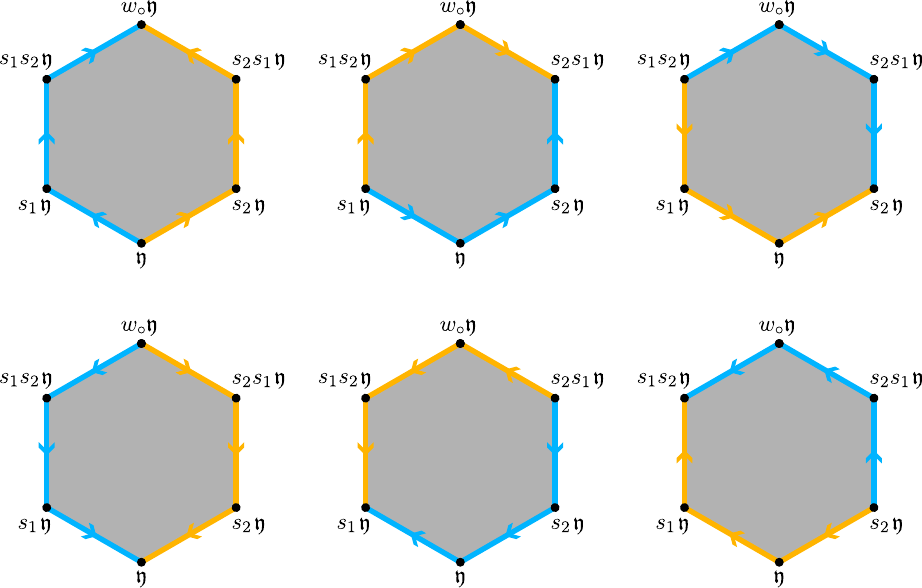}
  \end{center}
\caption{The Salvetti complex for the symmetric group $\SSS_3$ has six $0$-cells (one for each vertex of $\Perm_\yy$), twelve $1$-cells (one for each orientation of each edge of $\Perm_\yy$), and six $2$-cells (one for each of the oriented copies of $\Perm_\yy$ shown here). Quotienting by the action of $\SSS_3$ yields the quotient Salvetti complex. Each {\color{Teal}light blue} path corresponds to the expression ${\color{Teal}\mathbf{s}_1\mathbf{s}_2\mathbf{s}_1}$, while each {\color{Orange}orange} path corresponds to the expression ${\color{Orange}\mathbf{s}_2\mathbf{s}_1\mathbf{s}_2}$. One can push a {\color{Teal}light blue} path through a {\color{Gray}gray} 2-cell and onto an {\color{Orange}orange} path, obtaining a topological manifestation of the relation ${\color{Teal}\mathbf{s}_1\mathbf{s}_2\mathbf{s}_1}={\color{Orange}\mathbf{s}_2\mathbf{s}_1\mathbf{s}_2}$ in the Artin presentation of $\B_{\SSS_3}$.}\label{fig:6Salvetti}
\end{figure}

\begin{figure}[htbp]
  \begin{center}\includegraphics[height=9.937cm]{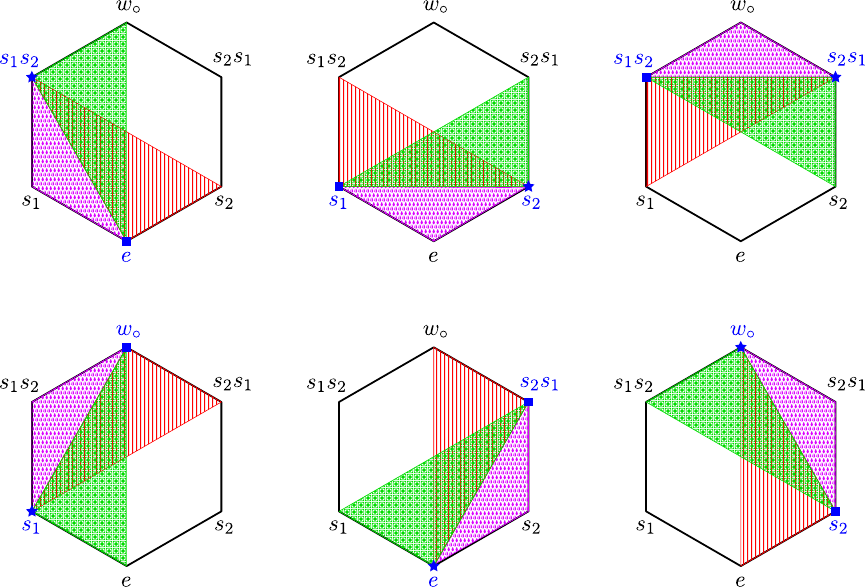}
  \end{center}
\caption{The Bessis--Brady--Watt complex $\Bessis(\SSS_3,s_1s_2)$ is a pure simplicial complex with $3!\times 3=18$ maximal (2-dimensional) simplices, shown here as colored triangles. Each edge of these simplices is oriented from some element $w\in\SSS_3$ (represented by a {\color{blue}blue} square) to the element $ws_1s_2$ (represented by a {\color{blue}blue} star). Each maximal simplex is of the form $\{w\pi_0,w\pi_1,w\pi_2\}$, where $\{\pi_0\lessdot_T \pi_1\lessdot_T \pi_2\}$ is one of the maximal chains in the noncrossing partition lattice shown on the right in \cref{fig:S3}. Quotienting by the action of $\SSS_3$ yields the quotient Bessis--Brady--Watt complex. }\label{fig:6BBW}
\end{figure}

\subsection{Triangulating the permutahedron using totally stable linear stability} 

Recall that a \dfn{subdivision} of a polytope $P$ is a collection $\Theta$ of polytopes of the same dimension as $P$ such that $\bigcup\Theta=P$ and such that for all $P_1,P_2\in\Theta$, the intersection $P_1\cap P_2$ is a (possibly empty) common face of $P_1$ and $P_2$. A subdivision $\Theta$ is a \dfn{triangulation} if all polytopes in $\Theta$ are simplices. 

Fix a standard Coxeter element $c$ of $W$ and a point $\yy\in \BB$. The \dfn{Bruhat order} is the partial order $\leqB$ on $W$ with a cover relation $u \lessdot_\hB v$ if and only if $\ell_S(v)= \ell_S(u) + 1$ and $v=ut$ for some $t \in T$. 
We will construct a triangulation of $\Perm_\yy$ by leveraging an interplay between the noncrossing partition lattice $\NC(W,c)$ and the Bruhat order.   
Let \begin{equation}W_c^+:=\{u\in W:\ell_S(uc)=\ell_S(u)+\ell_S(c)\}.\label{eq:wplus}\end{equation}  (One can show that $W_c^+$ is the set of elements less than or equal to $w_\circ c^{-1}$ in left weak order.) For each $u\in W_c^+$, the Bruhat interval $[u,uc]_\hB$---consisting of those elements between $u$ and $uc$ in the Bruhat order---is necessarily isomorphic to a subposet of $\NC(W,c)$ under the map $x \mapsto u^{-1}x$.

A maximal chain $\{u=u_0\lessdot_\hB u_1\lessdot_\hB \cdots \lessdot_\hB u_r=uc\}$ in the interval $[u,uc]_\hB$ now defines a polytope $\conv\big(\{u_0\yy,u_1\yy,\ldots,u_r\yy\}\big)$, where $\conv(\cdot)$ denotes convex hull. We will see that this polytope is in fact an $r$-dimensional simplex. 

\begin{definition}\label{def:SBDW} 
Fix a standard Coxeter element $c$ of $W$ and a point $\yy\in \BB$. We define the \dfn{Salvetti--Brady--Delucchi--Watt triangulation} (or \dfn{SBDW triangulation}) $\SBDW_\yy(W,c)$ to be the collection of simplices in the permutahedron $\Perm_\yy(W)$ of the form $\conv(\{u_0\yy,u_1\yy,\ldots,u_r\yy\})$, where $u\in W_c^+$ and ${\{u=u_0\lessdot_\hB u_1\lessdot_\hB \cdots \lessdot_\hB u_r=uc\}}$ is a maximal chain of the Bruhat interval $[u,uc]_\hB$. 
\end{definition} 

\begin{example}
\cref{fig:S3} illustrates \cref{def:SBDW} for the symmetric group \[\SSS_3=\langle s_1,s_2 : s_1s_2s_1=s_2s_1s_2, s_1^2=s_2^2=e\rangle,\] with Coxeter element $c=s_1s_2$ (where $s_i=(i\,\,i+1)$).
\end{example}

It is not obvious that $\SBDW_\yy(W,c)$ is a triangulation of the $W$-permutahedron $\Perm_\yy$. Indeed, it is not clear that the union of the simplices is the entire $W$-permutahedron, nor that two simplices do not intersect in their interior, nor that intersections occur along common faces.  Remarkably, these simplices coming from the interaction of $\NC(W,c)$ and the Bruhat order do form a triangulation of $\Perm_\yy$---to steal a phrase from Bessis, the noncrossing partition lattice is ``somehow real''~\cite{bessisslides}. 

\begin{theorem}\label{thm:triangulation} 
Let $W$ be a finite Coxeter group. For each standard Coxeter element $c$ of $W$ and each point $\yy\in\BB$, the collection of simplices $\SBDW_\yy(W,c)$ is a triangulation of the $W$-permutahedron $\Perm_\yy(W)$. 
\end{theorem} 

To prove \cref{thm:triangulation}, we first assume that $W$ is simply-laced (i.e., of type $A$, $D$, or $E$). Under this assumption, we provide a type-uniform argument that relies on recent developments in quiver representation theory. Specifically, we invoke \emph{totally stable linear stability functions}, the existence of which was established in \cite{apruzzese2020stability,HuangHu,kinser2022total} in type~$A$ and in \cite{chang2024geometric} in general.  Geometrically, specifying a totally stable linear stability function amounts to finding a point $\yy^\star \in \BB$ and a direction $\gamma$ so that shooting a laser from $\yy^\star$ in the direction of $\gamma$ crosses a certain piece of each hyperplane in the arrangement $\HH$ in a certain order (more precisely, the laser crosses the \emph{$c$-noncrossing shards} in the order specified by the inversion sequence of the $c$-sorting word for $w_\circ$; this is essentially the same data required for determining the generators in Salvetti's presentation of $\mathbf{P}_W$).  The time at which the hyperplane piece is crossed gives a height function, and the existence of such functions allow us to show that $\SBDW_{\yys}(W,c)$ is a \emph{regular} triangulation of $\Perm_{\yys}$ (for the particular $\yys\in\BB$ used in the stability function---see~\Cref{prop:regular_simply_laced}).

We also establish a result---which may be of independent interest---providing sufficient conditions under which a triangulation is preserved under a continuous deformation (see \cref{prop:deform}). Using this result, we can deform $\Perm_{\yys}$ into $\Perm_{\yy}$ for an arbitrary choice of $\yy$, thereby proving \cref{thm:triangulation} in general for simply-laced Coxeter groups.  We use \emph{folding} to deduce the theorem for the remaining finite Coxeter groups (types $B$, $F_4$, $G_2$, $H_3$, $H_4$, and $I_2(m)$ for $m\geq 4$).

\begin{remark}
The preceeding construction breaks if we define simplices using the noncrossing partition lattice coming from a \emph{nonstandard} Coxeter element $c$---that is, a Coxeter element $c$ such that $\ell_S(c)>r$.  In this case, the collection of simplices coming from taking the convex hull of maximal Bruhat-increasing chains ${\{u=u_0\lessdot_\hB u_1\lessdot_\hB \cdots \lessdot_\hB u_r=uc\}}$ will not necessarily triangulate the permutahedron, and one can have simplices that intersect in their interiors.
\end{remark}

\subsection{Combinatorics of the SBDW triangulation}\label{subsec:properties-of-triang}
While our height function gives us a coarse global understanding of the SBDW triangulation, we actually have very fine combinatorial control of the dual graph (see~\cite{nathansite} for an interactive animation, and \Cref{ex:visualization,fig:TrianglesLinear} for a static visualization for $\SSS_4$).

\begin{figure}[htbp]
\includegraphics[height=14.75cm]{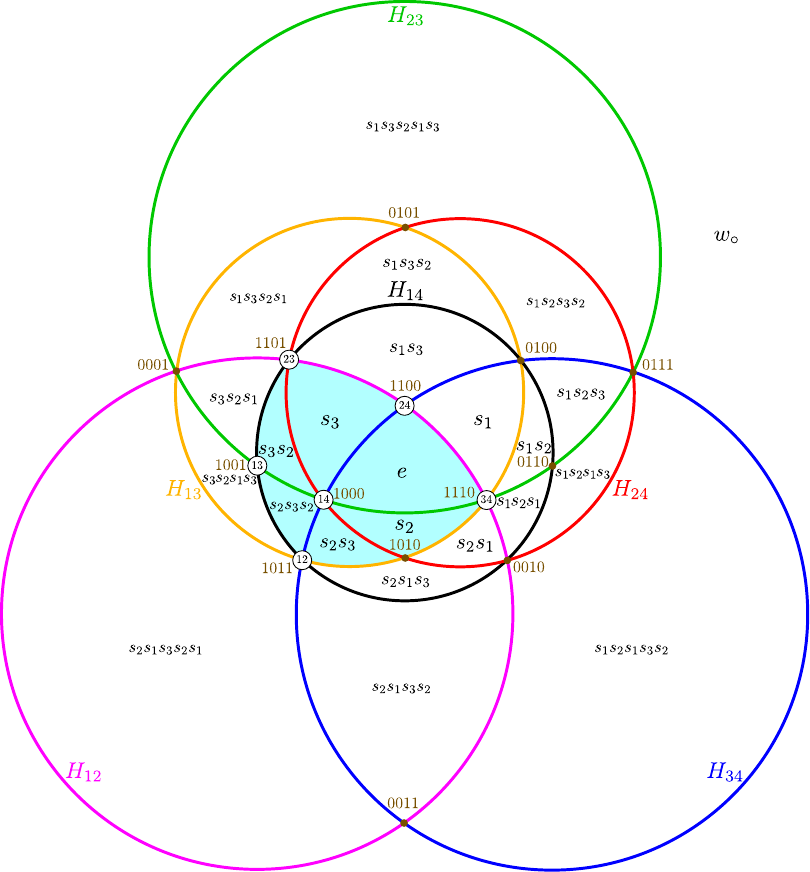}
\caption{The hyperplane arrangement $\mathcal H$ for $W=\SSS_4$, represented via a stereographic projection of great circles. The region $\Delta_{s_1s_2s_3}^+$ is shaded in {\color{Cyan}cyan}. We have labeled each region $w\BB$ by $w$. Each vertex $\dd_{(i\,j)}$ (for $(i\,j)\in T$) is represented by~\raisebox{-0.10cm}{\includegraphics[height=0.5cm]{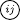}}\,. }
\label{fig:a3arrangement}
\end{figure}

To analyze the combinatorics of the SBDW triangulation, it is helpful to use the following alternative description. Let $\MCh(W,c)$ denote the set of maximal chains in the $c$-noncrossing partition lattice $\NC(W,c)$. There is a straightforward bijection from $\MCh(W,c)$ to the set $\Red_T(c)$ of reduced $T$-words for $c$ that sends the chain $\C=\{\pi_0\lessdot_T \pi_1 \lessdot_T \cdots \lessdot_T \pi_r\}$ to the word $\rw(\C)=t_1\cdots t_r$, where $t_i=\pi_{i-1}^{-1}\pi_i$.    
\begin{definition}
We say that a chain $\C=\{\pi_0\lessdot_T \pi_1 \lessdot_T \cdots \lessdot_T \pi_r\}\in\MCh(W,c)$ and an element $u \in W_c^+$ are \dfn{concordant} if $u\pi_0\leqB u\pi_1\leqB\cdots\leqB u\pi_r$. Let $\Omega(W,c)$ denote the set of pairs $(u,\C)\in W_c^+\times \MCh(W,c)$ such that $\C$ and $u$ are concordant. 
\end{definition}
In this language, we have \[\SBDW_\yy(W,c)=\{\conv(u\C \yy):(u,\C)\in\Omega(W,c)\},\]
where $\conv(u\C\yy)=\conv(\{u\pi\yy:\pi\in\C\})\subseteq\Perm_\yy.$  We now establish some combinatorial properties of the SBDW triangulation that allow us to use the positive cluster complex to interpret $W_c^+$ and $\MCh(W,c)$ in a common geometric setting.

\subsubsection{Commutation classes}
We say two chains $\bpi,\bpi'\in\MCh(W,c)$ are \dfn{commutation equivalent} if their corresponding reduced $T$-words $\rw(\bpi)$ and $\rw(\bpi')$ are related by commutation moves (equivalently, these words use the same set of reflections).  Write $\Class_c$ for the set of commutation classes of maximal chains in $\MCh(W,c)$, and let $\CL_{\bpi} \in \Class_c$ denote the commutation equivalence class of $\bpi$. For $u\in W_c^+$, let $\Class_c(u)$ denote the set of commutation equivalence classes of chains in $\MCh(W,c)$ that are concordant with $u$. We say $u$ and $\CL$ are concordant if $\CL \in \Class_c(u)$.

There is a partial order $\preceq_c$ on $T$ coming from the heap of the $c$-sorting word for the long element $w_\circ$ (see \cref{subsec:Coxeter_sorting} for the definition). We say a chain $\bpi\in\MCh(W,c)$ with $\rw(\bpi)=t_1\cdots t_r$ is \dfn{$\preceq_{c}$-increasing} (respectively, \dfn{$\preceq_c$-decreasing}) if there do not exist $1\leq i<j\leq r$ such that $t_j\preceq_c t_i$ (respectively, $t_i\preceq_c t_j$). We say the class $\CL_{\bpi}$ is \dfn{$\preceq_c$-increasing} (respectively, \dfn{$\preceq_c$-decreasing}) if $\bpi$ is (this is well defined). It is known that there is a unique $\preceq_c$-increasing commutation class, which we denote by~$\mathcal I_c$. When $W$ is irreducible, it is well known that the number of $\preceq_c$-decreasing commutation classes is the \defn{positive $W$-Catalan number} 
\begin{equation}\label{eq:Cat}
\Cat^+(W):=\prod_{i=1}^r \frac{h-1+e_i}{d_i},
\end{equation} 
where $h$ is the Coxeter number (the order of $c$), $d_1,\ldots,d_r$ are the degrees, and $e_1,\ldots,e_r$ are the exponents (with $e_i=d_i-1$). In particular, the number of such classes is independent of $c$. 

\begin{theorem}\label{thm:unique_decreasing} 
For each $u\in W_c^+$, the set $\Class_c(u)$ contains the $\preceq_c$-increasing class $\mathcal I_c$ and contains a unique $\preceq_c$-decreasing class, denoted $\DD_c(u)$. 
\end{theorem} 

\subsubsection{Reflection isomorphism} 
Define two Bruhat intervals $[u_1,v_1]_{\mathrm B}$ and $[u_2,v_2]_{\mathrm B}$ to be \dfn{reflection isomorphic} if there is a poset isomorphism $\varphi\colon[u_1,v_1]_{\mathrm B}\to[u_2,v_2]_{\mathrm B}$ such that for every cover relation $x\lessdot_{\mathrm{B}}y$ in $[u_1,v_1]_{\mathrm B}$, we have $x^{-1}y=\varphi(x)^{-1}\varphi(y)$.
If $u_1,u_2\in W_c^+$, then $[u_1,u_1c]_{\mathrm B}$ and $[u_2,u_2c]_{\mathrm B}$ are reflection isomorphic if and only if $\Class_c(u_1)=\Class_c(u_2)$.  The following theorem uses the \emph{$c^{-1}$-Cambrian congruence}, a lattice congruence of the right weak order of $W$ introduced by Reading \cite{ReadingCambrian,reading2007sortable}.

\begin{theorem}\label{thm:Cambrian} 
Let $u_1,u_2\in W_c^+$. The following are equivalent: 
\begin{itemize} 
\item $\DD_c(u_1)=\DD_c(u_2)$; 
\item $\Class_c(u_1)=\Class_c(u_2)$;
\item $[u_1,u_1c]_\hB$ and $[u_2,u_2c]_\hB$ are reflection isomorphic; 
\item $u_1^{-1}$ and $u_2^{-1}$ belong to the same $c^{-1}$-Cambrian class. 
\end{itemize}
\end{theorem} 

Despite the fact that the number of simplices in the SBDW triangulation and the number of elements in $W_c^+$ both depend on the choice of the Coxeter element $c$, we have the following corollary. 

\begin{corollary}\label{cor:reflection-iso-classes-pos-catalan}
When $W$ is irreducible, the number of reflection isomorphism classes of Bruhat intervals of the form $[u,uc]_\hB$ with $u\in W_c^+$ is the positive $W$-Catalan number $\Cat^+(W)$. In particular, the number of these classes is independent of $c$. 
\end{corollary}

\subsubsection{The positive cluster complex}
The positive $c^{-1}$-cluster complex $\mathrm{Clus}^+(W,c^{-1})$ is a certain flag simplicial complex on the vertex set $[N]$ (see \cref{subsubsec:positive_preliminaries}).  Let $\MCl(W,c^{-1})$ be the set of maximal simplices in $\mathrm{Clus}^+(W,c^{-1})$; these maximal simplices are naturally in bijection with the $\preceq_c$-decreasing commutation classes of chains in $\MCh(W,c)$. In particular, $|\MCl(W,c^{-1})|=\Cat^+(W)$. For $A\in\MCl(W,c^{-1})$, write $W^+_c(A)$ for the collection of elements $u\in W_c^+$ such that $\DD_c(u)$ is the $\preceq_c$-decreasing class corresponding to $A$ (so that $W^+_c(A)$ is the $c^{-1}$-Cambrian congruence class containing $u^{-1}$).

\subsubsection{The common geometric setting}\label{subsubsec:common} 

Each $t\in T$ acts on $V$ by reflecting through some hyperplane $H_t\in\mathcal H$; let $H_t^-$ be the closed half-space determined by $H_t$ that contains the region $\BB$. Let $\beta_t^\vee\in V^*$ be the coroot of $W$ that is normal to $H_t$ and lies in $H_t^-$, and write $(\Phi^+)^\vee=\{\beta_t^\vee:t\in T\}$. Let $T_\LL(c)=\{t\in T:\ell_S(tc)<\ell_S(c)\}$ denote the set of left inversions of $c$. Our constructions will live in the simplicial cone given by any of the following equivalent definitions (see~\Cref{prop:region}):
\[\Delta_c^+:= \bigcap_{t \in T_\LL(c)} H_{t}^- =\bigcup_{u\in W_c^+}\overline{u^{-1}\BB} = (1-c^{-1})^{-1}\spange (\Phi^+)^\vee \subseteq V.\]
(Here, $\spange$ denotes nonnegative span.) For each $t \in T$, there is a distinguished point ${\dd_t=(1-c^{-1})^{-1}\beta_t^\vee}$ inside $\Delta_c^+$, and $\Delta_c^+=\spange \{\dd_s:s\in S\}$. We will always draw $\Delta_c^+$ as a simplex by deleting the origin and passing to the positive projectivization (identifying points that have the same positive span). We will sometimes refer to the points $\dd_t$ for $t\in T$ as \emph{vertices} since they will naturally correspond to vertices of the positive $c^{-1}$-cluster complex.

\begin{example}\label{ex:a3-0}
Let $W=A_3=\SSS_4$ and $c=s_1s_2s_3=(1\,2\,3\,4)$, where $s_i=(i\,\,i+1)$. The set $\Delta_c^+$ is illustrated in~\Cref{fig:a3arrangement}. We have taken $V$ to be the quotient space $\mathbb R^4/\mathrm{span}\{(1,1,1,1)\}$. We have then passed to the positive projectivization of $V$, which can be identified with a unit $2$-sphere. In this model, the hyperplanes in the braid arrangement $\mathcal H$ are represented as great circles; we have drawn a stereographic projection of these circles onto the plane. Some points are indicated by representatives $(x_1,x_2,x_3,x_4)\in\mathbb R^4$ (with commas and parentheses omitted). Because $T_\LL(c)=\{(1\,2),(1\,3),(1\,4)\}$, we have $\Delta_c^+=H_{12}^- \cap H_{13}^- \cap H_{14}^-$ (we write $H_{ij}$ instead of $H_{(i\,j)}$).   
\end{example}

The following table summarizes our geometric interpretations of $W_c^+$, $\MCl(W,c^{-1})$, $\MCh(W,c)$, and $\Class_c$ by describing maps (by abuse of notation, each denoted $\Delta$) from each object to a simplicial cone inside of $\Delta_c^+$.  
\[
\begin{array}{cccc}
\toprule
\text{Set} & \text{Element} & \text{Description} & \text{Image under }\Delta \\ 
\midrule
W_c^+ & u & \ell_S(uc)=\ell_S(u)+r&  \Delta(u):=\overline{u^{-1}\BB} \\ 
\addlinespace[1em]
\MCl(W,c^{-1}) & A & \begin{array}{c}\text{positive $c^{-1}$-cluster, or}\\ \text{$\preceq_c$-decreasing $T$-word for $c$}\end{array}&  \Delta(A):=\bigcup\limits_{u\in W_c^+(A)}\Delta(u) \\ 
\addlinespace[2em]
\begin{array}{c}\MCh(W,c) \\ \Red_T(c) \\ \Class_c\end{array} & \begin{array}{c}\bpi \\ \rw(\bpi)=t_1\cdots t_r \\ \CL_{\bpi} \end{array} & \begin{array}{c}\text{maximal chain in $\NC(W,c)$}\\ \text{reduced $T$-word for $c$}\\ \text{commutation class of $\bpi$}\end{array} & \Delta(\bpi):=\spange\{\dd_{t_1},\ldots,\dd_{t_r}\} \\ 
\bottomrule
\end{array}
\] 

We note the following fundamental facts about these maps $\Delta$:  
\begin{itemize}
\item For $u\in W_c^+$, $A\in\MCl(W,c^{-1})$, and $\bpi\in\MCh(W,c)$, each of the cones $\Delta(u)$, $\Delta(A)$, and $\Delta(\bpi)$ is $r$-dimensional; it is the nonnegative span of some $r$-subset (determined from $u$, $A$, or $\bpi$) of the set $\{\dd_t:t\in T\}$.

\item The images under $\Delta$ of both $W_c^+$ and $\MCl(W,c^{-1})$ triangulate $\Delta_c^+$, and the image of $W_c^+$ refines the image of $\MCl(W,c^{-1})$ (specifically, $\Delta(u)\subseteq \Delta(A)$ if $u^{-1}$ belongs to the $c^{-1}$-Cambrian class $W_c^+(A)$). 

\item For each $\bpi \in \MCh(W,c)$, the cone $\Delta(\bpi)$ is the union of some of the cones in the image of $\MCl(W,c^{-1})$ under $\Delta$. 
\end{itemize} 

\begin{example}\label{ex:a3-1}
Continuing~\Cref{ex:a3-0}, we have redrawn $\Delta_c^+$ on the left side of~\Cref{fig:a3arrangement2}, showing the refinement of $\MCl(W,c^{-1})$ by $W_c^+$. Each simplicial cone is drawn as a triangle. Each distinguished point $\dd_{(i\,j)}$ (for $(i\,j)\in T$) is represented by a circle filled with $ij$.  Let $u=s_3s_2 \in W_c^+$. On the right side, we have marked $\Delta(u)$ with a white star (this triangle is labeled by $u^{-1}=s_2s_3$ on the left). Let $A \in \MCl(W,c^{-1})$ be the unique positive $c^{-1}$-cluster such that $u\in W_c^+(A)$. We have shaded the triangle $\Delta(A)$ in blue. The union of the blue and purple triangles is $\Delta(\bpi)=\Delta(\bpi')$, where $\rw(\bpi)=(1\,3)(1\,2)(3\,4)$ and $\rw(\bpi')=(1\,3)(3\,4)(1\,2)$. 

\begin{figure}[htbp]
\raisebox{-0.5\height}{\includegraphics[height=7.282cm]{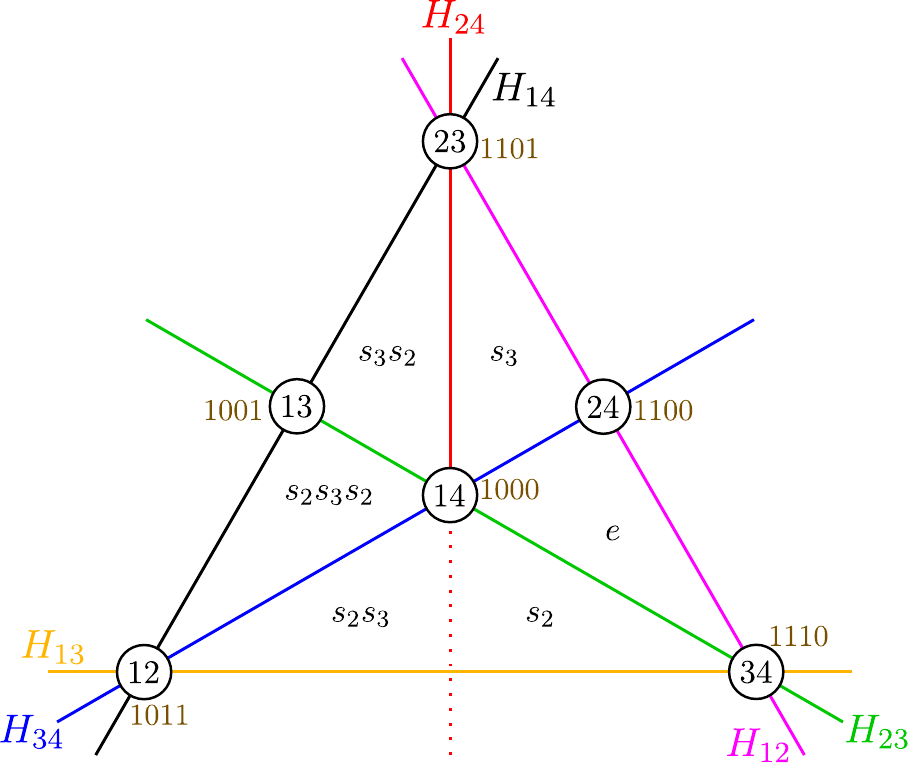}} \qquad\raisebox{-0.5\height}{\includegraphics[height=2.58cm]{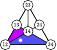}}
\caption{Fix $W=\SSS_4$ and $c=s_1s_2s_3$.  On the left is $\Delta_c^+$, with regions indexed by inverses of elements of $W_c^+$ and with edges from $\mathrm{Clus}^+(W,c^{-1})$ (note that half of the vertical red edge corresponding to a piece of $H_{24}$ is dotted, indicating that it is not an edge of the positive $c^{-1}$-cluster complex). On the right is a simplex representing $(u,\bpi) \in \Omega(W,c)$, with $u=s_3s_2$ and $\bpi$ chosen so that $\rw(\bpi)=(1\,3)(1\,2)(3\,4)$.  The triangle representing $\Delta(u)$ is marked with a white star.  The blue triangle is the image $\Delta(A)$ of the positive cluster $A \in \mathrm{MClus}^+(W,c^{-1})$ whose corresponding $c^{-1}$-Cambrian congruence class $W^+_c(A)$ contains $u^{-1}=s_2s_3$ (note that $s_2$ is also contained in this congruence class).  The triangle representing the cone $\Delta(\bpi)=\Delta(\CL_{\bpi})$ is the union of the blue and purple triangles. 
}
\label{fig:a3arrangement2}
\end{figure} 

\end{example}
\begin{remark}
While $W_c^+$ and $\MCl(W,c^{-1})$ (and even the noncrossing partition lattice $\NC(W,c)$~\cite{BradyWatt2008}) are often interpreted in $\Delta_c^+$, we have not seen $\MCh(W,c)$ encoded in this way before---for example, on the right side of~\Cref{fig:a3arrangement2}, the reader will enjoy finding all 12 collections of small triangles whose union is again a triangle, and then ordering the reflections labeling the vertices of these larger triangles to obtain the 12 commutation classes of reduced $T$-words for $c=(1\,2\,3\,4)$.
\end{remark}

\subsubsection{Concordancy}
We use the preceding geometric interpretations to characterize concordancy.

\begin{theorem}\label{thm:spans_new}
An element $u\in W_c^+$ and a chain $\bpi \in \MCh(W,c)$ are concordant if and only if $\Delta(u) \subseteq \Delta(\bpi)$.
\end{theorem}
\cref{thm:spans_new} characterizes the maximal chains that appear in an interval $[u, uc]_\hB$ (for $u\in W_c^+$), so it specifies the subposet of the noncrossing partition lattice $\NC(W,c)$ isomorphic to this interval.

\begin{example}\label{ex:visualization}
Let $W=A_3=\SSS_4$, and let $c=s_1s_2s_3$ (where $s_i=(i\,\,i+1)$). We have ${W_c^+ = \{e,s_2,s_3,s_2s_3,s_3s_2,s_2s_3s_2\}}$. \cref{fig:TrianglesLinear} is a visualization of the adjacency structure of our triangulation.  It shows:

\begin{itemize}
\item a \thuge{} triangle, representing the cone $\Delta_c^+$,
subdivided into
\item six \tlarge{} triangles, representing the cones $\Delta(u)$ for elements $u \in W_c^+$ (the exterior tip of $\Delta(u)$ is labeled by $u^{-1}$).  Each \tlarge{} triangle contains
\item some \tmedium{} triangles, each representing another copy of $\Delta_c^+$, 
each subdivided into 
\item five \tsmall{} triangles, representing the cones $\Delta(A)$ for ${A \in \MCl(W,c^{-1})}$.
\end{itemize}

Some \tsmall{} triangles in each \tmedium{} triangle are colored.  A \tmedium{} triangle and the coloring of its \tsmall{} triangles together represent a concordant pair $(u,\mathcal C)$; equivalently, this pair represents the collection $\{u\bpi\yy:\bpi\in\CL\}$ of commutation equivalent SBDW simplices with the same base point.  The element $u$ in the pair $(u,\mathcal C)$ is recorded in two ways: the \tmedium{} triangle is contained in the \tlarge{} triangle indexed by $u^{-1}$, and $\Delta(u)$ is indicated in the \tmedium{} triangle by a white star.

Two \tmedium{} triangles $M_1$ and $M_2$ corresponding to concordant pairs $(u_1,\CL_1)$ and $(u_2,\CL_2)$ are connected by an edge if and only if there is an SBDW simplex in the collection $\{u_1\bpi\yy:\bpi\in\CL_1\}$ that shares a facet with an SBDW simplex in the collection $\{u_2\bpi\yy:\bpi\in\CL_2\}$. 

\Cref{fig:TrianglesLinear} also illustrates~\Cref{thm:Cambrian} as follows. Let $w=s_2$ and $u=s_3s_2$. The  \tlarge{} triangles labeled by $w^{-1}$ and $u^{-1}$ contain the same \tmedium{} triangles (ignoring the white star).  This indicates that the intervals $[w,wc]_\hB$ and $[u, uc]_\hB$ are reflection isomorphic, which corresponds to the fact that $u^{-1}$ and $w^{-1}$ lie in the same $c^{-1}$-Cambrian class. This, in turn, is encoded by the fact that the white stars marking $\Delta(w)$ and $\Delta(u)$ both lie in the blue \tsmall{} triangle. 
\end{example} 

\begin{figure}[htbp]
\begin{center}\includegraphics[width=\linewidth]{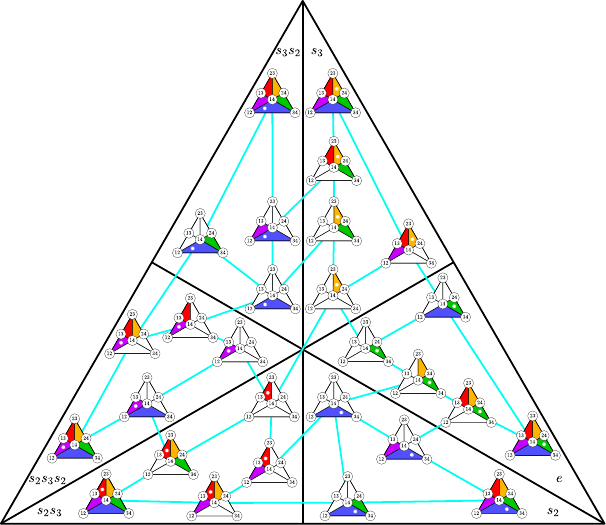}
  \end{center}
\caption{A representation of the pairs $(u,\mathcal C)$ with $\mathcal C\in\Class_c(u)$, where $W=A_3=\SSS_4$ and $c=s_1s_2s_3$. Edges represent adjacencies between corresponding collections of simplices in the SBDW triangulation.}\label{fig:TrianglesLinear}
\end{figure}

\subsubsection{Additional structure}\label{sec:additional}
Each SBDW simplex $u\bpi\yy$ is labeled by the chain $\bpi \in \MCh(W,c)$, but $\bpi$ may occur as the label of multiple simplices (for example, those chains whose corresponding reduced $T$-words use only simple reflections label a simplex for \emph{every} $u \in W_c^+$).  \Cref{thm:all_chains_appear} proves that the set $\{ u \in W_c^+ : (u,\bpi) \in  \Omega(W,c) \}$ of elements concordant with a particular $\bpi$ is a nonempty interval in the left weak order, so this set has a lowest element $\omega_\downarrow(\bpi)$.

With this in mind, it is natural to restrict our triangulation to the simplices corresponding to pairs of the form $(\omega_\downarrow(\bpi),\bpi)$. For this subset of simplices, each $\bpi$ is the label of exactly one simplex. \Cref{thm:trosable_thing} shows that the vertices of the permutahedron used by this subset of simplices correspond exactly to the $c^{-1}$-sortable elements of $W$.

\begin{remark}
Since the vertices of the simplices $\{(\omega_\downarrow(\bpi),\bpi): \bpi \in \MCh(W,c)\}$ coincide with the indexing set of the vertices of a $c^{-1}$-associahedron, and since the indexing set of our simplices can be identified with the indexing set of Reading's pulling triangulation of the associahedron~\cite[Theorem 8.12]{reading2011noncrossing}, one might expect that the collection $\{\conv(\omega_\downarrow(\bpi) \bpi \yy) : \bpi \in \MCh(W,c)\}$ is a new triangulation of a $W$-associahedron.  This is not the case---in fact, the union of these simplices is generally not even convex.
\end{remark}

\subsection{Past and present work} \label{sec:past_and_present} To put our work in context, we briefly survey some related previous and ongoing work (see also~\Cref{rem:new}).
\subsubsection{Bruhat interval polytopes and the permutahedral variety}

For $u,v\in W$ with $u\leqB v$, one can consider the interval $[u,v]_{\mathrm{B}}$. The corresponding \dfn{Bruhat interval polytope} \[\BIP_{u,v}=\mathrm{conv}([u,v]_{\mathrm{B}}\yy)\] is obtained by taking the convex hull of the vertices of $\Perm_\yy$ corresponding to the elements of $[u,v]_{\mathrm{B}}$. Bruhat interval polytopes were introduced by Kodama and Williams \cite{KodamaWilliams} and studied further by Tsukerman and Williams in relation with total positivity and Coxeter matroids \cite{TsukermanWilliams}. There has been interest in regular subdivisions of the $W$-permutahedron into Bruhat interval polytopes. In particular, in type $A$, such subdivisions are in one-to-one correspondence with cones in the tropical positive flag variety, which is equal to the positive flag Dressian \cite{Boretsky,JLLO} (see also \cite{BEW}).

Consider the collection \[\BIPs_c=\{Q_{u,uc}:u\in W_c^+\}\] of Bruhat interval polytopes. Knutson, Sanchez, and Sherman-Bennett \cite{KSSB} show that $\BIPs_c$ is a subdivision of the $W$-permutahedron when $W$ is a Weyl group. Their techniques are algebro-geometric: they show that the permutahedral variety in $G/B$ can be degenerated to the union of Richardson varieties $\bigcup_{u \in W_c^+} R_{u,uc}$. The moment polytope of the permutahedral variety is the $W$-permutahedron, and the moment polytope of $R_{u,uc}$ is the Bruhat interval polytope $\BIP_{u,uc}$. The subdivision $\BIPs_c$ is essentially obtained by applying the moment map to the degeneration result, and it can be viewed as a combinatorial shadow of the degeneration.

Our triangulation $\SBDW_\yy(W,c)$ is a refinement of the subdivision $\BIPs_c$ into simplices. The approach of Knutson--Sanchez--Sherman-Bennett to showing $\BIPs_c$ is a subdivision is quite different from our approach to showing $\SBDW_\yy(W,c)$ is a triangulation, and their results do not seem to imply or follow from ours. 

\begin{remark}
When $W$ is of type $A$ or $B$, Knutson, Sanchez, and Sherman-Bennett show that $\BIPs_c$ is in fact a \emph{regular} subdivision of $\Perm_\yy$. In this case, we could tweak their height function to prove \cref{thm:triangulation} without using tools from representation theory. In fact, in this case, we could prove that $\SBDW_\yy(W,c)$ is a regular triangulation for \emph{every} $\yy\in\BB$. However, this approach does not work in arbitrary types since it is not known in general if $\BIPs_c$ is regular. On the other hand, one could deduce from our \cref{thm:triangulation} that the Bruhat interval polytopes in $\BIPs_c$ do not intersect in their interiors and that $\bigcup\BIPs_c=\Perm_\yy$. However, this does not imply that $\BIPs_c$ is a subdivision since it does not imply that the polytopes in $\BIPs_c$ intersect along common faces. 
\end{remark} 

As the triangulation $\SBDW_\yy(W,c)$ refines the subdivision $\BIPs_c$, restricting $\SBDW_\yy(W,c)$ to any $\BIP_{u,uc}$ yields a triangulation of that Bruhat interval polytope. \cref{thm:Cambrian,cor:reflection-iso-classes-pos-catalan,thm:spans_new} provide information about the structure of these triangulations. For example, \cref{thm:Cambrian} characterizes when $\BIP_{u_1,u_1c}$ and $\BIP_{u_2,u_2c}$ have identical triangulations, along with when their ``translations to the identity" $u_1^{-1}\BIP_{u_1,u_1c}$ and $u_2^{-1} \BIP_{u_2,u_2c}$ coincide. \cref{thm:spans_new} characterizes which simplices appear in the triangulation of $\BIP_{u,uc}$.

\subsubsection{The quasisymmetric flag variety}

Recent work of Nadeau--Spink--Tewari \cite{NTS}, and those authors with Bergeron and Gagnon \cite{BGNTS}, features two (reducible) toric subvarieties of the type $A$ flag variety whose cohomology rings are related to quasisymmetric polynomials. The combinatorics of these toric subvarieties is related to what we present here, as we now explain.

Fix the Coxeter element $c= s_{n-1} s_{n-2}\dots s_2 s_1 = (n, n-1, \dots, 2, 1)$. The \dfn{$\Omega$-flag variety} defined in \cite{NTS} is 
\[\bigcup_{u \in W_c^+} u^{-1} R_{u, uc},\]
a union of translated toric Richardsons. The authors of \cite{NTS} noted that many of these translated Richardsons coincide and that there are exactly $\Cat^+(W)$ distinct translated Richardsons. Applying the moment map, this says there are exactly $\Cat^+(W)$ distinct translated Bruhat interval polytopes $u^{-1}\BIP_{u, uc}$. One can use \cref{cor:reflection-iso-classes-pos-catalan} to give an alternate proof of this fact and to generalize it to other choices of $c$.

The \dfn{quasisymmetric flag variety} $\mathrm{QFl}_n$ of \cite{BGNTS} is defined similarly to the $\Omega$-flag variety, but without ``repeats". The authors of \cite{BGNTS} show that the combinatorics of $\mathrm{QFl}_n$ is governed by noncrossing partitions. One can view the SBDW triangulation as an explanation for this: the moment map sends $\mathrm{QFl}_n$ to the union
\[\bigcup u^{-1} \BIP_{u, uc}\]
where $u$ ranges over representatives for each $c^{-1}$-Cambrian class.
Triangulating each (translated) $\BIP_{u, uc}$ according to the SBDW triangulation, this union becomes the order complex of $\NC(W,c)$.

Bergeron--Gagnon--Nadeau--Spink--Tewari \cite{BGNTS2} have also obtained independent proofs of \cref{thm:Cambrian,cor:reflection-iso-classes-pos-catalan} in the process of generalizing \cite{BGNTS} to other $W$ and $c$.

\subsubsection{Weighted chains in the noncrossing partition lattice}
Josuat-Verg\`es and his coauthors Biane and Douvropoulos have considered related notions coming from combining absolute order with Bruhat order~\cite{BianeJosuatVerges,DouvropoulosJosuatVerges1,JosuatVerges}.  For example, they have shown the elegant generating function identity~\cite[Proposition 3.6]{JosuatVerges} \[\sum_{\bpi \in \MCh(W,c)} q^{\mathrm{wt}(\bpi)} =\frac{r!}{|W|} \prod_{i=1}^r \left(d_i +q(h-d_i)\right),\] where $\mathrm{wt}(\bpi)$ denotes the number of times a maximal chain $\bpi$ of noncrossing partitions decreases in Bruhat order.  As another example,~\cite[Proposition 4.3]{BianeJosuatVerges} gives an elegant geometric description for when one has a Bruhat decrease in a chain of noncrossing partitions.

\subsection{Outline} 
\cref{sec:preliminaries} recalls and establishes necessary background and results about Coxeter groups, braid groups, representation theory, and polyhedral geometry. \cref{sec:subdivisions} is devoted to proving general results about polytopal subdivisions. In~\cref{sec:triangulating}, we construct a height function using totally stable linear stability, proving that $\SBDW_\yy(W,c)$ is a triangulation of $\Perm_\yy(W)$. In~\cref{sec:combinatorics}, we prove the more refined results about the structure of SBDW triangulations stated in \cref{subsec:properties-of-triang}. \cref{sec:presentations} provides details about how this triangulation allows us to derive the Artin presentation from the Bessis dual presentation of the braid group. In \cref{sec:future}, we collect several ideas for future work.  These include conjectural presentations for the pure braid group, the problem of constructing an explicit homotopy between the Salvetti and Bessis--Brady--Watt complexes, and several new partial orders that appear to have rich structure. For example, orienting the $1$-skeleton of the SBDW triangulation defines a partial order on $W$ that we call the \emph{$c$-noncrossing Bruhat order}, while restricting the Cayley graph of the \emph{dual braid monoid} to the image of the \emph{positive Artin monoid} produces another partial order on $W$ that we call the \emph{$c$-noncrossing weak order}. We conjecture that these posets are lattices and are equal to each other.

\section{Preliminaries}\label{sec:preliminaries}  

\subsection{Coxeter combinatorics} 

We assume basic familiarity with the theory of Coxeter groups, which can be found in \cite{BjornerBrenti, Humphreys}. 

\subsubsection{The reflection representation} Let $W$ be a finite Coxeter group of rank $r$ acting (faithfully) in its reflection representation on an $r$-dimensional Euclidean space $V$. Let $\HH$ denote the reflection arrangement of $W$. Let $\Vreg=V\setminus\bigcup\HH$ be the complement of the reflection arrangement in $V$. The connected components of $\Vreg$ are called \dfn{regions}. Fixing a region $\BB$ to call the \dfn{base region} induces a bijection $w\mapsto w\BB$ from $W$ to the set of regions. 

For $H\in\HH$, let $t_H\in W$ denote the element that acts on $V$ via the orthogonal reflection through $H$. The elements of the set $T=\{t_H:H\in\HH\}$ are the \dfn{reflections} of $W$. A reflection $t_H$ is \dfn{simple} if the hyperplane $H$ is a wall of $\BB$. Let $S\subseteq T$ be the set of simple reflections of $W$. We say $W$ is \dfn{simply-laced} if for all $s,s'\in S$, the  order of $ss'$ is at most $3$. Equivalently, $W$ is simply-laced if and only if its irreducible components are each of type $A$, $D$, or $E$. 

\subsubsection{Roots and coroots} Let $V^*$ be the dual space of $V$. Then $V^*$ is a Euclidean space with an inner product $(\cdot,\cdot)$. We write $\langle \cdot,\cdot\rangle$ for the canonical pairing between $V^*$ and $V$. In other words, $\langle\gamma,\xx\rangle=\gamma(\xx)$ for all $\gamma\in V^*$ and $\xx\in V$. Let $\Phi\subseteq V^*$ denote the root system of $W$, and let $\Phi^+$ and $\Phi^-$ denote the sets of positive roots and negative roots, respectively. Thus, $\Phi^-=-\Phi^+$, and $\Phi=\Phi^+\sqcup\Phi^-$.  

Each root $\beta\in\Phi$ has an associated \dfn{coroot} $\beta^\vee\in V$, which is uniquely determined by the property that $\langle\gamma,\beta^\vee\rangle=2\frac{(\beta,\gamma)}{(\beta,\beta)}$ for all $\gamma\in V^*$. Let $(\Phi^+)^\vee=\{\beta^\vee:\beta\in \Phi^+\}$. There is also a natural right action of $W$ on $V^*$, which is defined so that $\langle \gamma w,\xx\rangle=\langle \gamma,w\xx\rangle$ for all $w\in W$, $\xx\in V$, and $\gamma\in V^*$. 

For each $t\in T$, there is a unique $\beta_t\in\Phi^+$ such that \[t\xx=\xx-\langle \beta_t,\xx\rangle\beta_t^\vee\] for all $\xx\in V$. The map $t\mapsto \beta_t$ is a bijection from $T$ to $\Phi^+$. The (right) action of $t$ on $\gamma$ is given by $\gamma t=\gamma-\langle\gamma,\beta_t^\vee\rangle\beta_t$ for all $\gamma\in V^*$. A positive root $\beta_t$ is \dfn{simple} if $t$ is a simple reflection. The set $\{\beta_s:s\in S\}$ of simple roots is a basis of $V^*$, and the set $\{\beta_s^\vee:s\in S\}$ of simple coroots is a basis of~$V$.

\subsubsection{Parabolic subgroups} For $J\subseteq S$, we let $W_J$ denote the subgroup of $W$ generated by $J$; we call $W_J$ a \dfn{standard parabolic subgroup} of $W$. For $s\in S$, we will find it convenient to write $\langle s\rangle$ for the set $S\setminus\{s\}$ so that $W_{\langle s\rangle}$ is the standard parabolic subgroup generated by all simple reflections other than $s$. A \dfn{parabolic subgroup} of $W$ is a subgroup generated by a subset of $T$. A parabolic subgroup has rank $2$ if it is generated by two reflections. 

\subsubsection{Reduced words} A \dfn{reduced $S$-word} for an element $w\in W$ is a word over the alphabet $S$ that represents $w$ and has the minimum possible length among all such words. The \dfn{Coxeter length} of $w$, denoted $\ell_S(w)$, is the length of a reduced $S$-word for $w$. A \dfn{reduced $T$-word} for $w$ is a word over the alphabet $T$ that represents $w$ and has the minimum possible length among all such words. The \dfn{reflection length} of $w$, denoted $\ell_T(w)$, is the length of a reduced $T$-word for $w$. The \dfn{long element} of $W$, denoted $\wo$, is the unique element of $W$ with Coxeter length $|T|$. A \dfn{subword} of a word $\mathsf{w}$ is a word obtained from $\mathsf{w}$ by deleting some letters. 

\subsubsection{Inversions and weak order} A \dfn{right inversion} (respectively, \dfn{left inversion}) of an element $w\in W$ is a reflection $t\in T$ such that $\ell_S(wt)<\ell_S(w)$ (respectively, $\ell_S(tw)<\ell_S(w)$). Let $T_\R(w)$ (respectively, $T_\LL(w)$) denote the set of right inversions (respectively, left inversions) of $w$. For $t\in T$, $w\in W$, and $\yy\in\BB$, we have $\langle\beta_t,w\yy\rangle<0$ if and only if $t$ is a left inversion of $w$. A \dfn{right descent} (respectively, \dfn{left descent}) of $w$ is a right inversion (respectively, left inversion) that is a simple reflection. A simple reflection $s$ is a right (respectively left) descent of $w$ if and only if there is a reduced $S$-word for $w$ that ends (respectively, begins) with $s$.   We define the \dfn{inversion sequence} of a reduced $S$-word $\mathsf{w}=s_{1}\cdots s_{k}$ to be the tuple 
$\inv(\mathsf{w})=(t_1,t_2,\ldots, t_k)$, where $t_i=s_1\cdots s_{i-1}s_is_{i-1}\cdots s_1$. If $\mathsf{w}$ is a reduced $S$-word for $w$, then $\inv(\mathsf{w})$ lists each left inversion of $w$ exactly once. 

The \dfn{left weak order} is the partial order $\leq_\LL$ on $W$ defined so that $u\leq_\LL v$ if and only if ${T_\R(u)\subseteq T_\R(v)}$ if and only if $v$ has a reduced $S$-word with a word for $u$ as a suffix. The \dfn{right weak order} is the partial order $\leq_{\R}$ on $W$ defined so that $u\leq_\R v$ if and only if $T_\LL(u)\subseteq T_\LL(v)$ if and only if $v$ has a reduced $S$-word with a word for $u$ as a prefix. We have $u\leq_\R v$ if and only if $u^{-1}\leq_\LL v^{-1}$. 

\subsection{Coxeter--Catalan combinatorics} 
Fix a standard Coxeter element $c$ of $W$. Note that ${\ell_S(c)=\ell_T(c)=r}$. 

\subsubsection{Reduced reflection words}
A lemma due to Carter~\cite[Lemma~3]{carter1972conjugacy} states that if $t_1\cdots t_r$ is a reduced $T$-word for $c$, then the roots $\beta_{t_1},\ldots,\beta_{t_r}$ are linearly independent.  In fact, one can strengthen this statement to a full characterization of the sets of reflections that can appear in a reduced $T$-word for $c$.  Say two reflections $t,t'$ are \defn{$c$-noncrossing} if they appear simultaneously in some reduced word for $c$---we can move this pair to be a prefix of the word by applying Hurwitz moves, so $t,t'$ are $c$-noncrossing if and only if $tt' \leq_T c$ or $t't \leq_T c$. The following lemma combines \cite[Lemma~3]{carter1972conjugacy} with \cite[Proposition~4.1.2]{stump2025cataland} (see also \cite[Lemma~4.8]{BradyWatt2008}).

\begin{lemma}\label{lem:linearly_independent} 
A set of reflections can be ordered to give a reduced $T$-word for $c$ if and only if they are pairwise $c$-noncrossing and the corresponding positive roots are linearly independent.  This ordering is unique, up to commutations. 
\end{lemma}

\subsubsection{Coxeter-sorting words}\label{subsec:Coxeter_sorting}
Now fix a reduced $S$-word $\mathsf c$ for $c$, and consider the infinite word $\mathsf c^\infty=\mathsf{ccc}\cdots$ obtained by concatenating $\mathsf{c}$ with itself infinitely many times. Following Reading~\cite{reading2007sortable,reading2007clusters}, we define the \dfn{c-sorting word} of $w$, denoted $\w(c)$, to be the lexicographically first subword of ${\sf c}^\infty$ that is a reduced $S$-word for $w$. While this word depends on the choice of the reduced $S$-word ${\sf c}$, it is well defined up to commutation equivalence. We define a partial order $\preceq_c$ on $T$ by declaring that $t\preceq_c t'$ if $t$ appears to the left of $t'$ in every word that is commutation equivalent to $\inv(\w_\circ(c))$ (the inversion sequence of the $c$-sorting word of $w_\circ$). 

It is known (see \cite[Lemma 2.9.1]{stump2025cataland} and \cite[Lemmas 3.7 \& 3.8]{reading2011sortable}) that $t\preceq_c t'$ if and only if $t'\preceq_{c^{-1}}t$. This implies the following lemma. 

\begin{lemma}\label{lem:preceq_c}
Let $c$ be a standard Coxeter element of $W$, and let $s$ be a right descent of $c$. Then $s$ is a maximal element of the poset $(T,\preceq_c)$.
\end{lemma} 

\subsubsection{Sorting orders}
Let $\sw_\circ(c^{-1})=s_1\cdots s_N$ be the $c^{-1}$-sorting word for $w_\circ$, and let \begin{equation}\label{eq:woinv}\inv(\w_\circ(c^{-1})) = (a_1,a_2,\ldots,a_N).\end{equation}
Every reflection in $T$ appears exactly once in $\inv(\w_\circ(c^{-1}))$. 

While the ordering $\inv(\w_\circ(c^{-1}))$ on reflections is useful to index hyperplanes, we also need to be able to index vertices.  To this end, let $\{\lambda_s\}_{s \in S} \subset V$ be the dual basis to the basis of simple roots $\{\beta_s\}_{s \in S} \subset V^*$ (if $W$ is crystallographic, then $\lambda_{s}$ is the fundamental weight dual to the simple root $\beta_s$).  Then $\lambda_s \in \bigcap_{s'\in S\setminus\{s\}}H_{s'}$, so $s'(\lambda_s)=\lambda_s$ for every $s' \in S \setminus \{s\}$.  Let $\dd_{a_i}=(s_1\cdots s_{i-1})(\lambda_{s_i})$, and let $\dd(\w_\circ(c^{-1}))=(\dd_{a_1},\ldots,\dd_{a_N})$. 
Recall that $H_t^-$ denotes the half-space bounded by $H_t$ that contains $\BB$. We define \[\Delta_c^+:=\bigcap_{t\in T_\LL(c)}H_t^-\subseteq V.\]

We now give several characterizations of the simplicial cone $\Delta_c^+$.
\begin{lemma}\label{lem:deltabeta}
For every $i\in[N]$, we have $(1-c^{-1})\dd_{a_i}=\beta_{a_i}^\vee$. 
\end{lemma} 
\begin{proof}
First suppose $i=1$. Let $s=s_1$. Note that $s$ is a left descent of $c^{-1}$ and that it is the first entry in each of  $\w_\circ(c^{-1})$ and $\inv(\w_\circ(c^{-1}))$. Then $\dd_{a_1}=\lambda_{s}$ and $\beta_{a_1}=\beta_{s}$. Since $\delta_{s}$ is fixed by all simple reflections other than $s$, we find that \[(1-c^{-1})(\dd_{s})=\dd_s-c^{-1}(\dd_{s})=\dd_{s}-s(\dd_{s})=\langle \beta_{s},\dd_{s}\rangle \beta_{s}^\vee=\beta_{s}^\vee.\]

Now suppose $i\geq 2$ and proceed by induction on $i$. Consider the standard Coxeter element $c'=sc^{-1} s$. For $2\leq j\leq N$, let $a_j'=sa_js$. The word $\w_\circ(c')$ is commutation equivalent to the word  $(a_2',\ldots,a_N',s)$.  By induction, we have that $(1-c')\dd_{a_i'}=\beta_{a_i'}^\vee$, so \[(1-c^{-1})\dd_{a_i} = (1-sc's)\dd_{a_i} = \dd_{a_i}-sc's\dd_{a_i}=s\dd_{a_i'}-sc'\dd_{a_i'}=s(1-c')\dd_{a_i'}=s\beta_{a_i'}^\vee=\beta_{a_i}^\vee.\qedhere\] 
\end{proof}

\begin{proposition}\label{prop:region}
We have
\[\Delta_c^+ =\bigcup_{u\in W_c^+}\overline{u^{-1}\BB} = \spange \{\dd_t:t\in T\} = (1-c^{-1})^{-1}\spange (\Phi^+)^\vee = \spange \{\dd_s:s\in S\}.\] 
\end{proposition}
\begin{proof}
It is straightforward to show that $W_c^+=\{u\in W:T_\LL(u^{-1})\cap T_\LL(c)=\emptyset\}$. This implies the first equality. The second equality follows from standard Coxeter--Catalan theory~\cite[Section 6]{pilaud2015brick}, where we combine the description of the $c^{-1}$-cluster complex using the subword complex, the bijection between $c^{-1}$-clusters and $c^{-1}$-sortable elements, and the restriction of both sets to those elements of full support. The third and fourth equalities now follow from \cref{lem:deltabeta} because $\spange(\Phi^+)^\vee=\spange\{\beta_s^\vee:s\in S\}$.
\end{proof}

For $u\in W_c^+$, we define \[\Delta(u):=\overline{u^{-1}\BB}.\]

\begin{example}\label{ex:a3-1'}
Continuing~\Cref{ex:a3-0}, we describe the data of this section.  With $s_i$ acting on $\mathbb{R}^4$ by permuting the $i$th and $(i+1)$st coordinates, we have chosen $\lambda_{s_1}=(1,0,0,0)$, $\lambda_{s_2}=(1,1,0,0)$, and $\lambda_{s_3}=(1,1,1,0)$. The relevant data is summarized in the following table: \[\begin{array}{l|cccccc}\toprule
w_\circ(c^{-1}) & s_3 & s_2 & s_1 & s_3 & s_2 &s_3 \\
\inv(\w_\circ(c^{-1})) & (34) & (24) & (14) & (23) & (13) & (12) \\
\dd(\w_\circ(c^{-1})) & (1,1,1,0) & (1,1,0,0) & (1,0,0,0) & (1,1,0,1) & (1,0,0,1) & (1,0,1,1)\\ \bottomrule
\end{array}
\]
Note that $\{\dd_{(1\,2)},\dd_{(2\,3)},\dd_{(3\,4)}\}=\{(1,0,1,1),(1,1,0,1),(1,1,1,0)\}$, so that \[\Delta_c^+=H_{12}^- \cap H_{13}^- \cap H_{14}^- = \spange\{\dd_{(1\,2)},\dd_{(2\,3)},\dd_{(3\,4)}\}.\]  Finally, observe that $c^{-1}=(1\,4\,3\,2)$ rotates vectors to the left, so that
\begin{align*}
(1-c^{-1})\dd_{(1\,2)}&=(1,0,1,1)-(0,1,1,1)=(1,-1,0,0)=\beta_{(1\,2)}^\vee;\\
(1-c^{-1})\dd_{(2\,3)}&=(1,1,0,1)-(1,0,1,1)=(0,1,-1,0)=\beta_{(2\,3)}^\vee;\\
(1-c^{-1})\dd_{(3\,4)}&=(1,1,1,0)-(1,1,0,1)=(0,0,1,-1)=\beta_{(3\,4)}^\vee.
\end{align*}
\end{example}

\subsubsection{The positive cluster complex}\label{subsubsec:positive_preliminaries} 
We now recall the positive $c$-cluster complex. Consider a chain $\bpi\in\MCh(W,c)$, and let $\rw(\bpi)=t_1\cdots t_r$. We say $\bpi$ is \dfn{$\preceq_{c}$-increasing} (respectively, \dfn{$\preceq_c$-decreasing}) if there do not exist $1\leq i<j\leq r$ such that $t_j\preceq_c t_i$ (respectively, $t_i\preceq_c t_j$). We say the class $\CL_{\bpi}$ is \dfn{$\preceq_c$-increasing} (respectively, \dfn{$\preceq_c$-decreasing}) if $\bpi$ is (this is well defined).   
\begin{definition}\label{def:positive_cluster_complex} The \defn{positive $c^{-1}$-cluster complex} $\mathrm{Clus}^+(W,c^{-1})$ is the flag simplicial complex on $[N]:=\{1,2,\ldots,N\}$ with simplices given by those subsets $\{i_1<i_2<\cdots<i_k\} \subseteq [N]$ for which $a_{i_1}a_{i_2}\cdots a_{i_k}$ is a subword of a $\preceq_c$-decreasing reduced $T$-word for $c$. 
\end{definition}

The maximal simplices of $\mathrm{Clus}^+(W,c^{-1})$ correspond naturally to the $\preceq_c$-decreasing chains in $\MCh(W,c)$.  We visualize the positive cluster complex $\mathrm{Clus}^+(W,c^{-1})$ within $\Delta_c^+$ by sending the (abstract) simplex $\{i_1,\ldots,i_k\} \in \mathrm{Clus}^+(W,c^{-1})$ to the simplicial cone \[\Delta\big(\{i_1,\ldots,i_k\}\big):=\spange\{\dd_{a_{i_1}},\ldots,\dd_{a_{i_k}}\}.\]  
In particular, this defines $\Delta(A)$ for each $A$ in the set $\MCl(W,c^{-1})$ of maximal simplices in $\mathrm{Clus}^+(W,c^{-1})$. It is well known that this definition of $\Delta(A)$ agrees with the one given in \cref{subsubsec:common}.  

\begin{example}\label{ex:a3-2}
Continuing~\Cref{ex:a3-1'}, the image under $\Delta$ of $\mathrm{Clus}^+(W,c^{-1})$ is illustrated in~\Cref{fig:a3arrangement2}.  There are six vertices, five edges (the lines between the pairs $\{\dd_{a_i},\dd_{a_3}\}$ for $i\in \{1,2,4,5,6\}$), and five triangles whose corresponding reduced $T$-words are $\preceq_c$-decreasing.  These triangles are listed in~\Cref{fig:max_clust}.
\begin{figure}[htbp]
\[
\begin{array}{ccc}\toprule
\text{simplex $A \subset [6]$} & \text{vertices $\dd_{a_i}$ for ${i \in A}$}  & \text{factorization $\prod_{i \in A} a_i$} \\ \midrule
\{1,2,3\} & (1,1,1,0),(1,1,0,0),(1,0,0,0) & (34)(24)(14)\\
\{2,3,4\} & (1,1,0,0),(1,0,0,0),(1,1,0,1) & (24)(14)(23)\\
\{3,4,5\} & (1,0,0,0),(1,1,0,1),(1,0,0,1) & (14)(23)(13)\\
\{3,5,6\} & (1,0,0,0),(1,0,0,1),(1,0,1,1) & (14)(13)(12)\\
\{1,3,6\} & (1,1,1,0),(1,0,0,0),(1,0,1,1) & (34)(14)(12)\\ \bottomrule
\end{array}\]
\caption{The five maximal simplices in $\mathrm{Clus}^+(\SSS_4,s_3s_2s_1)$.}
\label{fig:max_clust}
\end{figure}
\end{example}

\subsubsection{Commutation classes of chains} 

By~\Cref{lem:linearly_independent}, the commutation class $\CL_{\bpi}$ of a maximal chain ${\bpi\in\MCh(W,c)}$ is specified by the set of reflections used in $\rw(\bpi)$.  We visualize $\CL_{\bpi}$ inside $\Delta_c^+$ by letting \[\Delta(\CL_{\bpi}):=\Delta(\bpi)=\spange\{\dd_{t_1},\ldots,\dd_{t_r}\},\] where $\rw(\bpi)=t_1\cdots t_r$. 
In this visualization, each $\Delta(\bpi)$ is a union of the image under $\Delta$ of some collection of maximal simplices of $\mathrm{Clus}^+(W,c^{-1})$.

\begin{example}\label{ex:a3-3}
Continuing~\Cref{ex:a3-2}, there are 12 commutation classes in $\Class_c$ corresponding to the 12 triangles on the right side of~\Cref{fig:a3arrangement2}.  The 5 maximal simplices in $\MCl(W,c^{-1})$ correspond to the 5 \tsmall{} triangles in this picture---that is, those triangles that are not contained in any other triangle. 
\end{example}

\subsubsection{Parabolic subgroups}

The next lemma follows from \cite[Lemmas 2.6.5~\&~2.9.1]{stump2025cataland}.
 
\begin{lemma}\label{lem:preceq_c2}
Let $c$ be a standard Coxeter element of $W$, and let $s$ be a right descent of $c$. For all $t,t'\in T\setminus\{s\}$, we have $t\preceq_c t'$ if and only if $s t s\preceq_{scs}st's$. If $t,t'\in T\cap W_{\langle s\rangle}$, then $t\preceq_c t'$ if and only if $t\preceq_{cs} t'$. 
\end{lemma}

Given reflections $t$ and $t'$, let \[[t\mid t']_m=\underbrace{tt't\cdots}_m\] be the product of the word of length $m$ that starts with $t$ and alternates between $t$ and $t'$. 
Let $W'$ be a rank-2 parabolic subgroup of $W$ with reflections $T'=T\cap W'$, and let $k=|T'|$ be the number of reflections in $W'$. It is possible to order the reflections in $W'$ as $p_1,\ldots,p_k$ so that 
\[\Span_{\geq 0}\{\beta_{p_{i'}},\beta_{p_{j'}}\}\subseteq\Span_{\geq 0}\{\beta_{p_i},\beta_{p_{j}}\}\] for all $1\leq i\leq i'\leq j'\leq j\leq k$. The reflections $p_1$ and $p_k$ are called the \dfn{canonical generators} of $W'$. 
In fact, $W'$ is a Coxeter group with Coxeter generators $p_1$ and $p_k$. We have $p_i=[p_1\mid p_k]_{2i-1}$ for all $1\leq i\leq k$. If $c$ is a standard Coxeter element of $W$ and $k\geq 3$, then either $p_1\prec_cp_2\prec_c\cdots\prec_c p_k$ or $p_k\prec_cp_{k-1}\prec_c\cdots\prec_c p_1$. 

\begin{lemma}[{\cite[Proposition 4.1]{reading2011sortable}}]\label{lem:i-1modk} 
Let $c$ be a standard Coxeter element of $W$. Let $W'$ be a nonabelian rank-2 parabolic subgroup of $W$, and write $T\cap W'=\{p_1\prec_cp_2\prec_c\cdots\prec_c p_k\}$. Suppose $p_i \neq p_j$ and $p_i$ appears to the left of $p_j$ in some reduced $T$-word for $c$. Then ${j\equiv i-1\pmod{k}}$. 
\end{lemma}

\begin{remark}
Let $W'$ be a nonabelian rank-2 noncrossing parabolic subgroup of $W$ with canonical generators $p_1$ and $p_k$ and parabolic Coxeter element $c'=p_1p_k$.  \Cref{lem:i-1modk} says that the reflections in $W'\cap T$ that are $c$-noncrossing are exactly the ones that are adjacent in $\prec_{c'}$-order, along with the canonical generators $p_1$ and $p_k$.  Thus, when $W$ is simply-laced so that all rank-2 parabolic subgroups have two or three reflections, all pairs of reflections in each rank-$2$ parabolic subgroup are $c$-noncrossing. 
\end{remark}

\subsubsection{Coxeter-sortable elements}
Let $I_c^{(k)}(w)$ denote the set of letters (simple reflections) in the $k$-th copy of ${\sf c}$ in ${\sf c}^\infty$ that are used when constructing the $c$-sorting word $\w(c)$; it is known that this is well defined in the sense that it does not depend on the choice of the reduced $S$-word ${\sf c}$. We say $w$ is \dfn{$c$-sortable} if we have the chain of containments $I_c^{(1)}(w)\supseteq I_c^{(2)}(w)\supseteq I_c^{(3)}(w)\supseteq\cdots$. Let $\Sort(W,c)$ denote the set of $c$-sortable elements of $W$. Reading proved in~\cite[Proposition 3.2]{reading2007sortable} that for each $u\in W$, the set
\[\{v \in \Sort(W, c): v \leq_R u\}\]
has a maximum element, which we denote by $\pi_c^\downarrow(u)$. The fibers of the map $\pi_c^\downarrow\colon W\to \Sort(W,c)$ are the equivalence classes of an equivalence relation (in fact, a lattice congruence of the right weak order) called the \dfn{$c$-Cambrian congruence}.

\subsection{Bruhat order} 
Recall that we write $\leqB$ for the Bruhat order on $W$. There is a standard characterization of the Bruhat order in terms of subwords of reduced $S$-words. Namely, for $u,v\in W$, the following are equivalent: 
\begin{itemize}
\item $u\leqB v$; 
\item there exists a reduced $S$-word for $v$ that contains a reduced $S$-word for $u$ as a subword; 
\item every reduced $S$-word for $v$ contains a reduced $S$-word for $u$ as a subword.  
\end{itemize} 
Suppose $u,w\in W$ are such that $u\leqB w$. Let $s\in S$. The \dfn{lifting property} of Bruhat order (see \cite[Proposition~2.2.7]{BjornerBrenti}) states that if $u\leqB us$ and $ws\leqB w$, then $u\leqB ws$ and $us\leqB w$. There is an equivalent version of the lifting property, which states that if $u\leqB su$ and $sw\leqB w$, then $u\leqB sw$ and $su\leqB w$. An immediate consequence of the lifting property is the following lemma. 

\begin{lemma}\label{lem:Bruhat_descents}
Let $u,v\in W$. If $s$ is a right descent of both $u$ and $v$, then $u\leqB v$ if and only if $us\leqB vs$. If $s$ is a left descent of both $u$ and $v$, then $u\leqB v$ if and only if $su\leqB sv$.  
\end{lemma}

As mentioned in \cref{sec:intro}, we are interested in the set $W_c^+$ of elements $u\in W$ such that $\ell_S(uc)=\ell_S(u)+r$, where $c$ is a standard Coxeter element of $W$. One can show that \[W_c^+=\{u\in W:u\leq_{\LL}\wo c^{-1}\}.\] We are interested in the Bruhat intervals of the form $[u,uc]_{\hB}$ for $u\in W_c^+$, which are used to construct the Bruhat interval polytopes in the subdivision $\BIPs_c$. Let $\widetilde\Omega(W,c)$ denote the set of maximal chains in these intervals. That is, $\widetilde\Omega(W,c)$ is the set of saturated chains ${\{u_0\lessdot_{\hB}u_1\lessdot_{\hB}\cdots\lessdot_{\hB} u_r\}}$ such that $u_r=u_0c$. The following straightforward lemma states that these chains correspond precisely to the vertex sets of the simplices in $\SBDW_\yy(W,c)$. 

\begin{lemma}\label{lem:(wC)-to-saturated-chains}
    The map $\Omega(W,c) \to \widetilde\Omega(W,c)$ given by $(w, \C) \mapsto w\C$ is a bijection. We have
    \[\SBDW_\yy(W,c)= \{\conv(\tC\yy): \tC \in \widetilde\Omega(W,c)\}.\]
\end{lemma}
\begin{proof}
Let $(w,\C)\in\Omega(W,c)$, and write $\C=\{e=\pi_0 \lessdot_T \pi_1 \lessdot_T \cdots \lessdot_T \pi_r=c\}.$
    By the definition of $\Omega(W,c)$, the elements of $w\C$ satisfy $w=w\pi_0 \le_{\hB} w\pi_1 \le_{\hB} \cdots \le_{\hB} w\pi_r=wc$. As Bruhat order is ranked by Coxeter length, this implies that $\ell_S(wc) \ge \ell_S(w) +r =\ell_S(w)+ \ell_S(c)$. On the other hand, $\ell_S(wc) \le \ell_S(w)+\ell_S(c)$ since concatenating reduced $S$-words for $w$ and $c$ gives a word for $wc$. So $\ell_S(wc)=\ell_S(w)+\ell_S(c)$, and $w=w\pi_0 \lessdot_{\hB} w\pi_1 \lessdot_{\hB} \cdots \lessdot_{\hB} w\pi_r=wc$. The image of the map is thus contained in $\widetilde\Omega(W,c)$. We can also clearly recover $w$ and $\C$ from $w\C$, so the map is injective.

    To prove surjectivity, choose a chain $\{u_0 \lessdot_{\hB} u_1 \lessdot_{\hB} \cdots \lessdot_{\hB} u_r=u_0c\}$ in $\widetilde\Omega(W,c)$. Set $t_i:=u_{i-1}^{-1} u_i$. Then $t_1 \cdots t_r =c$ is a reduced $T$-word for $c$, so it is of the form $\rw(\C)$ for some $\C \in \MCh(W,c)$. The pair $(u_0, \C)$ is an element of $\Omega(W,c)$ that maps to the selected chain in $\widetilde\Omega(W,c)$. 

    The second sentence of the lemma follows immediately from the first. 
\end{proof}

We will need some lemmas about rank-2 intervals of the Bruhat order and rank-2 parabolic subgroups of $W$, which we collect here. 

Given points $\xx,\xx'$, let us write $[\xx,\xx']$ for the line segment with endpoints $\xx$ and $\xx'$. 

\begin{lemma}\label{lem:quadrilateral} 
Let $v,w\in W$ be such that $v\leqB w$ and $\ell_S(w)=\ell_S(v)+2$. There exists distinct $x,y\in W$ such that $[v,w]_{\hB}=\{v,x,y,w\}$ and such that $v\lessdot_{\hB} x\lessdot_{\hB}w$ and $v\lessdot_{\hB} y\lessdot_{\hB} w$. 

Let $t_1,t_2,t_3,t_4\in T$ be the reflections such that $x=t_1v$, $w=t_2x$, $y=t_3v$, and $w=t_4y$. Then the subgroup of $W$ generated by $t_1,t_2,t_3,t_4$ is a parabolic subgroup of $W$ of rank $2$. 

For $\yy\in \BB$, the points $v\yy,x\yy,y\yy,w\yy$ are the vertices of a quadrilateral whose edges are $[v\yy,x\yy]$, $[x\yy,w\yy]$, $[v\yy,y\yy]$, $[y\yy,w\yy]$. 
\end{lemma} 

\begin{proof}
The fact that $[v,w]_{\hB}$ is a $4$-element diamond is a standard fact about short intervals in Bruhat order; see \cite[Section~2.8]{BjornerBrenti}.
Let $W'$ be the subgroup of $W$ generated by $t_1,t_2,t_3,t_4$. The fact that $W'$ is a parabolic subgroup of $W$ of rank $2$ is due to Dyer~\cite[Lemma 3.1]{dyer1991bruhat}. Thus, $v\yy,x\yy,y\yy,w\yy$ are the vertices of a quadrilateral. Let $H_1,H_2,H_3,H_4$ be the hyperplanes in $\mathcal H$ corresponding to $t_1,t_2,t_3,t_4$, respectively. Let $\mathcal H'=W'H_1$ be the smallest reflection subarrangement of $\mathcal H$ containing $H_1,H_2,H_3,H_4$. Let $\BB'$ be the region of $\mathcal H'$ containing $\BB$. Let $V'$ be the $2$-dimensional affine span of $v\yy,x\yy,y\yy,w\yy$. The hyperplane arrangement $\mathcal H'_{V'}=\{H\cap V':H\in\mathcal H'\}$ is the Coxeter arrangement of a dihedral group $W''\cong W'$. The base region of $\mathcal H'_{V'}$ is $\BB'\cap V'$, and the simple reflections of $W''$ are the reflections through the hyperplanes in $\mathcal H'_{V'}$ that bound $\BB'\cap V'$. There is a unique point $\yy'$ in $W'v\yy$ that also lies in $\BB'\cap V'$. Let $u_v,u_x,u_y,u_w$ be the elements of $W''$ such that $u_v\yy'=v\yy$, $u_x\yy'=x\yy$, $u_y\yy'=y\yy$, $u_w\yy'=w\yy$. The point $u_x\yy'=x\yy$ is obtained by reflecting the point $u_v\yy'=v\yy$ through the line $H_1\cap V'$, and the points $u_x\yy'$ and $\yy'$ lie on opposite sides of $H_1$. This implies that $u_v<_{\hB}u_x$ in the Bruhat order on $W''$. By similar arguments, we have $u_x<_{\hB}u_w$, $u_v<_{\hB}u_y$, and $u_y<_{\hB}u_w$ (all in the Bruhat order on $W''$). Upon inspection of the dihedral Coxeter arrangement of $W''$, we find that this implies that the edges of the quadrilateral with vertices $u_v\yy'=v\yy,u_x\yy'=x\yy,u_y\yy'=y\yy,u_w\yy'=w\yy$ are exactly as claimed in the statement of the lemma.  
\end{proof}

The next lemma will be crucial for proving \cref{thm:triangulation,thm:spans_new}. 

\begin{lemma}\label{lem:basification} 
Let $c$ be a standard Coxeter element of $W$. Let $v,x,y,w\in W$ form a diamond in Bruhat order, where $v\lessdot_{\hB} x\lessdot_{\hB}w$ and $v\lessdot_{\hB} y\lessdot_{\hB} w$. Let $t_1,t_2,t_3,t_4\in T$ be the reflections such that $x=vt_1$, $w=xt_2$, $y=vt_3$, and $w=yt_4$. Then the subgroup $W'$ of $W$ generated by $t_1,t_2,t_3,t_4$ is a parabolic subgroup of $W$ of rank $2$. If there exists $u\in W_c^+$ such that $u\leqB v\leqB w\leqB uc$, then the set of canonical generators of $W'$ is either $\{t_1,t_2\}$ or $\{t_3,t_4\}$. 
\end{lemma}
\begin{proof} 
The fact that $W'$ has rank $2$ follows from \cite[Lemma 3.1]{dyer1991bruhat}. Now suppose $u\in W_c^+$ is such that $u\leqB v\leqB w\leqB uc$. Suppose by way of contradiction that the set of canonical generators of $W'$ is neither $\{t_1,t_2\}$ nor $\{t_3,t_4\}$. Then $W'$ has more than two reflections, so it is nonabelian. Let $p_1,\ldots,p_k$ be the reflections in $W'$, listed so that $p_1\preceq_c p_2\preceq_c\cdots\preceq_c p_k$. The assumption that $u\leqB v\leqB w\leqB uc$ ensures that $t_1$ and $t_2$ appear in a reduced $T$-word for $c$, with $t_1$ preceding $t_2$. Similarly, $t_3$ and $t_4$ appear in a reduced $T$-word for $c$, with $t_3$ preceding $t_4$. It follows from \cref{lem:i-1modk} that there exist $i,j\in[k-1]$ such that $t_1=p_{i+1}$, $t_2=p_{i}$, $t_3=p_{j+1}$, and $t_4=p_j$. Without loss of generality, we may assume $i<j$. Let $a=k+i-j+1$, and note that $i+1\leq a\leq k$.

Observe that $t_1=p_{i+1}\in T_\R(x)$ and $t_2=p_i\not\in T_\R(x)$. Because the set of positive roots associated to the right inversions of $x$ is biclosed (see \cite[Proposition~2.2]{Hohlweg} and \cite[Proposition~5.10]{Kostant}), we have $T_\R(x)\cap W'=\{p_m:i+1\leq m\leq k\}$. Similarly, we have $t_3=p_{j+1}\in T_\R(y)$ and $t_4=p_j\not\in T_\R(y)$, so $T_\R(y)\cap W'=\{p_m:j+1\leq m\leq k\}$. In particular, we have $p_{a}\in T_\R(x)$ and $p_{1}\not\in T_\R(y)$. Hence, 
\[\ell_S(yp_1)>\ell_S(y)=\ell_S(x)>\ell_S(xp_{a}).\] 
To obtain our desired contradiction, we will show that $yp_1=xp_{a}$. We have 
\[yp_1=vp_{j+1}p_1=v[p_1\mid p_k]_{2j+1}p_1=v[p_1\mid p_k]_{2j}\] and 
\[xp_a=vp_{i+1}p_a=v[p_1\mid p_k]_{2i+1}[p_1\mid p_k]_{2a-1}=v[p_k\mid p_1]_{2a-2i-2}=v[p_k\mid p_1]_{2k-2j}.\] Because $W'$ is a dihedral Coxeter group with $k$ reflections whose canonical generators are $p_1$ and $p_k$, we have $[p_1\mid p_k]_{b}=[p_k\mid p_1]_{2k-b}$ for all $0\leq b\leq k$. Hence, $[p_1\mid p_k]_{2j}=[p_k\mid p_i]_{2k-2j}$, so $yp_1=xp_a$. 
\end{proof} 

\subsection{Quivers and stability functions}\label{subsec:quivers} 
One ingredient in our proof of \cref{thm:triangulation} are certain \emph{slope functions} which arise in quiver representation theory. We briefly review background here, and give some combinatorial translations of representation theoretic concepts which will be sufficient for our purposes.

Let $\Q$ be a quiver (i.e., a finite directed graph) with vertex set $\Q_0$ and arrow set $\Q_1$. A \dfn{path} in $\Q$ is a finite sequence $a_0\to a_1\to\cdots \to a_k$ of arrows of $\Q$. 
A \dfn{representation} $M$ of $\Q$ (over $\CC$) is an assignment of a finite-dimensional complex vector space $M_a$ to each vertex $a\in\Q_0$ and a linear map $M_{a\to a'}\colon M_a\to M_{a'}$ to each arrow $a\to a'$ in $\Q_1$. There is a natural notion of the direct sum of two representations of $\Q$; a representation is \dfn{indecomposable} if it is not isomorphic to a direct sum of two nonzero representations.

Let us assume in this subsection that the finite Coxeter group $W$ is simply-laced. Let $s_1,\ldots,s_r$ be the simple reflections of $W$, and let $\alpha_i=\beta_{s_i}$ be the simple root corresponding to $s_i$. The \dfn{Dynkin diagram} of $W$ is the (undirected graph) with vertices $a_{1},\ldots,a_r$ in which two vertices $a_i$ and $a_{i'}$ are adjacent if and only if $s_i$ and $s_{i'}$ do not commute (equivalently, $s_is_{i'}$ has order $3$). A \dfn{Dynkin quiver of type $W$} is a quiver obtained by orienting each edge in the Dynkin diagram of $W$ into an arrow. There is a bijection $c\mapsto\Q^c$ from the set of standard Coxeter elements of $W$ to the set of Dynkin quivers of type $W$. In the quiver $\Q^c$, an edge with vertices $a_i,a_{i'}$ is oriented to form an arrow $a_i\to a_{i'}$ if and only if $s_i$ appears to the left of $s_{i'}$ in some (equivalently, every) reduced $S$-word for~$c$. 

\begin{example}
Suppose $W$ is of type $A_r$, and order the simple reflections of $W$ so that $s_is_{i+1}$ has order $3$ for every $1\leq i\leq r-1$. The Dynkin diagram of $W$ is a path graph with vertices $a_{1},\ldots,a_{r}$ and with $a_{i}$ adjacent to $a_{i+1}$ for all $1\leq i\leq r-1$. For $r=6$ and $c=s_1s_2s_4s_3s_6s_5$, $\Q^c$ is the quiver 
\[\begin{array}{l}\includegraphics[height=0.595cm]{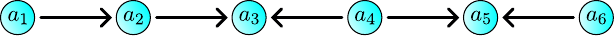}\end{array}.\]
\end{example} 

The \dfn{dimension vector} of a representation $M$ of $\Q^c$ is the vector \[\dim(M)=\sum_{i=1}^r\dim(M_{a_i})\alpha_i\in V^*.\] 
The \dfn{Auslander--Reiten quiver} of $\Q^c$ is a quiver $\mathrm{AR}(\Q^c)$ whose vertices correspond to isomorphism classes of indecomposable representations of $\Q^c$ and whose arrows correspond to irreducible morphisms. The map $\dim$ restricts to a one-to-one correspondence from the set of (isomorphism classes of) indecomposable representations of $\Q^c$ to the set $\Phi^+$ of positive roots, so one can also identify the vertex set of $\mathrm{AR}(\Q^c)$ with $\Phi^+$. With this identification, there is a purely combinatorial description of the arrows in $\mathrm{AR}(\Q^c)$. Namely, for reflections $t,t'\in T$, there is a path from $\beta_t$ to $\beta_{t'}$ in $\mathrm{AR}(\Q^c)$ if and only if $t\preceq_c t'$ \cite[Theorem~9.3.1]{stump2025cataland}. When $\beta$ and $\beta c^{-1}$ are both positive roots, $\beta c^{-1}$ is the \dfn{Auslander--Reiten translation} of $\beta$. We refer to \cite{ASS,stump2025cataland} for the more standard representation-theoretic definition of Auslander--Reiten translation in terms of indecomposable modules; our need for this operation is restricted to the next lemma. 

\begin{lemma}[{\cite[Theorem~9.3.1]{stump2025cataland}}]\label{lem:AR_translation}
Let $c$ be a standard Coxeter element of $W$. If $t\in T$ is such that $\beta_t$ and $\beta_t c^{-1}$ are both positive roots, then $t\preceq_c ctc^{-1}$. Equivalently, there is a path from $\beta_t$ to $\beta_t c^{-1}$ in $\mathrm{AR}(\Q^c)$. 
\end{lemma}

Let $K_0(\Q^c)$ denote the Grothendieck group of the category of finite-dimensional representations of $\Q^c$. The map $\dim$ descends to a group isomorphism from $K_0(\Q^c)$ to the lattice $\mathrm{span}_{\ZZ}\{\alpha_1,\ldots,\alpha_r\}$ in $V^*$. It is typical to define stability functions as maps from $K_0(\Q^c)$ to $\CC$; however, we will find it more convenient to use dimension vectors in order to work directly with $V^*$ rather than $K_0(\Q^c)$. Thus, a \dfn{linear stability function} is a linear function $Z\colon V^*\to\CC$. 

Given a point $\yy\in\BB$ and a vector $\gamma\in V^*$, we obtain a linear stability function $Z_{\yy,\gamma}$ given by 
\[Z_{\yy,\gamma}(\beta)=\langle\gamma,\beta^\vee\rangle+ \langle\beta,\yy\rangle\sqrt{-1}.\]
Associated to $Z_{\yy,\gamma}$ is the \dfn{slope function} $\mu_{\yy,\gamma}\colon V^*\to\RR$ defined by 
\[\mu_{\yy,\gamma}(\beta)=\frac{\langle\gamma,\beta^\vee\rangle}{\langle\beta,\yy\rangle}.\] 
A representation $M$ of $\Q^c$ is said to be \dfn{$Z_{\yy,\gamma}$-stable} if $\mu_{\yy,\gamma}(\dim(M'))>\mu_{\yy,\gamma}(\dim(M))$ for every nonzero proper subrepresentation $M'$ of $M$. The function $Z_{\yy,\gamma}$ is \dfn{totally stable} for $\Q^c$ if every indecomposable representation of $\Q^c$ is $Z_{\yy,\gamma}$-stable. 

\begin{lemma}[{\cite[Lemma~2.9]{diaz2022total}}] 
Let $c$ be a standard Coxeter element of $W$, and fix $\yy\in\BB$ and $\gamma\in V^*$. The linear stability function $Z_{\yy,\gamma}$ is totally stable for $\Q^c$ if and only if $\mu_{\yy,\gamma}(\beta_t)<\mu_{\yy,\gamma}(\beta_{t'})$ for all $t,t'\in T$ such that $t\preceq_c t'$. 
\end{lemma}

Let $\rho\in V$ denote the sum of the fundamental weights; in other words, $\rho$ is determined by the fact that $\langle\alpha_i,\rho\rangle=1$ for every $i\in[r]$. Reineke \cite{reineke2003harder} conjectured that for every standard Coxeter element $c$, there exists $\gamma\in V^*$ such that $Z_{\rho,\gamma}$ is totally stable for $\Q^c$. In this situation, the denominator of the slope function is just $\dim(M)$. When $W$ is of type~$A$, this conjecture was proven independently by Apruzzese--Igusa~\cite{apruzzese2020stability}, Huang--Hu~\cite{HuangHu}, and Kinser~\cite{kinser2022total}. Reineke's conjecture was also proven when $c$ is a bipartite Coxeter element in unpublished work by Hille–Juteau~\cite[Corollary 1.7]{keller2011cluster} (see also~\cite[Section 7.4]{qiu2015stability}).\footnote{\label{foot:bipartite}In his lecture at the Banff International Research Station workshop 18w5178 on ``Stability Conditions and Representation Theory of Finite-Dimensional Algebras''~\cite{hillevideo}), Hille appears to indicate in representation-theoretic language that when $c$ is bipartite, one should fold the arrangement $\mathcal H$ to the real Coxeter plane (see, e.g.,~\cite{steinberg:FiniteReflectionGroups} or~\Cref{sec:folding}). The key point is that, for bipartite $c$, the hyperplanes for the simple reflections of $W$ fold to the simple reflections for the dihedral group $I_2(h)$.  With a small perturbation, one can then see the hyperplanes in the Coxeter plane appearing in the same order as they do in $\inv(w_\circ(c))$---total stability follows from an analysis of rank-$2$ arrangements.  For further discussion of the relation of stability conditions to the Coxeter plane, see~\cite[Lemma 2.7]{qiu2025geometric}.}
However, Diaz--Gilbert--Kinser \cite{diaz2022total} and Marczinzik \cite{marczinzik2023total} found counterexamples showing that the conjecture fails for certain choices of $c$ when $W$ is of type $D_n$ for $n\geq 9$ or of type $E_7$ or $E_8$. The following weaker version of Reineke's conjecture, which is sufficient for our purposes, was proven by Chang--Qiu--Zhang \cite{chang2024geometric}. 

\begin{proposition}[{\cite[Theorem~1.4]{chang2024geometric}}]\label{prop:Chang} 
Let $c$ be a standard Coxeter element of $W$. There exist $\yy^\star\in\BB$ and $\gamma^\star\in V^*$ such that the linear stability function $Z_{\yy^\star,\gamma^\star}$ is totally stable. 
\end{proposition} 

\begin{remark}
When $W$ is of type $A$, the article \cite{kinser2022total} completely classifies all totally stable linear stability functions by a system of inequalities; this line of investigation is extended in~\cite{diaz2022total} to other simply-laced types. 
\end{remark}

\subsection{Polyhedral geometry}\label{subsec:subdivisions}  
Here we recall the necessary background on polytopes and their subdivisions.

Let $P$ be a $d$-dimensional polytope in a Euclidean space. Given a linear functional $f$ on the space, we write $P_f$ for the face of $P$ on which $f$ is minimized. We write $\VV(P)$ for the set of vertices of $P$. The \dfn{face lattice} of $P$ is the collection $\Faces(P)$ of faces of $P$ partially ordered by containment. Two polytopes $P$ and $P'$ are \dfn{combinatorially isomorphic} if there is a poset isomorphism  from $\Faces(P)$ to $\Faces(P')$. 

A \dfn{subdivision} of a $d$-dimensional polytope $P$ is a collection $\Theta$ of $d$-dimensional polytopes with vertices contained in $\VV(P)$ such that $\bigcup\Theta=P$ and such that for all $P_1,P_2\in\Theta$, the intersection $P_1\cap P_2$ is a (possibly empty) common face of $P_1$ and $P_2$. A subdivision $\Theta$ is a \dfn{triangulation} if all polytopes in $\Theta$ are simplices. We sometimes call the polytopes in a subdivision the \dfn{cells} of the subdivision.

The following lemma is an immediate consequence of the main result of \cite{facet-to-facet}. 

\begin{lemma}[\cite{facet-to-facet}]\label{lem:facet-to-facet}
Let $P$ be a $d$-dimensional polytope in a Euclidean space, and let $\Theta$ be a collection of $d$-dimensional polytopes with vertices in $\VV(P)$ such that $\bigcup \Theta=P$.  
Suppose that the polytopes in $\Theta$ have pairwise disjoint interiors. Suppose that for all $R,R'\in\Theta$ such that $R\cap R'$ has dimension $d-1$, the intersection $R\cap R'$ is a common facet of $R$ and $R'$. Then $\Theta$ is a subdivision of $P$. 
\end{lemma}

Let $U$ be a $d$-dimensional Euclidean space. Consider the $(d+1)$-dimensional space $U\oplus\RR$. We write the elements of this space as $(\zz,a)$, where $\zz\in U$ and $a\in\RR$. We will sometimes write ${\bf 0}$ for the zero vector in $U$. There is a natural projection $\ppp\colon U\oplus \RR\to U$ given by $\ppp(\zz,a)=\zz$. 

Let $P$ be a $d$-dimensional polytope in $U$. A \dfn{height function} on $P$ is a map $\hgt\colon\VV(P)\to\RR$. For $\vvv\in\VV(P)$, let $\vvv^\hgt=(\vvv,\hgt(\vvv))\in U\oplus\RR$ denote its \dfn{lift} via the height function $\hgt$. For $\mathfrak X\subseteq\VV(P)$, let $\mathfrak X^\hgt=\{\vvv^\hgt:\vvv\in\mathfrak X\}$. The height function $\hgt$ gives rise to the \dfn{lifted polytope} 
\[P^\hgt=\conv(\VV(P)^\hgt)\subseteq U\oplus\RR,\] which is obtained as the convex hull of all of the lifts of vertices of $P$. A \dfn{lower face} of $P^\hgt$ is a face $P_f^\hgt$ where $f$ is a linear functional on $U\oplus\RR$ such that $f({\bf 0},1)>0$. The collection
\[\{\ppp(P_f^\hgt) \colon P_f^\hgt \text{ a lower face of dimension } d  \}\]
is a subdivision $\Theta_\hgt$ of $P$. A subdivision obtained in this manner is said to be \dfn{regular}. We say the regular subdivision $\Theta_\hgt$ is \dfn{induced} by the height function $\hgt$.  

Recall that a \dfn{cone} in the vector space $U$ is a set that is closed under addition and multiplication by nonnegative scalars. The \dfn{lineality space} of a cone is the largest linear subspace contained in the cone. 

In our proof of \cref{thm:triangulation}, we will need an appropriate version of the Hyperplane Separation Theorem for intersecting polytopes. While this result seems similar to many well-known separation theorems, we were unable to find it in the literature, so we provide a proof here. First, we need the following lemma about cones. 

\begin{lemma}\label{lem:hyperplane_cones} 
Let $C$ and $C'$ be polyhedral cones in $U$ that have the same lineality space $L$. If $C\cap C'=L$, then there exists a hyperplane $H$ in $U$ such that $H \cap C=H\cap C'=L$. 
\end{lemma}
\begin{proof}
By passing to the quotient space $U/L$ if necessary, we may assume that $L=\{{\bf 0}\}$. Thus, we assume that $C\cap C'=\{{\bf 0}\}$. Let $\mathbb S$ be the unit sphere in $U$. The sets $C\cap \mathbb S$ and $C'\cap \mathbb S$ are disjoint and closed. Moreover, these sets are \emph{spherical convex} in the sense of~\cite{han2020spherical}. According to \cite[Theorem~1]{han2020spherical}, there is a linear functional $\psi\colon U\to\RR$ such that $\psi(\zz)>0>\psi(\zz')$ for all $\zz\in C\cap\mathbb S$ and $\zz'\in C'\cap \mathbb S$. Let $H$ be the kernel of $\psi$. Then $H \cap C=H\cap C'=\{{\bf 0}\}=L$, as desired. 
\end{proof}

\begin{lemma}\label{lem:hyperplane_polytopes}
Let $R$ and $R'$ be polytopes in $U$ such that the set $F=R\cap R'$ is a (possibly empty) common face of $R$ and $R'$. There exists a hyperplane $H$ in $U$ such that $H \cap R=H\cap R'=F$.  
\end{lemma}
\begin{proof}
If $F=\emptyset$, then the desired result follows immediately from the Hyperplane Separation Theorem. Now suppose $F\neq\emptyset$. By translating if necessary, we may assume ${\bf 0}$ is in the relative interior of $F$. Let $C$ and $C'$ be the cones generated by $R$ and $R'$, respectively. These cones have the same lineality space $L$, which is the linear span of $F$. We have $C\cap C'=L$, so by \cref{lem:hyperplane_cones}, there is a hyperplane $H$ such that $H \cap C=H\cap C'=L$. It follows that $H \cap R=H\cap R'=F$. 
\end{proof}

\section{Results about Subdivisions}\label{sec:subdivisions}  

The purpose of this section is to collect and prove some results about polytopal subdivisions that we will need in \cref{sec:triangulating}. As in \cref{subsec:subdivisions}, we let $U$ denote a $d$-dimensional Euclidean space. 

\subsection{Regular subdivisions} 
The following lemma provides a sufficient condition to guarantee that a collection of polytopes covers a given polytope. 

\begin{lemma}\label{lem:union-covers-real} 
Let $P\subseteq U$ be a $d$-dimensional polytope, and let $\Theta$ be a nonempty finite collection of $d$-dimensional polytopes with vertices contained in $\VV(P)$. Suppose that for every $R\in\Theta$ and every facet $F$ of $R$, either $F$ is contained in the boundary of $P$ or $F=R\cap R'$ for some $R'\in\Theta$. Then $P=\bigcup\Theta$. 
\end{lemma}
\begin{proof} 
Let $X$ be the union of the $(d-2)$-dimensional faces of the polytopes in $\Theta$. Thus, $X$ is a codimension-$2$ subset of $P$. 

Fix $\xx\in P$; we will show that $\xx$ belongs to some polytope in $\Theta$. Choose a point $\xx'$ in the interior of some polytope in $\Theta$. Let us choose $\xx'$ generically so that the line segment with endpoints $\xx$ and $\xx'$ does not intersect $X$. We can parameterize this line segment via the function $\psi\colon[0,1]\to P\setminus X$ defined by $\psi(\tau)=\tau\xx'+(1-\tau)\xx$. The set $\bigcup\Theta$ is closed, so its preimage $\psi^{-1}(\bigcup\Theta)$ is a closed subset of $[0,1]$. It follows that  $\psi^{-1}(\bigcup\Theta)$ contains its supremum, which we denote by $\tau^*$. Our goal is to show that $\tau^*=1$. 

Suppose by way of contradiction that $\tau^*\neq 1$. Note that $\psi(\tau^*)$ cannot lie in the interior of any of the polytopes in $\Theta$. Hence, there exists $R\in\Theta$ such that $\psi(\tau^*)$ is in the boundary of $R$. Since the point $\psi(\tau^*)$ is not in $X$, it must be in the relative interior of some facet $F$ of $R$. Because $\xx'$ is in the interior of $P$ and $\tau^*<1$, we know that $\psi(\tau^*)$ is not in the boundary of $P$. The hypothesis of the lemma ensures that $F= R\cap R'$ for some $R'\in\Theta$. It follows that there exists some $\epsilon>0$ such that $\psi(\tau^*+\epsilon)$ is in the interior of $R'$; this contradicts the definition of $\tau^*$. 
\end{proof} 

The next proposition provides sufficient conditions for a collection of polytopes to form the regular subdivision induced by a height function. The proof is very similar to that of \cite[Theorem~2.3.20]{DLRS-triang-book}, but since the assumptions of that theorem are slightly different, we repeat it here for the reader's convenience. Below, if $H$ is a hyperplane in $U\oplus\RR$ that is not orthogonal to a vector of the form $(\zz,0)$, then we say a point $(\xx,a)$ \dfn{lies strictly above} $H$ if it can be obtained from a point in $H$ by increasing the final coordinate. 

\begin{proposition}\label{prop:conditions-for-max-cells-in-reg-sub}
Let $P$ be a $d$-dimensional polytope, and let $\hgt\colon\VV(P)\to\RR$ be a height function. Suppose $\Theta$ is a collection of polytopes such that the following conditions hold:
	\begin{enumerate}
		\item[\emph{(i)}] Each polytope $R \in \Theta$ is $d$-dimensional and satisfies $\VV(R) \subset \VV(P)$. 
		\item[\emph{(ii)}] For each $R\in\Theta$, every facet of $R$ either lies on the boundary of $P$ or is equal to $R\cap R'$ for some $R'\in\Theta$.
		\item[\emph{(iii)}] For each $R\in\Theta$, there exists a hyperplane $H_R$ in $U\oplus \RR$ containing $\VV(R)^\hgt$.
            \item[\emph{(iv)}]\label{local_folding} If $R,R'\in\Theta$ are such that the set $F=R \cap R'$ is a common facet of $R$ and $R'$, then each vertex $\vvv\in\VV(R)\setminus F$ lifts to a point $\vvv^\hgt$ in $U \oplus \RR$ lying strictly above $H_{R'}$, and each vertex $\vvv'\in\VV(R')\setminus F$ lifts to a point $\vvv'^\hgt$ lying strictly above $H_R$.
	\end{enumerate}
	Then $\Theta$ is the regular subdivision of $P$ induced by $\hgt$.
\end{proposition}

\begin{proof}
	We will first show that for each $R\in\Theta$, the vertices of $P$ that are not vertices of $R$ lift to points strictly above the hyperplane $H_R$. 
	
	Choose $R\in\Theta$, and let $\vvv$ be a vertex of $P$ that is not in $R$. As in the proof of \cref{lem:union-covers-real}, let $X$ be the union of the $(d-2)$-dimensional faces of the polytopes in $\Theta$, and choose a point $\xx$ in the interior of $R$ such that the line segment with endpoints $\vvv$ and $\xx$ does not intersect $X$. 
	
	By \cref{lem:union-covers-real}, conditions (i) and (ii) imply that $P=\bigcup\Theta$, so $\vvv$ is in some polytope $R_0 \in\Theta$ with $R_0\neq R$. Thus, there are polytopes $R_0,R_1,\ldots,R_k\in\Theta$ with $R_k=R$ such that, as we traverse the line segment from $\vvv$ to $\xx$, we start at $R_{0}$, then cross a shared facet of $R_{0}$ and $R_{1}$ into $R_1$, then cross a shared facet of $R_1$ and $R_2$ into $R_2$, and so on, until ending at $\xx$ in $R_k$. We will prove by induction on $k$ that $\vvv^\hgt$ is strictly above the hyperplane $H_R$. If $k=1$, then this is assumption (iv). Now, suppose $k>1$. By induction, $\vvv^\hgt$ is strictly above the hyperplane $H_{R_{k-1}}$. Let $F$ be the facet equal to $R_{k-1} \cap R_k$, let $\nu$ be a normal vector of $F$, and let $\mathfrak C=\VV(R_{k-1})\setminus F$ be the set of vertices of $R_{k-1}$ that are not in $F$. Let $H \subset U \oplus \RR$ be the hyperplane spanned by $F$ and the vector $({\bf 0},1)$, which has normal vector $(\nu,0)$. By construction, the line segment from $\vvv$ to $\xx$ passes through $F$, so in $U$, the point $\vvv$ and the vertices in $\mathfrak C$ lie on the same side of the hyperplane spanned by $F$. By lifting, we find that $\vvv^\hgt$ lies on the same side of the hyperplane $H$ as $\mathfrak C^\hgt$. 
	
	Applying assumption (iv) to $R_{k-1}$ and $R_{k}$ tells us that $\mathfrak C^\hgt$ is above the hyperplane $H_{k}$. In particular, all points of $H_{R_{k-1}}$ that are on the same side of $H$ as $\mathfrak C^\hgt$ lie strictly above the hyperplane $H_{R_k}$. The point $\vvv^\hgt$ lies strictly above $H_{R_{k-1}}$ and on the same side of $H$ as $\mathfrak C^\hgt$. By decreasing the last coordinate of $\vvv^\hgt$, we obtain a point $\zz \in H_{R_{k-1}}$ that lies on the same side of $H$ as $\mathfrak C^\hgt$ and, therefore, lies strictly above $H_{R_k}$. Since $\vvv^{\hgt}$ lies above $\zz$, it also lies strictly above $H_{R_k}=H_R$.
	
	The above shows that each hyperplane $H_R$ for $R\in\Theta$ is a supporting hyperplanes of the lifted polytope $P^\hgt$, since all vertices of $P^\hgt$ are weakly above $H_R$. Moreover, $H_R \cap P^\hgt$ is exactly the lifted polytope $R^\hgt$, since all vertices of $P$ not in $R$ lift to points strictly above $H_R$. This implies that $R$ is a cell of the regular subdivision of $P$ induced by $\hgt$. Since $P$ is the union of the polytopes in $\Theta$ and cells of a regular subdivision have disjoint interiors, there are no other cells. 
\end{proof}

We call condition (iv) in \cref{prop:conditions-for-max-cells-in-reg-sub} the \dfn{local folding condition}. As noted in the sentence after \cite[Theorem~2.3.20]{DLRS-triang-book}, to show assumption (iv) holds for particular polytopes $R,R'\in \Theta$, it suffices to check that a single vertex of $R$ that is not in $F$ lifts to a point strictly above $H_{R'}$. 

\subsection{Deforming subdivisions} 

Our main result in this subsection is the following proposition, which allows us to continuously deform triangulations under certain conditions. 

\begin{proposition}\label{prop:deform}
Let $(P_\tau)_{\tau\in[0,1]}$ be a collection of $d$-dimensional polytopes in $U$. Let $\JJ_1,\ldots,\JJ_m\subseteq\VV(P_0)$ be $(d+1)$-element subsets of the vertex set of $P_0$. Suppose that for each $\tau\in[0,1]$, there is a bijection $\varphi_\tau\colon\VV(P_0)\to\VV(P_\tau)$. Assume the following conditions hold: 
\begin{itemize}
\item The map $\varphi_0$ is the identity map on $\VV(P_0)$. 
\item For each $\vvv\in\VV(P_0)$, the map $\tau\mapsto\varphi_\tau(\vvv)$ is a continuous function from $[0,1]$ to $U$. 
\item For every $\tau\in[0,1]$, there is a poset isomorphism from $\Faces(P_0)$ to $\Faces(P_\tau)$ that agrees with $\varphi_\tau$ on the set of vertices of $P_0$. 
\item For every ${\tau\in[0,1]}$, the polytopes in the set $\Theta_\tau=\{\conv(\varphi_\tau(\JJ_1)),\ldots,\conv(\varphi_\tau(\JJ_m))\}$ are $d$-dimensional simplices. 
\end{itemize} 
 If $\Theta_0$ is a triangulation of $P_0$, then $\Theta_\tau$ is a triangulation of $P_\tau$ for every $\tau\in[0,1]$.  
\end{proposition}

The main application of \cref{prop:deform} that we have in mind is proving \cref{thm:triangulation}. Indeed, in \cref{subsec:regular}, we will show that there exists $\yys\in\BB$ such that $\SBDW_{\yys}(W,c)$ is a triangulation of $\Perm_{\yys}$. For an arbitrary $\yy\in\BB$, we will then argue that we can continuously deform $\Perm_{\yys}$ into $\Perm_{\yy}$ in such a way that each of the simplices in the triangulation remains full-dimensional throughout the deformation.

Throughout this subsection, we preserve the hypotheses and notation from  \cref{prop:deform}. Before proving that proposition, we need the following lemmas. 

\begin{lemma}\label{lem:Z-open}
For all distinct $i,i'\in[m]$, the set \[\mathcal Z_{i,i'}=\{\tau\in[0,1]:\conv(\varphi_\tau(\JJ_i))\cap\conv(\varphi_\tau(\JJ_{i'}))\neq\conv(\varphi_\tau(\JJ_i\cap\JJ_{i'}))\}\] is a closed subset of $[0,1]$. 
\end{lemma} 
\begin{proof}
Let $\www_1,\ldots,\www_p$ be the vertices in $\JJ_i\cap \JJ_{i'}$. Fix affinely independent points $\uuu_1,\ldots,\uuu_p$ in $U$. For each $\tau\in[0,1]$, there exists an affine transformation $\chi_\tau\colon U\to U$ such that $\chi_\tau(\varphi_\tau(\www_k))=\uuu_k$ for all $k\in[p]$. We can choose the family $(\chi_\tau)_{\tau\in[0,1]}$ so that the polytopes $R_\tau=\conv(\chi_\tau(\varphi_\tau(\JJ_i)))$ and $R_\tau'=\conv(\chi_\tau(\varphi_\tau(\JJ_{i'})))$ are $d$-dimensional simplices. We can also assume that $\tau\in\mathcal Z_{i,i'}$ if and only if $R_\tau\cap R_\tau'\neq\conv(\{\uuu_1,\ldots,\uuu_k\})$. Moreover, because each of the maps $\tau\mapsto\varphi_\tau(\vvv)$ for $\vvv\in\VV(P_0)$ is continuous, we can choose the family $(\chi_\tau)_{\tau\in[0,1]}$ so that the map from $U\times [0,1]$ to $U$ given by $(\xx,\tau)\mapsto\chi_\tau(\xx)$ is continuous (we equip $U\times[0,1]$ with the natural product topology). Let $Y_\tau=\chi_\tau(\varphi_\tau(\JJ_i\setminus\JJ_{i'}))$ be the set of vertices of $R_\tau$ that are not in $\{\uuu_1,\ldots,\uuu_k\}$. Let $Y_\tau'=\chi_\tau(\varphi_\tau(\JJ_{i'}\setminus\JJ_{i}))$ be the set of vertices of $R_\tau'$ that are not in $\{\uuu_1,\ldots,\uuu_k\}$.

Now choose $\tau^*\in[0,1]\setminus\mathcal Z_{i,i'}$. By \cref{lem:hyperplane_polytopes}, there is a hyperplane $H$ in $U$ such that \[H \cap R_{\tau^*}=H\cap R_{\tau^*}'=\conv(\{\uuu_1,\ldots,\uuu_k\}).\] Let $H^+$ and $H^-$ be the open half-spaces bounded by $H$. The set $Y_{\tau^*}$ is contained in one of these open half-spaces, while $Y_{\tau^*}'$ is contained in the other; without loss of generality, assume $Y_{\tau^*}\subseteq H^+$ and $Y_{\tau^*}'\subseteq H^-$. For each $\vvv\in\VV(P_0)$, the map $\tau\mapsto \chi_\tau(\varphi_\tau(\vvv))$ is continuous. Therefore, there exists $\epsilon>0$ such that $Y_\tau\subseteq H^+$ and $Y_\tau'\subseteq H^-$ for all $\tau\in[0,1]$ satisfying $|\tau-\tau^*|<\epsilon$. It follows that every $\tau\in[0,1]$ satisfying $|\tau-\tau^*|<\epsilon$ is in $[0,1]\setminus\mathcal Z_{i,i'}$. As $\tau^*$ was arbitrary, this proves that $[0,1]\setminus\mathcal Z_{i,i'}$ is open, so $\mathcal Z_{i,i'}$ is closed.    
\end{proof}

\begin{lemma}\label{lem:common_facet}
Suppose $i,i'\in[m]$ are such that the simplices $\conv(\JJ_i)$ and $\conv(\JJ_{i'})$ intersect along a common facet. For every $\tau\in[0,1]$, the simplices $\conv(\varphi_\tau(\JJ_i))$ and $\conv(\varphi_\tau(\JJ_{i'}))$ intersect along a common facet, which is $\conv(\varphi_\tau(\JJ_i\cap\JJ_{i'}))$. 
\end{lemma}
\begin{proof}
Let us write $\JJ_i\setminus\JJ_{i'}=\{\www\}$ and $\JJ_{i'}\setminus\JJ_i=\{\www'\}$. Let $H_\tau$ be the hyperplane containing $\varphi_\tau(\JJ_i\cap\JJ_{i'})$. Let $\mathcal T$ be the set of $\tau\in[0,1]$ such that $\varphi_\tau(\www)$ and $\varphi_\tau(\www')$ lie in opposite open half-spaces bounded by $H_\tau$. Because each map $\tau\mapsto\varphi_\tau(\vvv)$ for $\vvv\in\VV(P_0)$ is continuous, $\mathcal T$ is an open subset of $[0,1]$ that contains $0$. We wish to show that $\mathcal T=[0,1]$, so suppose by way of contradiction that this is not that case. Then $[0,1]\setminus\mathcal T$ contains its minimum, which is some number $\tau^\#$. Either $\varphi_{\tau^\#}(\www)$ or $\varphi_{\tau^\#}(\www')$ lies in the hyperplane $H_{\tau^{\#}}$. This contradicts the fact that $\conv(\varphi_\tau(\JJ_i))$ and $\conv(\varphi_\tau(\JJ_{i'}))$ are $d$-dimensional. We conclude that $\mathcal T=[0,1]$, as desired. 
\end{proof} 

\begin{lemma}\label{lem:facet_or_boundary}
Suppose $\Theta_0$ is a triangulation of $P_0$. For every $\tau\in[0,1]$, every $i\in[m]$, and every facet $F$ of $\conv(\varphi_\tau(\JJ_i))$, either $F=\conv(\varphi_\tau(\JJ_i))\cap \conv(\varphi_\tau(\JJ_{j}))$ for some $j\in[m]$, or $F$ is contained in the boundary of $P_\tau$. 
\end{lemma}
\begin{proof}
The conclusion of the lemma certainly holds if $\tau=0$. Now suppose $\tau>0$. There is a facet $F_0$ of $\conv(\JJ_i)$ such that $\varphi_\tau(\VV(F_0))=\VV(F)$. If $F_0$ is contained in a facet of $P_0$, then it follows from the third bulleted item in the hypothesis of \cref{prop:deform} that $F$ is contained in a facet of $P_\tau$. Now suppose $F_0$ is not contained in a facet of $P_0$. There must be some index $j\in[m]$ such that $F_0=\conv(\JJ_i)\cap \conv(\JJ_{j})=\conv(\JJ_i\cap \JJ_j)$. Invoking \cref{lem:common_facet}, we find that  ${F=\conv(\varphi_\tau(\JJ_i\cap\JJ_j))=\conv(\varphi_\tau(\JJ_i))\cap \conv(\varphi_\tau(\JJ_{j}))}$.  
\end{proof}

\begin{proof}[Proof of \cref{prop:deform}]
Let $\mathcal T$ be the set of $\tau\in[0,1]$ such that $\Theta_\tau$ is not a triangulation of $P_\tau$; our goal is to show that $\mathcal T=\emptyset$. It follows from \cref{lem:union-covers-real} and the hypothesis of the proposition that $\bigcup_{i=1}^m\conv(\varphi_\tau(\JJ_i))=P_\tau$ for every $\tau\in[0,1]$. Therefore, $\mathcal T$ is equal to the set of $\tau\in[0,1]$ such that two simplices in $\Theta_\tau$ do not intersect along a common face. It follows from \cref{lem:Z-open} that $\mathcal T$ is a union of finitely many closed subsets of $[0,1]$, so it is also closed. 

Let $\mathcal T'$ be the set of $\tau\in[0,1]$ such that there exist two simplices in $\Theta_\tau$ that intersect in their interiors. Since interiors of simplices are open sets, we know that $\mathcal T'$ is an open subset of $[0,1]$. We will show that $\mathcal T=\mathcal T'$. This will prove that $\mathcal T$ is clopen; since $[0,1]$ is connected and $0\not\in\mathcal T$, it will follow that $\mathcal T=\emptyset$. 

We certainly have $\mathcal T'\subseteq\mathcal T$. Suppose that there exists $\tau\in\mathcal T\setminus\mathcal T'$; we will obtain a contradiction, which implies the desired equality. According to \cref{lem:facet-to-facet}, there exist $i,i'\in[m]$ such that ${\conv(\varphi_\tau(\JJ_i))\cap\conv(\varphi_\tau(\JJ_{i'}))}$ is $(d-1)$-dimensional but is not a common facet of $\conv(\varphi_\tau(\JJ_i))$ and $\conv(\varphi_\tau(\JJ_{i'}))$. Then $\conv(\varphi_\tau(\JJ_i))\cap\conv(\varphi_\tau(\JJ_{i'}))$ is properly contained in some facet $F$ of $\conv(\varphi_\tau(\JJ_i))$. Note that $F$ is not contained in the boundary of $P_\tau$. By \cref{lem:facet_or_boundary}, there exists $j\in[m]$ such that $F=\conv(\varphi_\tau(\JJ_i))\cap\conv(\varphi_\tau(\JJ_j))$. But then $\conv(\varphi_\tau(\JJ_{i'}))$ and $\conv(\varphi_\tau(\JJ_j))$ must intersect in their interior, which contradicts the assumption that $\tau\not\in\mathcal T'$.  
\end{proof}

\section{Triangulating the Permutahedron}\label{sec:triangulating}  

As before, let us fix a finite Coxeter group $W$ acting in its reflection representation on an $r$-dimensional Euclidean space $V$. Fix a standard Coxeter element $c$ of $W$. In this section, we prove that
$\SBDW_\yy(W,c)$
is a triangulation of the permutahedron $\Perm_\yy$ for every point $\yy$ in the base region~$\BB$. 

\subsection{SBDW simplices} 
We begin by verifying that $\SBDW_\yy(W,c)$ satisfies conditions (i) and (iii) of \cref{prop:conditions-for-max-cells-in-reg-sub} for every $\yy\in\BB$ and every height function $\hgt: W\yy \to \RR$. 

\begin{lemma}\label{lem:i_and_iii}
    Let $\nabla \in \SBDW_\yy(W,c)$. Then 
    \begin{itemize}
        \item the polytope $\nabla$ is an $r$-dimensional simplex, and its vertices are vertices of $\Perm_\yy$;
        \item for any height function $\hgt: W \yy \to \RR$, the set $\VV(\nabla)^\hgt$ of lifted vertices is contained in a hyperplane in $V\oplus\RR$. 
    \end{itemize}
\end{lemma}
\begin{proof}
By definition, $\nabla = \conv(w\C\yy)$ for some $(w, \C) \in \Omega(W,c)$.
Say $w\C=\{w_0 \leq_{\hB} \cdots \leq_{\hB} w_{r}\}$, and set $p_i:= w_i w_{i-1}^{-1}$. 

    For the first bullet point, note that $w_{i-1} \yy - w_{i}\yy$ is parallel to the coroot $\beta_{p_i}^\vee$. Since $p_1 p_2 \cdots p_r$ is a reduced $T$-word for the Coxeter element $c$, the coroots $\beta_{p_1}^\vee,\ldots,\beta_{p_r}^\vee$ are linearly independent by \cref{lem:linearly_independent}. This implies that the vectors $w_0 \yy,w_1\yy,\ldots,w_r\yy$ are affinely independent, so their convex hull $\nabla$ is a simplex. The fact that $\VV(\nabla)\subseteq\VV(\Perm_\yy)$ follows from the definition of $\nabla$. 

    The second bullet point is immediate from the fact that $\VV(\nabla)^\hgt=\{(w_i \yy)^\hgt:0\leq i\leq r\}$ is a set of $r+1$ points in the $(r+1)$-dimensional space $V\oplus \RR$. 
\end{proof}

\subsection{SBDW simplices meet facet-to-facet}

In this section, we show that $\SBDW_\yy(W,c)$ satisfies condition (ii) of \cref{prop:conditions-for-max-cells-in-reg-sub}. 

\begin{proposition}\label{prop:facet-matching}
    Let $\nabla \in \SBDW_\yy(W,c)$, and let $F$ be a facet of $\nabla$. Then either $F$ lies on the boundary of $\Perm_\yy$, or there is a unique $\nabla' \in \SBDW_\yy(W,c)$ such that $\nabla \cap \nabla' =F$.
\end{proposition}

\begin{proof}
By \cref{lem:(wC)-to-saturated-chains}, we have $\nabla=\conv(\C\yy)$ for some chain \[w\C =\{w=w_0 \lessdot_{\hB} w_1 \lessdot_{\hB} \cdots \lessdot_{\hB} w_r=wc\}\in \widetilde\Omega(W,c).\]
    The facet $F$ of $\nabla$ has vertex set $w\C\yy \setminus\{w_i\yy\}$ for some $i$. There are three cases: $i=0$, $i=r$, and $1\le i\le r-1$.

\medskip 

\noindent \textbf{Case 1.} Suppose $1 \le i \le r-1$. The Bruhat interval $[w_{i-1}, w_{i+1}]_{\hB}$ has rank $2$, so it is a diamond by \cref{lem:quadrilateral}. Let $p$ and $q$ be the middle elements of this diamond, where $p=w_i$. Exchanging $p$ for $q$ in $w\C$ produces $w\C' \in \widetilde\Omega(W,c)$. Let $\nabla':= \conv(w\C'\yy)$. Note that $\nabla'$ also has $F$ as a facet and that no other simplices in $\SBDW_\yy(W,c)$ have $F$ as a facet. We would like to show that $\nabla \cap \nabla' =F$. Let $H$ be the affine hull of $F$, which is a hyperplane. It suffices to show that $p$ and $q$ lie on opposite sides of $H$. 

Consider the Bruhat interval polytope $\BIP_{w_{i-1},w_{i+1}}=\conv([w_{i-1}, w_{i+1}]_{\hB}\yy)$, whose vertex set is exactly $[w_{i-1}, w_{i+1}]_{\hB}\yy$. It follows from \cref{lem:quadrilateral} that $\BIP_{w_{i-1},w_{i+1}}$ is a quadrilateral whose edges correspond to the four cover relations in $[w_{i-1}, w_{i+1}]_{\hB}$. The intersection $\BIP_{w_{i-1},w_{i+1}} \cap H$ is the line segment between $w_{i-1} \yy$ and $w_{i+1}\yy$, which is not an edge of $\BIP_{w_{i-1},w_{i+1}}$. In $\BIP_{w_{i-1},w_{i+1}}$, the vertices $p\yy$ and $q\yy$ lie on either side of this line segment, so they are on opposite sides of $H$. 

\medskip 

\noindent \textbf{Case 2.} Suppose $i=0$. Let $p=w_1$. We have $p=t_1 w = w t_2$, where $t_1$ is a left inversion of $p$ and $t_2$ is a right inversion of $p$. 

We first show that if $\ell_S(pc) < \ell_S(p) + r$, then $F$ lies on the boundary of $\Perm_\yy$. Choose reduced $S$-words ${\sf w}$ and ${\sf c}$ for $w$ and $c$. If $\ell_S(pc) < \ell_S(p) + r$, then $pc = t_1 wc <_{\hB} wc$. This implies that a (not necessarily reduced) word for $pc$ can be obtained from ${\sf w}{\sf c}$ by removing a letter. This letter cannot be in $\w$ or we would have $p=t_1w <_{\hB}w$. So this letter is in ${\sf c}$, and we have $t_2 c <_{\hB} c$. This implies that $p$ and $wc=pt_2 c$ differ by an element of a maximal parabolic subgroup $W_{\langle s\rangle}$ of $W$; in other words $pW_{\langle s\rangle} = wc W_{\langle s\rangle}$. It follows from \cite[Proposition~2.5.1]{BjornerBrenti} that every element of the interval $[p,wc]_{\hB}$ belongs to the same left coset of $W_{\langle s\rangle}$ as $p$ and $wc$. In particular, $w_jW_{\langle s\rangle} = pW_{\langle s\rangle}$ for all $2\leq j\leq r$. Thus, all vertices of $F$ lie on the facet of $\Perm_\yy$ with vertex set $\{u\yy: u \in pW_{\langle s\rangle}\}$.

Now, assume that $\ell_S(pc) = \ell_S(p) + r$. Let $\C' \in \widetilde\Omega(W,c)$ be the chain $p=w_1 \lessdot_{\hB} \cdots \lessdot_{\hB} w_r=wc \lessdot pc$, and let $\nabla':= \conv(\C'\yy)$. Note that $F$ is a facet of $\nabla'$ and that no other simplices in $\SBDW_\yy(W,c)$ have $F$ as a facet. Let $H$ be the affine hull of $F$. We will show that $w$ and $pc$ lie on opposite sides of $H$, which will imply that $\nabla \cap \nabla' = F$. 

Let $\beta=\beta_{t_1}$ be the positive root corresponding to $t_1$. Choose a normal vector $\gamma \in V^*$ of $H$ that satisfies $\langle \gamma, \beta^\vee \rangle >0$. We have 
\[H=\{\xx\in V:\langle \gamma, \xx - p\yy \rangle = \langle \gamma, \xx - wc\yy \rangle=0\}.\]
Since $w=t_1p$, we have 
\[w\yy= p\yy - 2\frac{\langle \beta, p\yy \rangle}{\langle \beta, \beta^\vee  \rangle} \beta^\vee, \quad \text{so} \quad w\yy -p\yy = -2 \frac{\langle \beta, p\yy \rangle}{\langle \beta, \beta^\vee  \rangle} \beta^\vee. \]
Since $t_1p<_{\hB}p$, we have $\langle \beta, p\yy \rangle <0$.
Similarly, since $pc = t_1 wc$, we have
\[pc\yy= wc\yy - 2\frac{\langle \beta, wc\yy  \rangle}{\langle \beta, \beta^\vee \rangle} \beta^\vee, \quad \text{so} \quad pc\yy -wc\yy = -2 \frac{\langle \beta, wc\yy \rangle}{\langle \beta, \beta^\vee \rangle} \beta^\vee. \]
Since $wc<_{\hB}t_1wc$, we have $\langle \beta, wc\yy \rangle >0$. We deduce that 
\[\langle \gamma, w\yy -p\yy \rangle = -2 \frac{  \langle \beta,p\yy \rangle}{\langle \beta, \beta^\vee \rangle} \langle \gamma, \beta^\vee \rangle >0 \quad \text{and} \quad \langle \gamma, pc\yy -wc\yy \rangle = -2 \frac{\langle \beta, wc\yy \rangle}{\langle \beta, \beta^\vee \rangle} \langle  \gamma, \beta^\vee\rangle <0,\]
    which shows that $w\yy$ and $pc\yy$ are on opposite sides of $H$, as desired.

\medskip 

\noindent \textbf{Case 3.} Suppose $i=r$. Let $q= w_{r-1}$. If $\ell_S(qc^{-1}) = \ell_S(q) - r$, then the argument in Case~2 (taking $qc^{-1}$ in place of $w$ and $w$ in place of $p$) shows that there is a unique simplex $\nabla' \neq \nabla$ in $\SBDW_\yy(W,c)$ with $F$ as a facet and that $\nabla \cap \nabla'=F$. It remains to show that if $\ell_S(qc^{-1}) \neq \ell_S(q) - r$, then $F$ lies on a facet of $\Perm_\yy$. The argument is quite similar to Case~2. Again, choose reduced $S$-words $\w$ and ${\sf c}$ for $w$ and $c$. Since $q \lessdot_{\hB} wc$, a reduced $S$-word for $q$ may be obtained by removing one letter from ${\sf w}{\sf c}$. The assumption that $\ell(qc^{-1}) = \ell(q) - r$ ensures that this letter cannot be from $\w$, so it must be from $\sf c$. Thus, $q$ and $w$ differ by an element of a maximal parabolic subgroup $W_{\langle s\rangle}$. An argument identical to the one in Case~2 shows that $wW_{\langle s\rangle}=w_jW_{\langle s\rangle}=qW_{\langle s\rangle}$ for all $1\leq j\leq r-1$, so $F$ lies on a facet of $\Perm_\yy$.
\end{proof}

\subsection{A regular triangulation in simply-laced type}\label{subsec:regular} 
Assume in this subsection that $W$ is a simply-laced Coxeter group of rank $r$. Let $c$ be a standard Coxeter element of $W$, and let $\Q^c$ be the corresponding Dynkin quiver of type $W$. 

\begin{definition}\label{def:height}
For $\yy\in\BB$, $\gamma\in V$, and $\epsilon\geq 0$, define a height function $\hgt_{\yy,\gamma}^\epsilon\colon W\yy\to\RR$ on the vertex set of $\Perm_\yy$ by 
\[\hgt_{\yy,\gamma}^\epsilon(w\yy)=\langle\gamma,w^{-1}\yy\rangle-\epsilon\, 2^{\ell_S(w)}.\] 
\end{definition}

Recall that for $\yy\in \BB$ and $\gamma\in V^*$, there is a linear stability function $Z_{\yy,\gamma}\colon V^*\to\CC$, which has the associated slope function $\mu_{\yy,\gamma}\colon\Phi^+\to\RR$ given by \[\mu_{\yy,\gamma}(\beta)=\frac{\langle \gamma, \beta\rangle}{\langle\beta^\vee,\yy\rangle}.\]  
According to \cref{prop:Chang}, there exist $\yys\in V$ and $\gamma^\star\in V^*$ such that $Z_{\yys,\gamma}$ is totally stable.

Our goal in this subsection is to prove the following result, which tells us that there exists some $\yys\in\BB$ such that $\SBDW_\yys(W,c)$ is a regular triangulation of $\Perm_{\yys}$. 

\begin{proposition}\label{prop:regular_simply_laced}
Let $W$ be a simply-laced Coxeter group, and let $c$ be a standard Coxeter element of $W$. Fix $\yys\in\BB$ and $\gamma^\star\in V^*$ such that the stability function $Z_{\yys,\gamma^\star}$ is totally stable. Then $\SBDW_{\yys}(W,c)$ is a regular triangulation of $\Perm_{\yys}$, induced by $\hgt_{\yys,\gamma^\star}^\epsilon$ when $\epsilon>0$ is sufficiently small. 
\end{proposition}

Throughout the rest of this subsection, we fix $\yys\in\BB$ and $\gamma^\star\in V^*$ such that $Z_{\yys,\gamma^\star}$ is totally stable. 
Our strategy to prove \cref{prop:regular_simply_laced} is to employ \cref{prop:conditions-for-max-cells-in-reg-sub}. We have already checked (in \cref{lem:i_and_iii,prop:facet-matching}) that $\SBDW_{\yys}(W,c)$ satisfies conditions (i), (ii), and (iii) in \cref{prop:conditions-for-max-cells-in-reg-sub}, so we just need to check the local folding condition (iv).  

To ease notation, let us write $\hgt^\epsilon$ instead of $\hgt_{\yys,\gamma^{\star}}^\epsilon$. Fix two simplices $\nabla$ and $\nabla'$ in $\SBDW_{\yys}(W,c)$ that intersect along a common facet. We have $\VV(\nabla)=u_0\C\yys$ and $\VV(\nabla')=u_0'\C'\yys$ for some Bruhat order chains \[u_0\C=\{u_0\lessdot_{\hB} u_1\lessdot_{\hB}\cdots\lessdot_{\hB} u_r\}\quad\text{and}\quad u_0'\C'=\{u_0'\lessdot_{\hB} u_1'\lessdot_{\hB}\cdots\lessdot_{\hB} u_r'\}.\] We assume without loss of generality that $\ell_S(u_0)\leq\ell_S(u_0')$. Let $u_i'$ be the unique element of $u_0'\C'\setminus u_0\C$ so that $u_i'\yys$ is the vertex of $\nabla'$ that is not in $\nabla$. Our aim is to show that the lifted vertex $(u_i'\yys)^{\hgt^{\epsilon}}$ lies strictly above the hyperplane $H_\nabla$ in $V\oplus\RR$ that contains $\VV(\nabla)^{\hgt^\epsilon}$. We will split the argument into two cases based on whether or not $u_0$ and $u_0'$ are equal. 

Let us first record a couple of observations.
There exist $\lambda\in V^*$ and $b,\xi\in\RR$ with $b>0$ such that 
\[H_\nabla=\{(\vvv,a)\in V\oplus\RR:\langle\lambda,\vvv\rangle+ba=\xi\}\] (note that $\lambda$, $b$, and $\xi$ depend on the parameter $\epsilon$). 
Since $\VV(\nabla)^{\hgt^\epsilon}\subseteq H_\nabla$, we have \[\langle\lambda,\vvv-\vvv'\rangle+b(\hgt^\epsilon(\vvv)-\hgt^\epsilon(\vvv'))=0\] for all $\vvv,\vvv'\in\VV(\nabla)$. If $u\in W$ and $t\in T$, then
\begin{align}
\nonumber \langle \lambda,u\yys-tu\yys\rangle+b(\hgt^\epsilon(u\yys)-\hgt^\epsilon(tu\yys))&=\langle \lambda,\langle\beta_t,u\yys\rangle\beta_t^\vee\rangle+b(\hgt^\epsilon(u\yys)-\hgt^\epsilon(tu\yys)) \\ 
\label{eq:utu1}&=\langle\beta_t,u\yys\rangle\langle \lambda,\beta_t^\vee\rangle+b(\hgt^\epsilon(u\yys)-\hgt^\epsilon(tu\yys)).
\end{align}
If we assume that $u$ and $tu$ are both in the chain $\C$, then it follows that 
\begin{equation}\label{eq:utu2}
\langle\beta_t,u\yys\rangle\langle \lambda,\beta_t^\vee\rangle+b(\hgt^\epsilon(u\yys)-\hgt^\epsilon(tu\yys))=0. 
\end{equation}

\begin{lemma}\label{lem:regular_simply_laced_1}
If $u_0\neq u_0'$, then for all sufficiently small $\epsilon>0$, the lifted vertex $(u_i'\yys)^{\hgt^\epsilon}$ lies strictly above the hyperplane $H_\nabla$. 
\end{lemma}
\begin{proof}
Suppose that $u_0\neq u_0'$, and choose $\epsilon>0$ sufficiently small. Using our assumption that ${\ell_S(u_0)\leq\ell_S(u_0')}$, we find that $u_j'=u_{j+1}$ for all $0\leq j\leq r-1$ and that $i=r$. We have $u_r'=u_1c$ and $u_r=u_0c$, so 
$u_r'u_r^{-1}=u_r'c^{-1}u_0^{-1}=u_1u_0^{-1}$. Let \[t=u_r'u_r^{-1}=u_1u_0^{-1}\in T.\] We want to show that $\langle\lambda,u_r'\yys\rangle+b\,\hgt^\epsilon(u_r'\yys)>\xi$. Since $\xi=\langle\lambda,u_r\yys\rangle+b\,\hgt^\epsilon(u_r\yys)$, we need to show that 
\[\langle\lambda,u_r'\yys-u_r\yys\rangle+b(\hgt^\epsilon(u_r'\yys)-\hgt^\epsilon(u_r\yys))>0.\] 
Using \eqref{eq:utu1}, we can rewrite this desired inequality as 
\[
\langle\beta_t,u_r'\yys\rangle\langle\lambda,\beta_t^\vee\rangle +b(\hgt^\epsilon(u_r'\yys)-\hgt^\epsilon(u_r\yys))>0.\]
Recall that $b>0$. We also have $\langle\beta_t,u_r'\yys\rangle<0$ because $t$ is a left inversion of $u_r'$. Therefore, we wish to show that 
\[\frac{\hgt^\epsilon(u_r\yys)-\hgt^\epsilon(u_r'\yys)}{\langle\beta_t,u_r'\yys\rangle}>\frac{\langle\lambda,\beta_t^\vee\rangle}{b}.\] 
Because $u_1$ and $u_0=tu_1$ are both in $\C$, we can apply \eqref{eq:utu2} to find that \[\frac{\langle\lambda,\beta_t^\vee\rangle}{b}=\frac{\hgt^\epsilon(u_0\yys)-\hgt^\epsilon(u_1\yys)}{\langle\beta_t,u_1\yys\rangle}.\] 
Hence, we need to show that 
\begin{equation}\label{eq:long_chain_1}
\frac{\hgt^\epsilon(u_r\yys)-\hgt^\epsilon(u_r'\yys)}{\langle\beta_t,u_r'\yys\rangle}>\frac{\hgt^\epsilon(u_0\yys)-\hgt^\epsilon(u_1\yys)}{\langle\beta_t,u_1\yys\rangle}.
\end{equation}

Using the definition of the height function $\hgt^\epsilon=\hgt_{\yys,\gamma^\star}^\epsilon$ from \cref{def:height}, we compute that 
\begin{align*}
\hgt^\epsilon(u_r\yys)-\hgt^\epsilon(u_r'\yys)&=\langle\gamma^\star,u_r^{-1}\yys-(u_r')^{-1}\yys\rangle-\epsilon(2^{\ell_S(u_r)}-2^{\ell_S(u_r')}) \\ 
&=\langle\gamma^\star,(u_r')^{-1}(t\yys-\yys)\rangle-\epsilon(2^{\ell_S(u_r)}-2^{\ell_S(u_r')}) \\ 
&=\langle\gamma^\star,(u_r')^{-1}(-\langle \beta_t,\yys\rangle\beta_t^\vee)\rangle-\epsilon(2^{\ell_S(u_r)}-2^{\ell_S(u_r')}) \\ 
&=-\langle\beta_t,\yys\rangle\langle\gamma^\star, (u_r')^{-1}\beta_t^\vee\rangle-\epsilon(2^{\ell_S(u_r)}-2^{\ell_S(u_r')}). 
\end{align*}
Similarly, 
\[\hgt^\epsilon(u_0\yys)-\hgt^\epsilon(u_1\yys)=-\langle\beta_t,\yys\rangle\langle\gamma^\star,u_1^{-1}\beta_t^\vee\rangle-\epsilon(2^{\ell_S(u_0)}-2^{\ell_S(u_1)}).\]
We will prove that 
\begin{equation}\label{eq:long_chain_2}
\frac{-\langle\beta_t,\yys\rangle\langle\gamma^\star, (u_r')^{-1}\beta_t^\vee\rangle}{\langle\beta_t,u_r'\yys\rangle}>\frac{-\langle\beta_t,\yys\rangle\langle\gamma^\star, u_1^{-1}\beta_t^\vee\rangle}{\langle\beta_t,u_1\yys\rangle};
\end{equation}
this will imply that \eqref{eq:long_chain_1} holds so long as $\epsilon$ is sufficiently small. 

We have $(u_r')^{-1}\beta_t^\vee=-u_r^{-1}\beta_t^\vee=-(\beta_tu_r)^\vee$ and $u_1^{-1}\beta_t^\vee=-u_0^{-1}\beta_t^\vee=-(\beta_tu_0)^\vee$. In addition, we have $\langle\beta_t,u_r'\yys\rangle=\langle\beta_tu_r',\yys\rangle=-\langle\beta_tu_r,\yys\rangle$ and $\langle\beta_t,u_1\yys\rangle=\langle\beta_tu_1,\yys\rangle=-\langle\beta_tu_0,\yys\rangle$. Since $\yys\in\BB$, we know that $\langle\beta_t,\yys\rangle>0$, so \eqref{eq:long_chain_2} is equivalent to the inequality 
\[\frac{\langle\gamma^\star, (\beta_tu_r)^\vee\rangle}{\langle\beta_tu_r,\yys\rangle}<\frac{\langle\gamma^\star, (\beta_tu_0)^\vee\rangle}{\langle\beta_tu_0,\yys\rangle}.\] In other words, we must show that \begin{equation}\label{eq:utu3}
\mu_{\yys,\gamma^\star}(\beta_tu_r)<\mu_{\yys,\gamma^\star}(\beta_tu_0).
\end{equation} 
Since $\beta_t u_0$ and $\beta_t u_r$ are positive roots and $\beta_tu_0=\beta_t u_rc^{-1}$, \cref{lem:AR_translation} implies that there is a path from $\beta_t u_r$ to $\beta_t u_0$ in the Auslander--Reiten quiver for $\Q^c$, so \eqref{eq:utu3} follows from the hypothesis that the stability function $Z_{\yys,\gamma^\star}$ is totally stable.  
\end{proof}

Assume now that $u_0=u_0'$. This implies that $u_r=u_0c=u_0'c=u_r'$. Hence, $1\leq i\leq r-1$, and $u_j=u_j'$ for all $0\leq j\leq r$ with $j\neq i$. The Bruhat order interval $[u_{i-1},u_{i+1}]_{\hB}$ is a diamond consisting of the elements $u_{i-1},u_i.u_i',u_{i+1}$. We will consider two cases based on whether or not the reflections $u_iu_{i-1}^{-1}$ and $u_{i+1}u_{i}^{-1}$ commute. 

\begin{lemma}\label{lem:regular_simply_laced_2}
If $u_0=u_0'$ and the reflections $u_iu_{i-1}^{-1}$ and $u_{i+1}u_{i}^{-1}$ commute, then for all $\epsilon>0$, the lifted vertex $(u_i\yys)^{\hgt^\epsilon}$ lies strictly above the hyperplane $H_\nabla$. 
\end{lemma}
\begin{proof}
Suppose that $u_0=u_0'$ and that the reflections $u_iu_{i-1}^{-1}$ and $u_{i+1}u_{i}^{-1}$ commute. Let $H_1$ and $H_2$ be the hyperplanes in $\HH$ such that $u_iu_{i-1}^{-1}=t_{H_1}$ and $u_{i+1}u_{i}^{-1}=t_{H_2}$. The normal vectors of $H_1$ and $H_2$ are orthogonal, and we have $u_{i}'=t_{H_2}u_{i-1}$ and $u_{i+1}=t_{H_1}u_i'$. It follows that the points $u_{i-1}\yys,u_i\yys,u_{i+1}\yys,u_i'\yys$ are the vertices of a rectangle,  so 
\begin{equation}\label{eq:rectangle}
u_{i}'\yys=u_{i-1}\yys-u_i\yys+u_{i+1}\yys. 
\end{equation} 
The map $W\to W$ given by $w\mapsto w^{-1}$ is an automorphism of the Bruhat order, so the elements $u_{i-1}^{-1},u_i^{-1},(u_i')^{-1},u_{i+1}^{-1}$ also form a diamond interval in the Bruhat order. Because $t_{H_1}$ and $t_{H_2}$ commute, the reflections $u_i^{-1}u_{i-1}=u_{i}^{-1}t_{H_1}u_{i}$ and $u_{i+1}^{-1}u_{i}=u_{i}^{-1}t_{H_2}u_{i}$ commute, so the same argument from above shows that the points $u_{i-1}^{-1}\yys,u_i^{-1}\yys,u_{i+1}^{-1}\yys,(u_i')^{-1}\yys$ are the vertices of a rectangle. Consequently, 
\begin{equation}\label{eq:rectangle2}
(u_{i}')^{-1}\yys=u_{i-1}^{-1}\yys-u_i^{-1}\yys+u_{i+1}^{-1}\yys. 
\end{equation} 

We want to show that $\langle\lambda,u_i'\yys\rangle+b\,\hgt^\epsilon(u_i'\yys)>\xi$. By \cref{def:height}, we have \[\langle\lambda,u_i'\yys\rangle+b\,\hgt^\epsilon(u_i'\yys)=\langle\lambda,u_i'\yys\rangle+b\langle\gamma^*,(u_i')^{-1}\yys\rangle-\epsilon 2^{\ell_S(u_i')}.\] Using \eqref{eq:rectangle}, \eqref{eq:rectangle2}, \cref{def:height}, and the fact that $u_{i-1},u_i,u_{i+1}\in\C$, we can rewrite this expression as  
\begin{align*}
&\phantom{{}={}}(\langle\lambda,u_{i-1}\yys\rangle+b\langle\gamma^\star,u_{i-1}^{-1}\yys\rangle)\!-\!(\langle\lambda,u_{i}\yys\rangle+b\langle\gamma^\star,u_{i}^{-1}\yys\rangle)\!+\!(\langle\lambda,u_{i+1}\yys\rangle+b\langle\gamma^\star,u_{i+1}^{-1}\yys\rangle)-b\epsilon 2^{\ell_S(u_i')} \\ 
&=(\xi+b\epsilon 2^{\ell_S(u_{i-1})})-(\xi+b\epsilon 2^{\ell_S(u_i)})+(\xi+b\epsilon 2^{\ell_S(u_{i+1})})-b\epsilon 2^{\ell_S(u_i')} \\ 
&=\xi +b\epsilon (2^{\ell_S(u_{i-1})}+2^{\ell_S(u_{i+1})}-2^{\ell_S(u_i)}-2^{\ell_S(u_i')}) \\ 
&=\xi +b\epsilon 2^{\ell_S(u_i)-1}(1+2^2-2^1-2^1) \\ 
&=\xi +b\epsilon 2^{\ell_S(u_i)-1}, 
\end{align*}
and this is clearly greater than $\xi$. 
\end{proof}

\begin{lemma}\label{lem:regular_simply_laced_3}
If $u_0=u_0'$ and the reflections $u_iu_{i-1}^{-1}$ and $u_{i+1}u_{i}^{-1}$ do not commute, then for all sufficiently small $\epsilon>0$, the lifted vertex $(u_i\yys)^{\hgt^\epsilon}$ lies strictly above the hyperplane $H_\nabla$. 
\end{lemma}
\begin{proof}
Suppose that $u_0=u_0'$ and that the reflections $t=u_iu_{i-1}^{-1}$ and $t'=u_{i+1}u_{i}^{-1}$ do not commute. Let $t''=u_i'u_{i-1}^{-1}$ and $t'''=u_{i+1}(u_{i}')^{-1}$. Let $W'$ be the subgroup of $W$ generated by $t,t',t'',t'''$. According to \cref{lem:quadrilateral}, $W'$ is a parabolic subgroup of $W$ of rank $2$; since $W$ is simply-laced, $W'$ is isomorphic to $A_1\times A_1$ or $A_2$. Because $t$ and $t'$ do not commute, $W'$ cannot be isomorphic to the abelian group $A_1\times A_1$, so it must be isomorphic to $A_2$. This implies that $W'$ has $3$ reflections, which forces either $t=t'''$ or $t'=t''$. We will assume that $t=t'''$; the other case is similar (in the other case, our same argument applies if one switches the roles of $\nabla$ and $\nabla'$). 

We want to show that $\langle\lambda,u_i'\yys\rangle+b\, \hgt^\epsilon(u_i'\yys)>\xi$. Since $\xi=\langle\lambda,u_{i+1}\yys\rangle+b\,\hgt^\epsilon(u_{i+1}\yys)$, we need to show that 
\[\langle\lambda,u_i'\yys-u_{i+1}\yys\rangle+b(\hgt^\epsilon(u_i'\yys)-\hgt^\epsilon(u_{i+1}\yys))>0.\] 
Using \eqref{eq:utu1}, we can rewrite this desired inequality as 
\[
\langle\beta_{t},u_{i}'\yys\rangle\langle\lambda,\beta_{t}^\vee\rangle +b(\hgt^\epsilon(u_i'\yys)-\hgt^\epsilon(u_{i+1}\yys))>0.\]
Recall that $b>0$. We also have $\langle\beta_{t},u_i'\yys\rangle>0$ because $t$ is not a left inversion of $u_i'$. Therefore, we wish to show that 
\[\frac{\hgt^\epsilon(u_{i+1}\yys)-\hgt^\epsilon(u_i'\yys)}{\langle\beta_{t},u_i'\yys\rangle}<\frac{\langle\lambda,\beta_{t}^\vee\rangle}{b}.\] 
Because $u_i$ and $u_{i-1}=tu_i$ are both in $\C$, we can apply \eqref{eq:utu2} to find that \[\frac{\langle\lambda,\beta_{t}^\vee\rangle}{b}=\frac{\hgt^\epsilon(u_{i-1}\yys)-\hgt^\epsilon(u_i\yys)}{\langle\beta_{t},u_i\yys\rangle}.\] 
Hence, we need to show that 
\begin{equation}\label{eq:diamond_1}
\frac{\hgt^\epsilon(u_{i+1}\yys)-\hgt^\epsilon(u_i'\yys)}{\langle\beta_{t},u_i'\yys\rangle}<\frac{\hgt^\epsilon(u_{i-1}\yys)-\hgt^\epsilon(u_i\yys)}{\langle\beta_t,u_i\yys\rangle}.
\end{equation}

Using the definition of the height function $\hgt^\epsilon=\hgt_{\yys,\gamma^\star}^\epsilon$ from \cref{def:height}, we compute that 
\begin{align*}
\hgt^\epsilon(u_{i+1}\yys)-\hgt^\epsilon(u_i'\yys)&=\langle\gamma^\star,u_{i+1}^{-1}\yys-(u_i')^{-1}\yys\rangle-\epsilon(2^{\ell_S(u_{i+1})}-2^{\ell_S(u_i')}) \\ 
&=\langle\gamma^\star,(u_i')^{-1}(t\yys-\yys)\rangle-\epsilon(2^{\ell_S(u_{i+1})}-2^{\ell_S(u_i')}) \\ 
&=\langle\gamma^\star,(u_i')^{-1}(-\langle \beta_{t},\yys\rangle\beta_{t}^\vee)\rangle-\epsilon(2^{\ell_S(u_{i+1})}-2^{\ell_S(u_i')}) \\ 
&=-\langle\beta_{t},\yys\rangle\langle\gamma^\star, (u_i')^{-1}\beta_{t}^\vee\rangle-\epsilon(2^{\ell_S(u_{i+1})}-2^{\ell_S(u_i')}). 
\end{align*}
Similarly, 
\[\hgt^\epsilon(u_{i-1}\yys)-\hgt^\epsilon(u_i\yys)=-\langle\beta_{t},\yys\rangle\langle\gamma^\star,u_i^{-1}\beta_{t}^\vee\rangle-\epsilon(2^{\ell_S(u_{i-1})}-2^{\ell_S(u_i)}).\] 
We will prove that 
\begin{equation}\label{eq:diamond_2}
\frac{-\langle\beta_{t},\yys\rangle\langle\gamma^\star, (u_i')^{-1}\beta_{t}^\vee\rangle}{\langle\beta_{t},u_i'\yys\rangle}<\frac{-\langle\beta_{t},\yys\rangle\langle\gamma^\star, u_i^{-1}\beta_{t}^\vee\rangle}{\langle\beta_{t},u_i\yys\rangle};
\end{equation}
this will imply that \eqref{eq:diamond_1} holds so long as $\epsilon$ is sufficiently small. 

We have $(u_i')^{-1}\beta_{t}^\vee=(\beta_{t}u_i')^\vee$ and $u_i^{-1}\beta_{t}^\vee=-u_{i-1}^{-1}\beta_{t}^\vee=-(\beta_{t}u_{i-1})^\vee$. In addition, we have $\langle\beta_{t},u_i'\yys\rangle=\langle\beta_{t}u_i',\yys\rangle$ and $\langle\beta_{t},u_{i}\yys\rangle=\langle\beta_{t}u_i,\yys\rangle=-\langle\beta_{t}u_{i-1},\yys\rangle$. Since $\yys\in\BB$, we know that $\langle\beta_{t},\yys\rangle>0$, so \eqref{eq:diamond_2} is equivalent to the inequality 
\[\frac{\langle\gamma^\star, (\beta_{t}u_i')^\vee\rangle}{\langle\beta_{t}u_i',\yys\rangle}>\frac{\langle\gamma^\star, (\beta_{t}u_{i-1})^\vee\rangle}{\langle\beta_{t}u_{i-1},\yys\rangle}.\] In other words, we must show that \begin{equation}\label{eq:diamond_3}
\mu_{\yys,\gamma^\star}(\beta_{t}u_i')>\mu_{\yys,\gamma^\star}(\beta_{t}u_{i-1}).
\end{equation} 
Note that $\beta_{t} u_i'$ and $\beta_{t} u_{i-1}$ are positive roots. Because the stability function $Z_{\yys,\gamma^\star}$ is totally stable, the inequality \eqref{eq:diamond_3} will follow if we can show that there is a path from $\beta_{t}u_{i-1}$ to $\beta_{t}u_i'$ in the Auslander--Reiten quiver for $\Q^c$. 

Consider the reflections 
\[t_1=u_{i-1}^{-1}u_i,\quad t_2=u_i^{-1}u_{i+1},\quad t_3=u_{i-1}^{-1}u_i',\quad t_4=(u_i')^{-1}u_{i+1}.\] We have \[
\beta_t u_{i-1}=\beta_{t_1}\quad \text{and} \quad \beta_tu_i'=\beta_{t_4}.  
\] Thus, we must show that $t_1\prec_c t_4$. 
It is straightforward to compute that 
\[t_1=u_{i-1}^{-1}tu_{i-1},\quad t_2=t_3=u_{i-1}^{-1}t''u_{i-1},\quad t_4=u_{i-1}^{-1}t'u_{i-1}.\] 
Hence, the subgroup of $W$ generated by the reflections $t_1$, $t_2=t_3$, and $t_4$ is $u_{i-1}^{-1}W'u_{i-1}$, which is isomorphic to $A_2$. Because $u_0\leqB u_{i-1}\leqB u_{i+1}\leqB u_{0}c$, it follows from \cref{lem:basification} that the set of canonical generators of $u_{i-1}^{-1}W'u_{i-1}$ is either $\{t_1,t_2\}$ or $\{t_2,t_4\}$. In either case, $t_2$ is a canonical generator. Because $u_0\C$ and $u_0'\C'$ are in $\widetilde\Omega(W,c)$, the expressions $t_1t_2$ and $t_2t_4$ appear in reduced $T$-words for $c$. By \cref{lem:i-1modk} (and the fact that $t_2$ is a canonical generator of $u_{i-1}^{-1}W'u_{i+1}$), we must have $t_2\prec_c t_1\prec_c t_4$ or $t_1\prec_c t_4\prec_c t_2$. In either case, $t_1\prec_c t_4$, as desired.  
\end{proof}

Put together, \cref{lem:regular_simply_laced_1,lem:regular_simply_laced_2,lem:regular_simply_laced_3} complete the proof of \cref{prop:regular_simply_laced}. 

\subsection{The SBDW triangulation in simply-laced type} 

We are now prepared to prove the following proposition, which is a restatement of \cref{thm:triangulation} in the case when $W$ is simply-laced. 

\begin{proposition}\label{prop:triangulation_simply_laced} 
Let $W$ be a finite simply-laced Coxeter group. For each Coxeter element $c$ of $W$ and each point $\yy\in\BB$, the collection $\SBDW_\yy(W,c)$ is a triangulation of the $W$-permutahedron $\Perm_\yy$. 
\end{proposition} 
\begin{proof}
We will invoke \cref{prop:deform}. Fix a standard Coxeter element $c$ and a point $\yy\in\BB$. We know by \cref{prop:Chang,prop:regular_simply_laced} that there exists $\yys\in\BB$ such that $\SBDW_{\yys}(W,c)$ is a triangulation of $\Perm_{\yys}$. For $\tau\in[0,1]$, let $\yy_\tau=\tau\yy+(1-\tau)\yys$, and let $P_\tau=\Perm_{\yy_\tau}$. Note that $\VV(P_\tau)=W\yy_\tau$. In particular, $P_0=\Perm_{\yys}$, so $\VV(P_0)=W\yys$. Define $\varphi_\tau\colon\VV(P_0)\to\VV(P_\tau)$ by $\varphi_\tau(w\yys)=w\yy_\tau$. This collection of maps satisfies the three bulleted items in the statement of \cref{prop:deform}. Now let $\JJ_1,\ldots,\JJ_m\subseteq\VV(P_0)$ be the vertex sets of the simplices in $\SBDW_{\yys}(W,c)$. For $\tau\in[0,1]$, we have \[\{\conv(\varphi_\tau(\JJ_1)),\ldots,\varphi_\tau(\JJ_m)\}=\SBDW_{\yy_\tau}(W,c),\] and we know that the polytopes in $\SBDW_{\yy_\tau}(W,c)$ are $r$-dimensional simplices by \cref{lem:i_and_iii}. Therefore, it follows from \cref{prop:deform} that $\SBDW_{\yy_\tau}(W,c)$ is a triangulation of $\Perm_{\yy_\tau}$ for every $\tau\in[0,1]$. This proves the desired result since $\yy_1=\yy$.  
\end{proof}

\subsection{Folding} \label{sec:folding}

To prove \cref{thm:triangulation} for arbitrary finite Coxeter groups, we will employ a common technique known as \emph{folding} \cite{Stembridge}. Let us start by briefly discussing the general setup of this technique. 

Suppose $(W,S)$ and $(W^{\fold},S^{\fold})$ are finite Coxeter systems. Suppose there is a surjective map $\ff\colon S\to S^{\fold}$ such that for each $s^{\fold}\in S^{\fold}$, the simple reflections in $\ff^{-1}(s^{\fold})$ commute with each other. Let us also assume that there is an injective group homomorphism $\g\colon W^{\fold}\to W$ such that $\g(s^{\fold})=\prod_{s\in \ff^{-1}(s^{\fold})}s$ for all $s^{\fold}\in S^{\fold}$. We say $(W^{\fold},S^{\fold})$ is a \dfn{folding} of $(W,S)$. We will identify $W^{\fold}$ with its image under $\g$, which is a subgroup of $W$. 

As before, let $W$ act via its reflection representation on $V$. Let $V^{\fold}$ be the orthogonal complement of the subspace of $V$ that is fixed pointwise by all elements of $W^{\fold}$. Then $W^{\fold}$ acts via its reflection representation on $V^{\fold}$, and its Coxeter arrangement is $\HH^{\fold}=\{H\cap V^{\fold}:H\in\HH\}$. Let $\BB^{\fold}$ be the region of $\HH^{\fold}$ that intersects the base region $\BB$ of $\HH$. 

Every finite Coxeter system can be obtained as a folding of a simply-laced Coxeter system, so the following lemma clinches the proof of \cref{thm:triangulation}. 

\begin{lemma}
Suppose $W$ is simply-laced. Let $c$ be a standard Coxeter element of $W^{\fold}$, and let $\yy\in\BB^{\fold}$. Then $\SBDW_\yy(W^{\fold},c)$ is a triangulation of the $W^{\fold}$-permutahedron $\Perm_\yy^{\fold}=\conv(W^{\fold}\yy)\subseteq V^{\fold}$. 
\end{lemma}

\begin{proof}
First, it follows from \cref{lem:union-covers-real,prop:facet-matching} that $\bigcup\SBDW_\yy(W^{\fold},c)=\Perm_{\yy}^{\fold}$. Hence, we just need to show that any two simplices in $\SBDW_\yy(W^{\fold},c)$ intersect along a common facet. Let $\nabla_1^{\fold},\nabla_2^{\fold}\in\SBDW_{\yy}(W^{\fold},c)$. For $i\in\{1,2\}$, we can write $\nabla_i^{\fold}=\conv(u_i\C_i^{\fold}\yy)$ for some $u_i\C_i^{\fold}\in\widetilde\Omega(W^{\fold},c)$. Note that $c$ is also a standard Coxeter element of $W$ (since each element of $S^{\fold}$ is identified via $\g$ with a product of elements of $S$). This implies that each chain $u_i\C_i^{\fold}$ is contained in a chain $v_i\C_i\in\widetilde\Omega(W,c)$. Then $\nabla_1=\conv(v_1\C_1\yy)$ and $\nabla_2=\conv(v_2\C_2\yy)$ are simplices in $\SBDW_\yy(W,c)$. Because $W$ is simply-laced, we know by \cref{prop:triangulation_simply_laced} that $\nabla_1$ and $\nabla_2$ intersect along a common face $F$. The vertex set of $F$ is $v_1\C_1\yy\cap v_2\C_2\yy$. 

Let $r^{\fold}$ be the rank of $W^{\fold}$. Then $\nabla_1^{\fold}$ is the convex hull of $r^{\fold}+1$ vertices of $\nabla_1$, so it is an $r^{\fold}$-dimensional face of $\nabla_1$. Since $V^{\fold}$ is an $r^{\fold}$-dimensional space whose intersection with $\nabla_1$ contains $\nabla_1^{\fold}$, we must have $\nabla_1^{\fold}=\nabla_1\cap V^{\fold}$. 
In other words, $\nabla_1^{\fold}=\conv((v_1\C_1\yy\cap W^{\fold}\yy)$. Similarly, $\nabla_2^{\fold}=\nabla_2\cap V^{\fold}=\conv(v_2\C_2\yy\cap W^{\fold}\yy)$. Thus, 
\[\nabla_1^{\fold}\cap\nabla_2^{\fold}=\nabla_1\cap\nabla_2\cap V^{\fold}=\conv(v_1\C_1\yy\cap v_2\C_2\yy\cap V^{\fold})=\conv(u_1\C_1^{\fold}\cap u_2\C_2^{\fold})\] 
is a common face of $\nabla_1^{\fold}$ and $\nabla_2^{\fold}$, as desired. 
\end{proof}

\section{Combinatorics of the SBDW Triangulation}\label{sec:combinatorics}

We now prove the results from \cref{subsec:properties-of-triang} in order to gain a much finer combinatorial understanding of SBDW triangulations, whose simplices are labeled by concordant pairs $(u, \bpi) \in W_c^+ \times \MCh(W,c)$. We characterize when elements of $W_c^+$ are concordant with the same set of chains in $\MCh(W,c)$ using Cambrian congruence classes, show the subset of $W_c^+$ concordant with a fixed $\bpi \in \MCh(W,c)$ is an interval in the left weak order, and characterize when $u \in W_c^+$ is concordant with $\bpi \in \MCh(W,c)$ using the positive cluster complex.

Recall that we write $\Sort(W,c)$ for the set of $c$-sortable elements of $W$, which are the minimum elements of the $c$-Cambrian congruence classes (in the right weak order).

For $u\in W$, recall that we write $\pi_c^\downarrow(u)$ for the maximum element of $\{w\in\Sort(W,c):w\leq_\R u\}$ in the right weak order. Let $J\subseteq S$, and let $c_J$ be a standard Coxeter element of the parabolic subgroup $W_J$. The right weak order on $W$ is a lattice; let $\wedge$ denote its meet operation. For $u\in W$, the set ${\{w\in\Sort(W_J,c_J):w\leq_\R u\}}$ has a maximum element, which is $\pi_{c_J}^\downarrow(u\wedge w_\circ(J))$, where $w_\circ(J)$ is the long element of $W_J$. Let us define $\pi_{c_J}^\downarrow(u)=\pi_{c_J}^\downarrow(u\wedge w_\circ(J))$. Thus, we have a well-defined map $\pi_{c_J}^\downarrow\colon W\to\Sort(W_J,c_J)$. 

We will need the next lemma in what follows.

\begin{lemma}\label{lem:parabolic} 
Let $J\subseteq S$, and let $c_J$ be a standard Coxeter element of $W_J$. 
Let $u_1,u_2\in W$ be such that $\ell_S(u_1c_J)=\ell_S(u_1)+\ell_S(c_J)$ and $\ell_S(u_2c_J)=\ell_S(u_2)+\ell_S(c_J)$. Suppose that $u_1\lessdot_\LL u_2$ and that $\pi_{c_J^{-1}}^\downarrow(u_1^{-1})=\pi_{c_J^{-1}}^\downarrow(u_2^{-1})$. Then $u_2\not\leq_{\hB}u_1c_J$. 
\end{lemma} 
\begin{proof}
We proceed by induction on $\ell_S(u_1)$ and on $|J|$. Let $w=\pi_{c_J^{-1}}^\downarrow(u_1^{-1})=\pi_{c_J^{-1}}^\downarrow(u_2^{-1})$. Let $s$ be a right descent of $c_J$, and let $J'=J\setminus\{s\}$. Let $c_{J'}=c_Js$; note that $c_{J'}$ is a standard Coxeter element of $W_{J'}$. We consider two cases. 

\medskip 

\noindent {\bf Case 1.} Suppose $s$ is a left descent of $w$. Since $w\leq_\R u_1^{-1}$ and $w\leq_\R u_2^{-1}$, we find that $s$ is a left descent of both $u_1^{-1}$ and $u_2^{-1}$. It follows from \cite[Equation~3.1]{reading2007sortable} that ${\pi_{sc_{J'}}^\downarrow(su_1^{-1})=\pi_{sc_{J'}}^\downarrow(su_2^{-1})=sw}$. By induction on $\ell_S(u_1)$, we know that $u_2s\not\leq_{\hB}u_1s(sc_{J'})=u_1c_Js$. Since $s$ is not a right descent of $u_2s$ or $u_1c_Js$, it follows from \cref{lem:Bruhat_descents} that $u_2\not\leqB u_1c_J$.

\medskip 

\noindent {\bf Case 2.} Suppose $s$ is not a left descent of $w$. Because $w\in\Sort(W,c_J^{-1})$, we have $w\in W_{J'}$. Because $\Sort(W_{J'},c_{J'}^{-1})\subseteq\Sort(W_J,c_J^{-1})$, we must have 
\[\pi_{c_J^{-1}}^\downarrow(u_1^{-1})=\pi_{c_J^{-1}}^\downarrow(u_2^{-1})=w.\] We also have $\ell_S(u_1c_{J'})=\ell_S(u_1)+\ell_S(c_{J'})$ and $\ell_S(u_2c_{J'})=\ell_S(u_2)+\ell_S(c_{J'})$. By induction, \[u_2\not\leqB u_1c_{J'}=u_1c_Js.\] Suppose by way of contradiction that $u_2\leqB u_1c_J$. This forces $s$ to be a left descent of $u_2^{-1}$. Because $s$ is also a left descent of the long element $w_\circ(J)$ of $W_J$, we find that $s$ is a left descent of $u_2^{-1}\wedge w_\circ(J)$ (the meet is computed in the right weak order). Each element $x\in W_J$ has the same set of left descents as $\pi_{c_J^{-1}}^{\downarrow}(x)$. This implies that $s$ is a left descent of the element \[\pi_{c_J^{-1}}^{\downarrow}(u_2^{-1}\wedge w_\circ(J))=\pi_{c_J^{-1}}^\downarrow(u_2^{-1})=w,\] which is a contradiction. 
\end{proof} 

The next lemma constitutes part of \cref{thm:Cambrian}. Recall that $\Class_c(u)$ denotes the set of commutation equivalence classes of chains $\bpi\in\MCh(W,c)$ such that $(u,\bpi)\in\Omega(W,c)$. 

\begin{lemma}\label{lem:Cambrian_one_direction} 
Suppose $u_1,u_2\in W_c^+$ are such that $u_1^{-1}$ and $u_2^{-1}$ belong to the same $c^{-1}$-Cambrian class. Then $\Class_c(u_1)=\Class_c(u_2)$. 
\end{lemma} 
\begin{proof}
Each $c^{-1}$-Cambrian congruence class is an interval in the right weak order, so we may assume that $u_1^{-1}\lessdot_{\R}u_2^{-1}$. Thus, there is a simple reflection $s^*$ such that $u_2=s^* u_1$. Setting $J=S$ in \cref{lem:parabolic}, we find that $u_2\not\leqB u_1c$.  

We claim that no element of $[u_1,u_1c]_{\hB}$ has $s^*$ as a left descent and that every element of $[u_2,u_2c]_{\hB}$ has $s^*$ as a left descent. We will prove the first part of the claim about $[u_1,u_1c]_{\hB}$; an analogous argument handles the other part. Thus, suppose by way of contradiction that there exists some element $x\in[u_1,u_1c]_{\hB}$ that has $s^*$ as a left descent. Then there is a reduced $S$-word ${\sf x}$ for $x$ that starts with $s^*$. Since $u_1\leqB x$, there is a reduced $S$-word ${\sf u}_1$ for $u_1$ that appears as a subword of ${\sf x}$. But ${\sf u}_1$ does not start with $s^*$ because $s^*$ is not a left descent of $u_1$. Therefore, $s^*{\sf u}_1$ is a reduced $S$-word for $u_2$ that appears as a subword of ${\sf x}$. This implies that $u_2\leqB x\leqB u_1c$, which is a contradiction. 

It follows immediately from the above claim and \cref{lem:Bruhat_descents} that there is a poset isomorphism $\varphi\colon[u_1,u_1c]_{\hB}\to[u_2,u_2c]_{\hB}$ given by $\varphi(z)=s^*z$. Moreover, for each maximal chain $u_1\bpi$ in $[u_1,u_1c]_{\hB}$, the corresponding maximal chain of $[u_2,u_2c]_{\hB}$ is $\varphi(u_1\bpi)=u_2\bpi$. Hence, $\Class_c(u_1)=\Class_c(u_2)$. 
\end{proof}

\begin{lemma}\label{lem:all_have_descent}
Suppose $v,w\in W$ are such that $v\leqB w$. Let $t_1\cdots t_k$ be a reduced $T$-word for $v^{-1}w$, and let $\bpi=\{e=\pi_0<_T \pi_1<_T\cdots<_T\pi_k=v^{-1}w\}$, where $\pi_i=t_1\cdots t_i$. Let $s$ be a right descent of $w$. Suppose $s\not\in\{t_1,\ldots,t_k\}$. Then $s$ is a right descent of every element of $v\bpi$. 
\end{lemma}
\begin{proof}
Let $\mathsf{w}$ be a reduced $S$-word for $w$ that ends with $s$. Starting from this word, we can delete one letter at a time to produce reduced $S$-words for $v\pi_{k-1},v\pi_{k-2},\ldots,v\pi_1,v$. At no point during this process can we delete the letter $s$ on the very right; indeed, if we deleted this rightmost $s$ when transitioning from the reduced $S$-word for $v\pi_{i}$ to $v\pi_{i-1}$, then it would follow that $t_i=s$. This shows that all of the elements in the chain $v\bpi$ have reduced $S$-words ending in $s$.  
\end{proof}

We can now complete the proof of \cref{thm:unique_decreasing}, which states that each set $\Class_c(u)$ for $u\in W_c^+$ contains the $\preceq_c$-increasing class $\mathcal I_c$ and a unique $\preceq_c$-decreasing class.

\begin{proof}[Proof of \cref{thm:unique_decreasing}] 
Fix a standard Coxeter element $c$ of $W$, and let $u\in W_c^+$. Let $s_1\cdots s_r$ be a reduced $S$-word for $c$. Let $c'=s_1\cdots s_{r-1}=cs_r$. The unique $\preceq_c$-increasing class $\mathcal I_c$ contains the chain $\{e<_T s_1<_Ts_1s_2<_T\cdots <_T s_1s_2\cdots s_r\}$. Since $u\in W_c^+$, we have 
\[u\lessdot_{\hB} us_1\lessdot_{\hB}us_1s_2\lessdot_{\hB}\cdots\lessdot_{\hB}us_1s_2\cdots s_r.\] This shows that $\mathcal I_c\in\Class_c(u)$. We are left to prove that $\Class_c(u)$ contains a unique $\preceq_c$-decreasing chain. By \cref{lem:Cambrian_one_direction}, we may assume $u^{-1}$ is $c^{-1}$-sortable. We proceed by induction on the rank $r$ and the length $\ell_S(u)$. We consider two cases based on whether or not $u$ belongs to the parabolic subgroup $W_{\langle s_r\rangle}$. 

\medskip 

\noindent {\bf Case 1.} Suppose $u\in W_{\langle s_r\rangle}$. Assume that $\bpi=\{e=\pi_0<_T\pi_1<_T\cdots<_T\pi_r=c\}$ is a $\preceq_c$-decreasing chain such that $(u,\bpi)\in\Omega(W,c)$, and let $\rw(\bpi)=t_1\cdots t_r$ (so $t_i=\pi_{i-1}^{-1}\pi_i$). Since $u\in W_{\langle s_r\rangle}$ and $uc\not\in W_{\langle s_r\rangle}$, there is some reflection in $\rw(\bpi)$ that is not in $W_{\langle s_r\rangle}$. Let $k$ be the smallest index such that $t_k\not\in W_{\langle s_r\rangle}$. Thus, $u\pi_k\not\in W_{\langle s_r\rangle}$. Let $q=uc'=us_1\cdots s_{r-1}$, and let ${\sf q}$ be a reduced $S$-word for $q$. Then ${\sf q}s_r$ is a reduced $S$-word for the element $uc=u\pi_r$; note that this word contains only one copy of the letter $s_r$ (on the very right). Starting from this word ${\sf q}s_r$, we can delete one letter at a time to produce reduced $S$-words for $u\pi_{r-1},u\pi_{r-2},\ldots,u\pi_k$. At no point during this process can we delete the letter $s_r$ on the very right; indeed, if we did, then at the end we would have a reduced $S$-word for the element $u\pi_k$ that does not use the letter $s_r$. This shows that $s_r$ is a right descent of each of the elements $u\pi_k,u\pi_{k+1},\ldots,u\pi_r$. Since $s_r$ is not a right descent of $u\pi_{k-1}$, we must have $t_k=s_r$. But $s_r$ is a maximal element of the poset $(T,\preceq_c)$ by \cref{lem:preceq_c}, so since $\bpi$ is decreasing, all of the reflections $t_1,\ldots,t_{k-1}$ commute with $t_k$. Thus, after applying commutation moves, we may assume that $k=1$. Hence, $s_r$ is a right descent of each of the elements $u\pi_1,u\pi_2,\ldots,u\pi_r$. It follows from \cref{lem:Bruhat_descents} that $u\bpi'=\{u\lessdot_{\hB}u\pi_1'\lessdot_{\hB}\cdots\lessdot_{\hB}u\pi_{r-1}'\}$ is a maximal chain in $[u,uc']_{\hB}$, where $\pi_i'=\pi_{i+1}s_r$. Moreover, $\rw(\bpi')=t_1'\cdots t_{r-1}'$, where $t_i'=s_rt_{i+1}s_r$. This implies that $t_1',\ldots,t_{r-1}'\in W_{\langle s_r\rangle}$. According to \cref{lem:preceq_c2}, the chain $\bpi'$ is $\preceq_{c'}$-decreasing, so by induction on $r$, its commutation equivalence class $\CL_{\bpi'}$ is the unique $\preceq_{c'}$-decreasing class in $\Class_{c'}(u)$. 

The preceding paragraph proves that there is at most one $\preceq_c$-decreasing class in $\Class_c(u)$. The existence of such a class follows from reversing this argument. More precisely, given the $\preceq_{c'}$-decreasing chain $\bpi'=\{e<_T\pi_1'<_T\cdots<_T\pi_{r-1}'\}$ in $\Class_{c'}(u)$, we construct the $\preceq_c$-decreasing chain $\bpi=\{e<_T\pi_1<_T\cdots<_T\pi_r\}$ with $\CL_{\bpi}\in\Class_c(u)$ by letting $\pi_1=s_r$ and letting $\pi_{i+1}=\pi_is_r$ for $1\leq i\leq r-1$. 

\medskip 

\noindent {\bf Case 2.} Suppose $u\not\in W_{\langle s_r\rangle}$. Assume that $\bpi=\{e=\pi_0<_T\pi_1<_T\cdots<_T\pi_r=c\}$ is a $\preceq_c$-decreasing chain such that $(u,\bpi)\in\Omega(W,c)$, and let $\rw(\bpi)=t_1\cdots t_r$ (so $t_i=\pi_{i-1}^{-1}\pi_i$). The hypothesis that $u^{-1}$ is $c^{-1}$-sortable implies that $s_r$ is a right descent of $u$. Therefore, it is impossible to perform commutation moves to the word $t_1\cdots t_r$ to produce a $T$-word that starts with $s_r$. Since $\bpi$ is $\preceq_c$-decreasing, it follows from \cref{lem:preceq_c} that none of the reflections $t_1,\ldots,t_r$ is equal to $s_r$. By \cref{lem:all_have_descent}, all of the elements in the chain $u\bpi$ have $s_r$ as a right descent. Hence, 
\[us_r\lessdot_{\hB}u\pi_1s_r\lessdot_{\hB}\cdots\lessdot_{\hB}u\pi_rs_r\] by \cref{lem:Bruhat_descents}. Let $\pi_i'=s_r\pi_is_r$. The chain $\bpi'=\{e=\pi_0'<_T\pi_1'<_T\cdots<_T\pi_r'=s_rc'\}$ is such that $(us_r,\bpi')\in\Omega(W,s_rc')$. Moreover, the chain $\bpi'$ is $\preceq_{s_rc'}$-decreasing by \cref{lem:preceq_c2}. By induction on $\ell_S(u)$, we know that $\Class_{s_rc'}(us_r)$ has a unique decreasing commutation equivalence class. 

The preceding paragraph proves that there is at most one $\preceq_c$-decreasing class in $\Class_c(u)$. The existence of such a class follows from reversing this argument. More precisely, given the $\preceq_{s_rc'}$-decreasing chain $\bpi'=\{e<_T\pi_1'<_T\cdots<_T\pi_{r}'\}$ with $\CL_{\bpi'}\in\Class_{s_rc'}(us_r)$, we construct a $\preceq_c$-decreasing chain $\bpi$ with $\CL_{\bpi}\in\Class_c(u)$ by conjugating the elements of $\bpi'$ by $s_r$. 
\end{proof} 

For $u\in W_c^+$, recall that we write $\DD_c(u)$ for the $\preceq_c$-decreasing class in $\Class_c(u)$. 

\begin{lemma}\label{lem:different_sortables} 
Suppose $u_1,u_2\in W_c^+$ are such that $u_1^{-1},u_2^{-1}\in\Sort(W,c^{-1})$. If $\DD_c(u_1)=\DD_c(u_2)$, then $u_1=u_2$. 
\end{lemma}
\begin{proof}
Let $s_1\cdots s_r$ be a reduced $S$-word for $c$, and let $c'=cs_r$. For $u\in W_c^+$, we saw in the proof of \cref{thm:unique_decreasing} that $u\in W_{\langle s_r\rangle}$ if and only if the class $\DD_c(u)$ contains a chain $\bpi$ such that $\rw(\bpi)$ begins with $s_r$. Now suppose $\DD_c(u_1)=\DD_c(u_2)$. Then $u_1$ and $u_2$ are either both in $W_{\langle s_r\rangle}$ or both in $W\setminus W_{\langle s_r\rangle}$. We will prove by induction on $r$ and $\min\{\ell_S(u_1),\ell_S(u_2)\}$ that $u_1=u_2$. We consider two cases. 

\medskip 

\noindent {\bf Case 1.} Suppose $u_1,u_2\in W_{\langle s_r\rangle}$. Then there exists a chain $\bpi=\{e=\pi_0<_T\pi_1<_T\cdots<_T\pi_r=c\}$ in the class $\DD_c(u_1)=\DD_c(u_2)$ such that $\rw(\bpi)$ begins with $s_r$. For $0\leq i\leq r-1$, let $\pi_i'=\pi_{i+1}s_r$. We saw in the proof of \cref{thm:unique_decreasing} that the chain $\bpi'=\{e=\pi_0'<_T\pi_1'<_T\cdots<\pi_{r-1}'=c'\}$ belongs to both $\DD_{c'}(u_1)$ and $\DD_{c'}(u_2)$. Hence, $\DD_{c'}(u_1)=\DD_{c'}(u_2)$. By induction on $r$, we conclude that $u_1=u_2$. 

\medskip 

\noindent {\bf Case 2.} Suppose $u_1,u_2\in W\setminus W_{\langle s_r\rangle}$. Since $u_1^{-1}$ and $u_2^{-1}$ are $c^{-1}$-sortable, $s_r$ is a right descent of both $u_1$ and $u_2$. Choose a chain $\bpi=\{e=\pi_0<_T\pi_1<_T\cdots<_T\pi_r=c\}$ in the class ${\DD_c(u_1)=\DD_c(u_2)}$. For $0\leq i\leq r$, let $\pi_i'=s_r\pi_is_r$. We saw in the proof of \cref{thm:unique_decreasing} that the chain ${\bpi'=\{e=\pi_0'<_T\pi_1'<_T\cdots<_T\pi_r'=s_rc'\}}$ belongs to both $\DD_{s_rc'}(u_1s_r)$ and $\DD_{s_rc'}(u_1s_r)$. Hence, $\DD_{s_rc'}(u_1s_r)=\DD_{s_rc'}(u_2s_r)$. By induction on $\min\{\ell_S(u_1),\ell_S(u_2)\}$, we conclude that $u_1s_r=u_2s_r$, so $u_1=u_2$. 
\end{proof}

We can now complete the proof of \cref{thm:Cambrian}. 

\begin{proof}[Proof of \cref{thm:Cambrian}]
Let $c$ be a standard Coxeter element of $W$, and let $u_1,u_2\in W_c^+$. Let $w_1=\pi_{c^{-1}}^\downarrow(u_1^{-1})$ and $w_2=\pi_{c^{-1}}^\downarrow(u_2^{-1})$. Then $u_1^{-1}$ and $u_2^{-1}$ are in the same $c^{-1}$-Cambrian class if and only if $w_1=w_2$. By definition, the intervals $[u_1,u_1c]_\hB$ and $[u_2,u_2c]_\hB$ are reflection isomorphic if and only if $\Class_c(u_1)=\Class_c(u_2)$. 

If $u_1^{-1}$ and $u_2^{-1}$ are in the same $c^{-1}$-Cambrian class, then ${\Class_c(u_1)=\Class_c(u_2)}$ by \cref{lem:Cambrian_one_direction}. If $\Class_c(u_1)=\Class_c(u_2)$, then $\DD_c(u_1)=\DD_c(u_2)$ by \cref{thm:unique_decreasing}. If $\DD_c(u_1)=\DD_c(u_2)$, then we have ${\DD_c(w_1^{-1})=\DD_c(w_2^{-1})}$ by \cref{lem:Cambrian_one_direction}, so it follows from \cref{lem:different_sortables} that $w_1=w_2$. 
\end{proof} 

\begin{remark}\label{rem:explicit} 
Suppose $u\in W_c^+$ is such that $u^{-1}\in\Sort(W,c^{-1})$. While we have proven the existence and uniqueness of the $\preceq_c$-decreasing commutation equivalence class $\DD_c(u)$ by induction, it is also possible to construct this class explicitly. Fix a reduced $S$-word $s_1\cdots s_r$ for $c$. Let us consider the infinite word 
\begin{equation}\label{eq:infinite_word}
s_r^{(1)}\cdots s_1^{(1)}s_r^{(2)}\cdots s_1^{(2)}s_r^{(3)}\cdots s_1^{(3)}\cdots,
\end{equation} 
where $s_i^{(1)},s_i^{(2)},s_i^{(3)},\ldots$ are copies of the same letter $s_i$. Recall that the $c^{-1}$-sorting word for $u^{-1}$ is the lexicographically first subword of the word in \eqref{eq:infinite_word} that is a reduced $S$-word for $u^{-1}$. For each $i\in[r]$, let $b_i$ be the smallest positive integer such that $s_i^{(b_i)}$ is not used in the $c^{-1}$-sorting word for $u^{-1}$. The copy $s_i^{(b_i)}$ of the letter $s_i$ is called a \dfn{skip}. Let $q_i$ be the element of $W$ obtained by multiplying the letters to the left of $s_i^{(b_i)}$ that either are skips or are used in the $c^{-1}$-sorting word for $u^{-1}$. 
Let $t_i=q_is_iq_i^{-1}$. Let $j_1,\ldots,j_r$ be the permutation of $[r]$ such that the skips appear in the order $s_{j_1}^{b_{j_1}},\ldots,s_{j_r}^{(b_{j_{r}})}$ from left to right in \eqref{eq:infinite_word}. Define $\pi_m=t_{j_1}t_{j_{2}}\cdots t_{j_m}$. Then one can show that the chain \[{\{e=\pi_0<_T\pi_1<_T\pi_2<_T\cdots<\pi_r=c\}}\] is $\preceq_c$-decreasing and that $(u,\bpi)\in\Omega(W,c)$. 

For example, suppose $W=A_4=\SSS_5$, and let $c=s_1s_2s_3s_4$, where $s_i$ is the transposition $(i\,\,i+1)$. Let ${u=(2\,4\,5)=s_3s_2s_3s_4}$. The word in \eqref{eq:infinite_word} is 
\begin{equation}\label{eq:infinite_word2}{\color{red}s_4^{(1)}s_3^{(1)}s_2^{(1)}}{\color{NormalGreen}s_1^{(1)}s_4^{(2)}}{\color{red}s_3^{(2)}}{\color{NormalGreen}s_2^{(2)}}s_1^{(2)}s_4^{(3)}{\color{NormalGreen}s_3^{(3)}}s_2^{(3)}s_1^{(3)}s_4^{(4)}s_3^{(4)}s_2^{(4)}s_1^{(4)}\cdots,\end{equation} where the $c^{-1}$-sorting word for $u^{-1}$ 
is colored {\color{red}red} and the skips are colored {\color{NormalGreen}green}. We have $(b_1,b_2,b_3,b_4)=(1,2,3,2)$ and $(j_1,j_2,j_3,j_4)=(1,4,2,3)$. For each $i$, the element $q_i$ is the product of the {\color{red}red} and {\color{NormalGreen}green} simple reflections that appear to the left of $s_i^{(b_i)}$ in \eqref{eq:infinite_word2}. Thus, 
\begin{align*}
&q_1={\color{red}s_4s_3s_2}=(2\,5\,4\,3), &&q_2={\color{red}s_4s_3s_2}{\color{NormalGreen}s_1s_4}{\color{red}s_3}=(1\,5\,3\,4\,2), \\ &q_3={\color{red}s_4s_3s_2}{\color{NormalGreen}s_1s_4}{\color{red}s_3}{\color{NormalGreen}s_2}=(1\,5\,3)(2\,4), &&q_4={\color{red}s_4s_3s_2}{\color{NormalGreen}s_1}=(1\,5\,4\,3\,2).
\end{align*}
We compute that 
\[t_1=q_1s_1q_1^{-1}=(1\,5),\quad t_2=q_2s_2q_2^{-1}=(1\,4),\quad t_3=q_3s_3q_3^{-1}=(1\,2),\quad t_4=q_4s_4q_4^{-1}=(3\,4).\] 
Then 
\[\pi_1=t_1=(1\,5),\quad \pi_2=t_1t_4=(1\,5)(3\,4),\quad \pi_3=t_1t_4t_2=(1\,3\,4\,5),\quad \pi_4=t_1t_4t_2t_3=(1\,2\,3\,4\,5),\] and one can verify that the chain $\bpi=\{e=\pi_0<_T\pi_1<_T\pi_2<_T\pi_3<_T\pi_4=c\}$ is $\preceq_c$-decreasing and that $(u,\bpi)\in\Omega(W,c)$.  
\end{remark} 

\begin{remark}\label{rem:positive_Catalan}
When $W$ is irreducible, it is known that $\left|\Sort^+(W,c^{-1})\right|$ is the positive $W$-Catalan number $\Cat^+(W)$ (defined in \eqref{eq:Cat}), which is also equal to the number of $\preceq_c$-decreasing commutation equivalence classes of chains in $\MCh(W,c)$. Therefore, it follows from \cref{thm:unique_decreasing,thm:Cambrian} that the map $u^{-1}\mapsto\DD_c(u)$ is a bijection from $\Sort^+(W,c^{-1})$ to the set of $\preceq_c$-decreasing classes.  Such a bijection is well known; see, for example,~\cite[Section~6]{pilaud2015brick}.
\end{remark}

Given a commutation equivalence class $\CL_{\bpi}$ of a chain $\bpi\in\MCh(W,c)$, define $\RRR(\CL_{\bpi})$ to be the set of positive roots corresponding to the reflections in the reduced $T$-word $\rw(\bpi)$. That is, if $\rw(\bpi)=t_1\cdots t_r$, then $\RRR(\CL_{\bpi})=\{\beta_{t_1},\ldots,\beta_{t_r}\}$. We now wish to prove \cref{thm:spans_new}, which states that a chain $\bpi\in\MCh(W,c)$ is concordant with an element $u\in W_c^+$ if and only if $\Delta(u)\subseteq\Delta(\bpi)$. Our proof relies on the following lemma. 

\begin{lemma}\label{lem:roots_weights}
For $\bpi\in\MCh(W,c)$ and $u\in W_c^+$, we have $\spange\RRR(\DD_c(u))\subseteq\spange\RRR(\CL_{\bpi})$ if and only if $\Delta(u)\subseteq\Delta(\bpi)$. 
\end{lemma} 
\begin{proof}
The unique $\preceq_c$-decreasing commutation class $\CL$ such that $\Delta(u)\subseteq\Delta(\CL)$ is $\DD_c(u)$. There exist $\preceq_c$-decreasing commutation classes $\CL_1,\ldots,\CL_k$ such that $\Delta(\bpi)=\bigcup_{i=1}^k\Delta(\CL_i)$. Therefore, $\Delta(u)\subseteq\Delta(\bpi)$ if and only if $\Delta(\DD_c(u))\subseteq\Delta(\bpi)$. For each chain $\bpi'\in\MCh(W,c)$, the cone $\Delta(\bpi')=\Delta(\CL_{\bpi'})$ is the nonnegative span of the vectors $\dd_t$ such that $t$ appears in $\rw(\bpi')$. There is a linear isomorphism from $V^*$ to $V$ that sends each root $\beta$ to a positive scalar multiple of the associated coroot $\beta^\vee$. \cref{lem:deltabeta} tells us that there is a linear automorphism of $V$ that sends $\beta_t$ to $\dd_t$ for every $t\in T$. Consequently, $\Delta(\DD_c(u))\subseteq\Delta(\bpi)$ if and only if $\spange\RRR(\DD_c(u))\subseteq\spange\RRR(\CL_{\bpi})$. 
\end{proof} 

\cref{lem:roots_weights} implies that \cref{thm:spans_new} is equivalent to the next proposition. 

\begin{proposition}\label{thm:spans} 
For $u\in W_c^+$, we have $\CL\in\Class_c(u)$ if and only if \[\Span_{\geq 0}\RRR(\DD_c(u))\subseteq\Span_{\geq 0}\RRR(\CL).\]
\end{proposition} 

We start by proving one direction of this result. 

\begin{lemma}\label{lem:span_proof1}
Choose an element $u\in W_c^+$ such that ${u^{-1}\in\Sort(W,c^{-1})}$. If $\CL\in\Class_c(u)$, then $\Span_{\geq 0}\RRR(\DD_c(u))\subseteq\Span_{\geq 0}\RRR(\CL)$.     
\end{lemma} 
\begin{proof} 
Let $s_1\cdots s_r$ be a reduced $S$-word for $c$, and let $c'=cs_r$. Choose a chain \[\bpi=\{e=\pi_0<_T\pi_1<_T\cdots<_T\pi_r=c\}\] in $\CL$, and let $\rw(\bpi)=t_1\cdots t_r$ so that $\RRR(\CL)=\{\beta_{t_1},\ldots,\beta_{t_r}\}$. Let $t_1',\ldots,t_r'$ be the reflections appearing in the reduced $T$-words corresponding to chains in $\DD_c(u)$ so that $\RRR(\DD_c(u))=\{\beta_{t_1'},\ldots,\beta_{t_r'}\}$.  
We proceed by induction on $r$ and $\ell_S(u)$. We consider three cases. 

\medskip 

\noindent {\bf Case 1.} Suppose $s_r\not\in\{t_1,\ldots,t_r\}$. By \cref{lem:all_have_descent}, all of the elements in the chain $u\bpi$ have $s_r$ as a right descent. Hence, 
\[us_r\lessdot_{\hB}u\pi_1s_r\lessdot_{\hB}\cdots\lessdot_{\hB}u\pi_rs_r\] by \cref{lem:Bruhat_descents}. Letting $\pi_i'=s_r\pi_is_r$, we find that the chain \[\bpi'=\{e=\pi_0'<_T\pi_1'<_T\cdots<_T\pi_r'=s_rc'\}\] is such that $(us_r,\bpi')\in\Omega(W,s_rc')$. By induction on $\ell_S(u)$, we know that \[\Span_{\geq 0}\RRR(\DD_{s_rc'}(us_r))\subseteq\Span_{\geq 0}\RRR(\CL_{\bpi'}).\] Now, 
\[\RRR(\CL_{\bpi'})=\{\beta_{s_rt_1s_r},\ldots,\beta_{s_rt_rs_r}\}.\] It follows from the proof of \cref{thm:unique_decreasing} that \[\RRR(\DD_{s_rc'}(us_r))=\{\beta_{s_rt_1's_r},\ldots,\beta_{s_rt_r's_r}\}.\] The set $\RRR(\CL)=\{\beta_{t_1},\ldots,\beta_{t_r}\}$ is obtained by applying the action of $s_r$ to the set $\RRR(\CL_{\bpi'})$; that is, $\RRR(\CL)=s_r\RRR(\CL_{\bpi'})$. Similarly, $\RRR(\DD_c(u))=s_r\RRR(\DD_{s_rc'}(us_r))$. Hence, $\Span_{\geq 0}\RRR(\DD_c(u))\subseteq\Span_{\geq 0}\RRR(\CL)$.

\medskip 

\noindent {\bf Case 2.} 
Suppose $s_r=t_1$. Then $s_r$ is not a right descent of $u$, so $u\in W_{\langle s_r\rangle}$ because $u^{-1}$ is $c^{-1}$-sortable. By \cref{lem:all_have_descent}, the elements $u\pi_1,u\pi_2,\ldots,u\pi_r$ all have $s_r$ as a right descent. Hence, 
\[u\lessdot_{\hB}u\pi_2s_r\lessdot_{\hB}u\pi_3s_r\lessdot_{\hB}\cdots\lessdot_{\hB}u\pi_rs_r\] by \cref{lem:Bruhat_descents}. Letting $\pi_i'=s_r\pi_{i+1}s_r$, we find that the chain \[\bpi'=\{e=\pi_0'<_T\pi_1'<_T\cdots<_T\pi_{r-1}'=c'\}\] is such that $(u,\bpi')\in\Omega(W_{\langle s_r\rangle},c')$. By induction on $r$, we know that \[\Span_{\geq 0}\RRR(\DD_{c'}(u))\subseteq\Span_{\geq 0}\RRR(\CL_{\bpi'}).\] Now, 
\[\RRR(\CL_{\bpi'})=\{\beta_{s_rt_2s_r},\beta_{s_rt_3s_r},\ldots,\beta_{s_rt_rs_r}\}.\] It follows from the proof of \cref{thm:unique_decreasing} that we can order the reflections $t_1',\ldots,t_r'$ so that $t_1'=s_r$. It also follows from the same proof that \[\RRR(\DD_{c'}(u))=\{\beta_{s_rt_2's_r},\beta_{s_rt_3's_r},\ldots,\beta_{s_rt_r's_r}\}.\] The set $\RRR(\CL)\setminus\{s_r\}=\{\beta_{t_2},\beta_{t_3},\ldots,\beta_{t_r}\}$ is obtained by applying the action of $s_r$ to the set $\RRR(\CL_{\bpi'})$; that is, $\RRR(\CL)\setminus\{s_r\}=s_r\RRR(\CL_{\bpi'})$. Similarly, $\RRR(\DD_c(u))\setminus\{s_r\}=s_r\RRR(\DD_{c'}(u))$. Hence, $\Span_{\geq 0}(\RRR(\DD_c(u))\setminus\{s_r\})\subseteq\Span_{\geq 0}(\RRR(\CL)\setminus\{s_r\})$. As a consequence, $\Span_{\geq 0}\RRR(\DD_c(u))\subseteq\Span_{\geq 0}\RRR(\CL)$. 

\medskip 

\noindent {\bf Case 3.} Suppose $s_r=t_k$ for some $k\in[r]$. We proceed by induction on $k$, noting that we already handled the base case (when $k=1$) in Case~2 above. Now suppose $k\geq 2$. By \cref{lem:quadrilateral}, there is a unique element $\pi_{k-1}^\#\in W$ such that the Bruhat interval $[u\pi_{k-2},u\pi_k]_{\hB}$ is a diamond with elements $u\pi_{k-2},u\pi_{k-1},u\pi_{k-1}^\#,u\pi_k$. Let $t_{k-1}^\#=\pi_{k-2}^{-1}\pi_{k-1}^\#$ and $t_{k}^\#=(\pi_{k-1}^\#)^{-1}\pi_{k}$. The chain \[\bpi^\#=\{\pi_0<_T\cdots<_T\pi_{k-2}<_T\pi_{k-1}^\#<_T\pi_k<_T\cdots<_T\pi_r\}\] is in $\MCh(W,c)$, and $(u,\bpi^\#)\in\Omega(W,c)$. According to \cref{lem:basification}, the subgroup $W'$ of $W$ generated by $t_{k-1},t_k,t_{k-1}^\#,t_k^\#$ has rank $2$. The reflection $t_k=s_r$ is a right descent of $c$, and $t_{k-1}$ appears to the left of $t_k$ in a reduced $T$-word for $c$. Therefore, it follows from \cref{lem:i-1modk} that $t_{k-1}$ and $t_k$ are the canonical generators of $W'$. This implies that \[\Span_{\geq 0}\left\{\beta_{t_{k-1}^\#},\beta_{t_k^\#}\right\}\subseteq \Span_{\geq 0}\left\{\beta_{t_{k-1}},\beta_{t_k}\right\},\] so $\Span_{\geq 0}\RRR(\CL_{\bpi^\#})\subseteq\Span_{\geq 0}\RRR(\CL)$. Consequently, to complete the proof, it suffices to show that 
\begin{equation}\label{eq:span_proof1}
\Span_{\geq 0}\RRR(\DD_c(u))\subseteq\Span_{\geq 0}\RRR(\CL_{\bpi^\#}).
\end{equation} 
Because $\pi_{k}=\pi_{k-1}t_k=\pi_{k-1}^\#t_k^\#$, we know that $s_r=t_k\neq t_k^\#$. If $s_r\neq t_{k-1}^\#$, then $s_r$ does not appear in $\rw(\bpi^\#)$, so we already know from Case~1 above that \eqref{eq:span_proof1} holds. On the other hand, if $s_r=t_{k-1}^\#$, then we know by induction on $k$ that \eqref{eq:span_proof1} holds. 
\end{proof}

\begin{proof}[Proof of \cref{thm:spans}]
Let $c$ be a standard Coxeter element of $W$, and let $u\in W_c^+$. Let $\CL$ be a commutation equivalence class of chains in $\MCh(W,c)$. If $\CL\in\Class_c(u)$, then it is immediate from \cref{thm:Cambrian,lem:span_proof1} that $\Span_{\geq 0}\RRR(\DD_c(u))\subseteq\Span_{\geq 0}\RRR(\CL)$. We now aim to prove the converse. Thus, let us assume $\Span_{\geq 0}\RRR(\DD_c(u))\subseteq\Span_{\geq 0}\RRR(\CL)$; we will prove that $\CL\in\Class_c(u)$.     

Let $\bpi=\{e=\pi_0<_T\pi_1<_T\cdots<_T\pi_r=c\}$ be a chain in $\CL$, and let $\rw(\bpi)=t_1\cdots t_r$. We proceed by induction on the number of positive roots that are not in $\Span_{\geq 0}\RRR(\CL)$. For the base case, suppose $\Phi^+\subseteq\Span_{\geq 0}\RRR(\CL)$. This implies that $t_1,\ldots,t_r$ are precisely the simple reflections. It follows that $\CL$ is the $\preceq_c$-increasing class $\mathcal I_c$, which belongs to $\Class_c(u)$ by \cref{thm:unique_decreasing}. 

Now suppose $\Phi^+\not\subseteq\Span_{\geq 0}\RRR(\CL)$. Then $\CL\neq\mathcal I_c$, so there exist $k,k'\in[r]$ with $k<k'$ and $t_{k'}\prec_c t_k$. By applying commutation moves to the word $t_1\cdots t_r$ if necessary, we may assume that $k'=k+1$. Let $W'$ be the rank-2 parabolic subgroup of $W$ generated by $t_k$ and $t_{k+1}$. Since $t_1\cdots t_r$ is a reduced $T$-word for $c$ and $t_{k+1}\prec_c t_k$, it follows from \cref{lem:i-1modk} that $\{t_k,t_{k+1}\}$ is not the set of canonical generators of $W'$. Hence $W'$ is not abelian. Therefore, we may write $T\cap W'=\{p_1\prec_cp_2\prec_c\cdots\prec_cp_m\}$ for some $m\geq 3$. The canonical generators of $W'$ are $p_1$ and $p_m$, and, according to \cref{lem:i-1modk}, there exists $j\in[m-1]$ such that $t_k=p_{j+1}$ and $t_{k+1}=p_j$. 
We have $p_1p_m=t_kt_{k+1}$. Then $t_1\cdots t_{k-1}p_1p_mt_{k+2}\cdots t_r$ is a reduced $T$-word for $c$, so it is $\rw(\bpi')$ for some chain $\bpi'=\{e=\pi_0'<_T\pi_1'<_T\cdots<_T\pi_r'=c\}\in\MCh(W,c)$. We have $\Span_{\geq 0}\{p_1,p_m\}\supseteq\Span_{\geq 0}\{t_k,t_{k+1}\}$, so 
\[\Span_{\geq 0}\RRR(\CL_{\bpi'})\supseteq\Span_{\geq 0}\RRR(\CL)\supseteq\Span_{\geq 0}\RRR(\DD_c(u)).\] By induction, we know that $\CL_{\bpi'}\in\Class_c(u)$. Thus, 
\[u\lessdot_{\hB}u\pi_1'\lessdot_{\hB}\cdots\lessdot_{\hB}u\pi_r'.\] By \cref{lem:quadrilateral}, the Bruhat interval $[u\pi_{k-1}',u\pi_{k+1}']_{\hB}$ is a diamond consisting of $u\pi_{k-1}',u\pi_{k}',u\pi_{k+1}'$, and some other element $x$. Let $t_k'=(u\pi_{k-1}')^{-1}x$ and $t_{k+1}'=x^{-1}u\pi_{k+1}'$. The maximal chain 
\[\{u\lessdot_{\hB}u\pi_1'\lessdot_{\hB}\cdots u\pi_{k-1}'\lessdot_{\hB}x\lessdot_{\hB}u\pi_{k+1}'\lessdot_{\hB}\cdots\lessdot_{\hB}u\pi_r'\}\] in the interval $[u,uc]_{\hB}$ is of the form $u\bpi''$ for some $\bpi''\in\MCh(W,c)$. Note that $\CL_{\bpi''}\in\Class_c(u)$. We will show that $\bpi''=\bpi$, which will prove that the class $\CL=\CL_{\bpi}$ is in $\Class_c(u)$, as desired. 

By \cref{lem:basification}, the reflections $t_k'$ and $t_{k+1}'$ belong to the parabolic subgroup $W'$. The reflection $t_k'$ appears to the left of $t_{k+1}'$ in the reduced $T$-word $\rw(\bpi'')=t_1\cdots t_{k-1}t_k't_{k+1}'t_{k+2}\cdots t_r$. By \cref{lem:i-1modk}, there exists $j^*\in[m-1]$ such that $t_k'=p_{j^*+1}$ and $t_{k+1}'=p_{j^*}$. Our goal is to demonstrate that $j=j^*$. Let $\mathscr{X}=\{\beta_{t_1},\ldots,\beta_{t_{k-1}},\beta_{t_{k+2}},\ldots,\beta_{t_r}\}$ be the set of positive roots corresponding to the reflections in $\rw(\bpi'')$ other than $p_{j^*}$ and $p_{j^*+1}$. By \cref{lem:linearly_independent}, the roots in $\mathscr{X}\cup\{\beta_{p_{j^*}},\beta_{p_{j^*+1}}\}$ are linearly independent. Therefore, the linear span of $\mathscr{X}\cup\{\beta_{p_{j^*}}\}$ is a hyperplane $H_1$, and the linear span of $\mathscr{X}\cup\{\beta_{p_{j^*+1}}\}$ is a hyperplane $H_2$. Let $H_1^-$ and $H_1^+$ be the closed half-spaces bounded by $H_1$, where $H_1^-$ contains $p_{1}$. Then $\beta_{p_1},\ldots,\beta_{p_{j^*}}$ are in $H_1^-$, while $\beta_{p_{j^*}},\ldots,\beta_{p_m}$ are in $H_1^+$. Similarly, let $H_2^-$ and $H_2^+$ be the closed half-spaces bounded by $H_2$, where $H_2^-$ contains $p_{1}$. Then $\beta_{p_1},\ldots,\beta_{p_{j^*+1}}$ are in $H_2^-$, while $\beta_{p_{j^*+1}},\ldots,\beta_{p_m}$ are in $H_2^+$. We have 
\[\Span_{\geq 0}\RRR(\DD_c(u))\subseteq\Span_{\geq 0}\RRR(\CL_{\bpi''})=\Span_{\geq 0}(\mathscr X\cup\{\beta_{p_{j^*}},\beta_{p_{j^*+1}}\})\subseteq H_1^+\cap H_2^-.\] The $r$ roots in $\RRR(\DD_c(u))$ are linearly independent (by \cref{lem:linearly_independent}), so we know that $\Span_{\geq 0}\RRR(\DD_c(u))$ cannot be contained in $H_1^-$ (since it would then be contained in $H_1$) and cannot be contained in $H_2^+$ (since it would then be contained in $H_2$). 
We also know that 
\begin{equation}\label{eq:span_containment}
\Span_{\geq 0}\RRR(\DD_c(u))\subseteq\Span_{\geq 0}\RRR(\CL)=\Span_{\geq 0}(\mathscr X\cup\{\beta_{p_{j}},\beta_{p_{j+1}}\}).
\end{equation} 
If $j<j^*$, then \eqref{eq:span_containment} forces $\Span_{\geq 0}\RRR(\DD_c(u))$ to be contained in $H_1^-$, which is impossible. Similarly, if $j>j^*$, then \eqref{eq:span_containment} forces $\Span_{\geq 0}\RRR(\DD_c(u))$ to be contained in $H_2^+$, which is impossible. Consequently,~$j=j^*$. 
\end{proof} 

We now establish some additional combinatorial properties of the SBDW triangulation. 

\begin{theorem}\label{thm:all_chains_appear}  
Fix a chain $\bpi\in\MCh(W,c)$. Then the set of $u \in W_c^+$ that are concordant with $\bpi$ is a nonempty interval in the left weak order. That is,
there exist $\omega_\downarrow(\bpi) ,\omega^\uparrow(\bpi)\in W$ such that \[\{u\in W_c^+:(u,\bpi)\in\Omega(W,c)\}=[\omega_\downarrow(\bpi),\omega^\uparrow(\bpi)]_{\LL}\neq\emptyset.\] 
\end{theorem} 

Before proving \cref{thm:all_chains_appear}, we need the following lemma. This lemma requires the notion of a \dfn{$c^{-1}$-antisortable element} of $W$, which is an element that is maximal in its $c^{-1}$-Cambrian class in the right weak order (each $c^{-1}$-Cambrian class is an interval in the right weak order). Equivalently, $v$ is $c^{-1}$-antisortable if and only if $vw_\circ$ is $c$-sortable (see \cite{ReadingSpeyerFans}). As before, we let $\wedge$ and $\vee$ denote the meet and join, respectively, in the right weak order on $W$. 

\begin{lemma}\label{lem:antisortable}
If $x$ and $x'$ are $c^{-1}$-sortable elements of $W$, then $T_\LL(x\wedge x')=T_\LL(x)\cap T_\LL(x')$. If $y$ and $y'$ are $c^{-1}$-antisortable elements of $W$, then $T\setminus T_\LL(y\vee y')=(T\setminus T_\LL(y))\cup (T\setminus T_\LL(y'))$.  
\end{lemma} 

\begin{proof}
The first statement appears as \cite[Remark~7.5]{ReadingSpeyerInfinte}. The map $v\mapsto vw_\circ$ is an anti-automorphism of the right weak order that restricts to a bijection from the set of $c^{-1}$-antisortable elements to the set of $c$-sortable elements. Hence, the second statement follows from the first (with $c$ replaced by~$c^{-1}$). 
\end{proof} 

\begin{proof}[Proof of \cref{thm:all_chains_appear}]
Let $Y=\{u\in W_c^+:(u,\bpi)\in\Omega(W,c)\}$. Let $\rw(\bpi)=t_1\cdots t_r$. For $1\leq i\leq r$, let $t_i'=t_1\cdots t_{i-1}t_it_{i-1}\cdots t_1$, and let $Y_i$ be the set of elements $w\in W$ such that $t_i\not\in T_\R(wt_1\cdots t_{i-1})$. For each $w\in W$, we have $\ell_S(wc)\leq\ell_S(w)+r$. Therefore, 
\begin{align*}
Y=\{u\in W:\ell_S(w)<\ell_S(wt_1)<\ell_S(wt_1t_2)<\cdots\ell_S(wt_1t_2\cdots t_r)\}=\bigcap_{i=1}^r Y_i. 
\end{align*} 
An element $w\in W$ is in $Y_i$ if and only if the regions $\BB$ and $t_{i-1}\cdots t_1w^{-1}\BB$ are on the same side of the hyperplane $H_{t_i}\in\mathcal H$ orthogonal to the root $\beta_{t_i}$. This occurs if and only if $w^{-1}\BB$ and $t_1\cdots t_{i-1}\BB$ are on the same side of the hyperplane $H_{t_i'}$ orthogonal to $\beta_{t_i'}$. Because $t_r'\cdots t_1'$ is a reduced $T$-word for $c$, \cref{lem:linearly_independent} tells us that the roots $\beta_{t_1'},\ldots,\beta_{t_r'}$ are linearly independent. This implies that the hyperplanes $H_{t_1'},\ldots,H_{t_r'}$ separate $V$ into $2^r$ different (full-dimensional) regions: one for each way of selecting a side of each of these hyperplanes. It follows that the set $Y=\bigcap_{i=1}^r Y_i$ is nonempty. Moreover, for each $1\leq i\leq r$, the set $Y_i$ is equal to either $\{w\in W:t_i'\not\in T_\R(w)\}$ (if $t_1\cdots t_{i-1}\BB$ and $\BB$ are on the same side of $H_{t_i'}$) or $\{w\in W:t_i'\in T_\R(w)\}$ (if $t_1\cdots t_{i-1}\BB$ and $\BB$ are on opposite sides of $H_{t_i'}$). Overall, this shows that there is a partition $\mathscr A^-\sqcup\mathscr A^+$ of the set $\{t_1',\ldots,t_r'\}$ such that 
\begin{equation}\label{eq:YAA}
Y=\{w\in W:\mathscr A^-\subseteq T_\R(w)\subseteq T\setminus\mathscr A^+\}.
\end{equation} 
Consequently, to show that $Y$ is an interval, we just need to show that it has a minimum element and a maximum element. As before, let $\wedge$ and $\vee$ denote the meet and join, respectively, in the right weak order. It follows from \cref{thm:Cambrian} that the inverses of the elements of $Y$ form a union of $c^{-1}$-Cambrian classes. This implies that every minimal (respectively, maximal) element of $Y$ is the inverse of a $c^{-1}$-sortable (respectively, $c^{-1}$-antisortable) element. 

Suppose by way of contradiction that $Y$ has distinct minimal elements $u_1$ and $u_2$. Then $u_1^{-1}$ and $u_2^{-1}$ are $c^{-1}$-sortable. Let $x=u_1^{-1}\wedge u_2^{-1}$. By \cref{lem:antisortable}, we have \[T_\R(x^{-1})=T_\LL(x)=T_\LL(u_1^{-1})\cap T_\LL(u_2^{-1})=T_\R(u_1)\cap T_\R(u_2).\] Since $u_1,u_2\in Y$, it follows from \eqref{eq:YAA} that $x^{-1}\in Y$. However, we have $x^{-1}<_\LL u_1$ and $x^{-1}<_\LL u_2$, so this contradicts the minimality of $u_1$ and $u_2$. 

Now suppose by way of contradiction that $Y$ has distinct maximal elements $v_1$ and $v_2$. Then $v_1^{-1}$ and $v_2^{-1}$ are $c^{-1}$-antisortable. Let $y=v_1^{-1}\vee v_2^{-1}$. By \cref{lem:antisortable}, we have \[T\setminus T_\R(y^{-1})=T\setminus T_\LL(y)=(T\setminus T_\LL(v_1^{-1}))\cup (T\setminus T_\LL(v_2^{-1}))=(T\setminus T_\R(v_1))\cup(T\setminus T_\R(v_2)).\] Since $v_1,v_2\in Y$, it follows from \eqref{eq:YAA} that $y^{-1}\in Y$. However, we have $y^{-1}>_\LL v_1$ and $y^{-1}>_\LL v_2$, so this contradicts the maximality of $v_1$ and $v_2$. 
\end{proof} 

Our final theorem in this section is as follows. 

\begin{theorem}\label{thm:trosable_thing} 
Let $w\in W$. There exists a chain $\bpi\in\MCh(W,c)$ such that ${w\in\omega_\downarrow(\bpi)\bpi}$ if and only if $w^{-1}$ is $c^{-1}$-sortable. 
\end{theorem} 

We split the proof into \cref{lem:trosable1,lem:trosable2} below. 

\begin{lemma}\label{lem:trosable1}
Suppose $w\in W$ is such that $w^{-1}$ is $c^{-1}$-sortable. Then there exists $\bpi\in\MCh(W,c)$ such that $w\in\omega_\downarrow(\bpi)\bpi$. 
\end{lemma} 

\begin{proof} 
Let $s_1\cdots s_r$ be a reduced $S$-word for $c$. We proceed by induction on $r$. We consider two cases. 

\medskip 

\noindent {\bf Case 1.} Assume there is a simple reflection $s_i$ such that $w$ belongs to the parabolic subgroup $W_{\langle s_i\rangle}$. Let $c_{\langle s_i\rangle}=s_1\cdots s_{i-1}s_{i+1}\cdots s_r$. Let $t=s_{r}\cdots s_{i+1}s_is_{i+1}\cdots s_r$. Then $c_{\langle s_i\rangle}t=c$. Note that $w^{-1}$ is $c_{\langle s_i\rangle}^{-1}$-sortable. By induction, there exists $\bpi'\in\MCh(W_{\langle s_i\rangle},c_{\langle s_i\rangle})$ such that $w\in\omega_\downarrow(\bpi')\bpi'$. Let $\bpi$ be the chain in $\MCh(W,c)$ obtained by appending $c$ to the end of $\bpi'$. In other words, the word $\rw(\bpi)$ is obtained by appending the letter $t$ to the end of $\rw(\bpi')$. Then $w\in\omega_\downarrow(\bpi')\bpi'\subseteq\omega_\downarrow(\bpi')\bpi$. Thus, it suffices to show that $\omega_\downarrow(\bpi')=\omega_\downarrow(\bpi)$. First, note that $\omega_\downarrow(\bpi')c_{\langle s_i\rangle}<_{\hB}\omega_\downarrow(\bpi')c$ because $\omega_\downarrow(\bpi')c_{\langle s_i\rangle}\in W_{\langle s_i\rangle}$. This shows that $(\omega_\downarrow(\bpi'),\bpi)\in\Omega(W,c)$, so $\omega_\downarrow(\bpi')\geq_\LL\omega_\downarrow(\bpi)$. This implies that $\omega_\downarrow(\bpi)\in W_{\langle s_i\rangle}$. Since $(\omega_\downarrow(\bpi),\bpi)\in\Omega(W,c)$, we must have $(\omega_\downarrow(\bpi),\bpi')\in\Omega(W_{\langle s_i\rangle},c_{\langle s_i\rangle})$, so $\omega_\downarrow(\bpi')\leq_\LL\omega_\downarrow(\bpi)$.  

\medskip 

\noindent {\bf Case 2.} Assume there does not exist a simple reflection $s_i$ such that $w\in W_{\langle s_i\rangle}$. Let $x=wc^{-1}$. The $c^{-1}$-sorting word for $w^{-1}$ has the word $s_r\cdots s_1$ as a prefix. Deleting this prefix yields the $c^{-1}$-sorting word for $x^{-1}$, so $x^{-1}$ is $c^{-1}$-sortable. In addition, we have $\ell_S(w)=\ell_S(wc^{-1})+\ell_S(c)$, so $x\in W_c^+$. Let $\bpi$ be a chain in the $\preceq_c$-decreasing commutation equivalence class $\DD_c(x)$. It follows from \cref{thm:Cambrian} that $x^{-1}$ is the only $c^{-1}$-sortable element whose inverse belongs to $\{u\in W_c^+:(u,\bpi)\in\Omega(W,c)\}$. This implies that $x=\omega_\downarrow(\bpi)$. Because $c\in\bpi$, we have  $w=xc=\omega_\downarrow(\bpi)c\in\omega_\downarrow(\bpi)\bpi$. 
\end{proof}

\begin{lemma}\label{lem:trosable2} 
Let $\bpi\in\MCh(W,c)$. Every element of $\omega_\downarrow(\bpi)\bpi$ is the inverse of a $c^{-1}$-sortable element. 
\end{lemma} 

\begin{proof}
Let $t_1\cdots t_r=\rw(\bpi)$ and $u=\omega_\downarrow(\bpi)$. Let $s_1\cdots s_r$ be a reduced $S$-word for $c$. Let $c'=cs_r$. Let $\bpi=\{e=\pi_0<_T\pi_1<_T\cdots<_T\pi_r=c\}$. We proceed by induction on $r$ and $\ell_S(u)$, considering two cases. 

\medskip 

\noindent {\bf Case 1.} Suppose $u\in W_{\langle s_r\rangle}$. Then $s_r$ is a right descent of $uc$ but not of $u$, so it follows from \cref{lem:all_have_descent} that $t_k=s_r$ for some $k\in[r]$. For $1\leq i\leq k-1$, let $t_i'=t_i$. For $k\leq i\leq r-1$, let $t_i'=s_rt_{i+1}s_r$. Then $t_1'\cdots t_{r-1}'$ is a reduced $T$-word for $c'$. Let $\bpi'=\{e=\pi_0'<_T\pi_1'<_T\cdots<_T\pi_{r-1}'=c'\}$ be the chain such that $\rw(\bpi')=t_1'\cdots t_{r-1}'$. For $0\leq i\leq k-1$, we have $\pi_i'=\pi_i$. For $k-1\leq i\leq r-1$, we have $\pi_i'=\pi_{i+1}s_r$. Let $u'=\omega_\downarrow(\bpi')$. Since $uc$ has $s_r$ as a right descent, \cref{lem:all_have_descent} implies that each of the elements $u\pi_k,\ldots,u\pi_{r-1}$ has $s_r$ as a right descent. Since $u\pi_{k+1}\lessdot_{\hB}\cdots \lessdot_{\hB}u\pi_{r}$, it follows from \cref{lem:Bruhat_descents} that $u\pi_k'\lessdot_{\hB}\cdots \lessdot_{\hB}u\pi_{r-1}'$. We also know already that $u\pi_0'\lessdot_{\hB}\cdots \lessdot_{\hB}u\pi_{k-1}'$. We have 
\[\ell_S(u\pi_k')=\ell_S(u\pi_{k+1})-1=\ell_S(u\pi_{k-1})+1=\ell_S(u\pi_{k-1}')+1,\] so $u\pi_{k-1}'\lessdot_{\hB}u\pi_k'$. This shows that $(u,\bpi')\in\Omega(W_{\langle s_r\rangle},c')$, so $u'\leq_\LL u$. It is straightforward to see that $u\leq_\LL u'$, so $u=u'$. By induction on $r$, we know that every element of the set $u\bpi'=u'\bpi'$ is the inverse of a $c^{-1}$-sortable element. Choose an integer $j$ with $0\leq j\leq r$; we will show that $u\pi_j$ is the inverse of a $c^{-1}$-sortable element. If $j\leq k-1$, then this is immediate because $u\pi_j=u\pi_j'$. Now suppose $k\leq j\leq r$. We have $u\pi_j=u\pi_{j-1}'s_r$, and $u\pi_{j-1}'\in W_{\langle s_r\rangle}$. Because $u\pi_{j-1}'$ is the inverse of a $c^{-1}$-sortable element, this implies that $u\pi_j$ is the inverse of $c^{-1}$-sortable element. 

\medskip 

\noindent {\bf Case 2.} Suppose $u\not\in W_{\langle s_r\rangle}$. Because $u=\omega_\downarrow(\bpi)$, we know that $u^{-1}$ is $c^{-1}$-sortable. This implies that $s_r$ is a right descent of $u$. Choose a chain $\bpi^*\in\DD_c(u)$, and let $t_1^*\cdots t_r^*=\rw(\bpi^*)$. Let $t_i'=s_rt_is_r$ and $t_i^\#=s_rt_i^*s_r$. Then $t_1'\cdots t_r'$ and $t_1^\#\cdots t_r^\#$ are reduced $T$-words for $s_rc'$. Let $\bpi',\bpi^\#\in\MCh(W,s_rc')$ be the chains such that $\rw(\bpi')=t_1'\cdots t_r'$ and $\rw(\bpi^\#)=t_1^\#\cdots t_r^\#$. It follows from the proof of \cref{thm:unique_decreasing} that $\bpi^\#\in\DD_{s_rc'}(us_r)$. We have 
\[\mathcal R(\mathcal C_{\bpi'})=s_r\mathcal R(\mathcal C_{\bpi})\quad\text{and}\quad \mathcal R(\DD_c(u))=s_r\mathcal R(\DD_{s_rc'}(us_r)).\] We know by \cref{thm:spans} that 
\[\Span_{\geq 0}\mathcal R(\mathcal C_{\bpi})\supseteq \Span_{\geq 0}\mathcal R(\DD_c(u)),\] so 
\[\Span_{\geq 0}\mathcal R(\mathcal C_{\bpi'})\supseteq \Span_{\geq 0}\mathcal R(\DD_{s_rc'}(us_r)).\] By \cref{thm:spans}, this implies that $(us_r,\bpi')\in\Omega(W,s_rc')$, so $u'\leq_\LL us_r$, where we let $u'=\delta_\downarrow(\bpi')$. In particular, since $s_r$ is not a right inversion of $us_r$, it is also not a right inversion of $u'$. 

We can reverse the preceding argument to show that $(u's_r,\bpi)\in\Omega(W,c)$, which implies that $u\leq_\LL u's_r$. Since $s_r$ is a right inversion of both $u's_r$ and $u$, this implies that $us_r\leq_\LL u'$. Hence, $u'=us_r$. 

We have 
\[u\bpi=\{u=u\pi_0\lessdot_{\hB}u\pi_1\lessdot_{\hB}\cdots \lessdot_{\hB} u\pi_r\}\quad\text{and}\quad u'\bpi'=\{us_r=u\pi_0s_r\lessdot_{\hB}u\pi_1s_r\lessdot_{\hB}\cdots \lessdot_{\hB} u\pi_rs_r\}.\] Since $\ell_S(us_r)=\ell_S(u)-1$, this implies that $s_r$ is a right descent of every element of $u\bpi$. We know by induction on $\ell_S(u)$ that every element of the set $us_r\bpi'$ is the inverse of an $(s_rc')^{-1}$-sortable element. Since $u\bpi=us_r\bpi's_r$, this implies that every element of $u\bpi$ is the inverse of a $c^{-1}$-sortable element. 
\end{proof} 

\section{The Artin and Dual Presentations}\label{sec:presentations}  
Recall that the Artin presentation of the braid group of $W$ is 
\[\mathbf{B}_{W,w_\circ}=\langle \mathbf{S} : [\mathrm{Red}_S(w_\circ)] \rangle,\] where $[\mathrm{Red}_S(w_\circ)]$ denotes the relations setting equal all reduced $S$-words for $w_\circ$ (with each $s\in S$ replaced by it corresponding $\mathbf{s}\in\mathbf{S}$). Recall that Bessis's dual presentation is
\[\B_{W,c} = \langle \dT : [\mathrm{Red}_T(c)] \rangle,\]
where $[\mathrm{Red}_T(c)]$ stands for the relations setting equal all reduced $T$-words for $c$ (with each $t \in T$ replaced by its corresponding $\mathbf{t} \in \dT)$.

\subsection{Artin relations from dual relations} 

One can view the Salvetti--Brady--Delucchi--Watt triangulation group-theoretically as an explicit geometric mechanism for passing from the Bessis dual presentation to the Artin presentation of the braid group $\mathbf{B}_W$. We illustrate this perspective in the following example. 

\begin{example}\label{exam:S3}
Let $W$ be the symmetric group $\SSS_3$, and let $S=\{s_1,s_2\}$, where $s_i$ denotes the transposition $(i\,\,i+1)$. Let $c=s_1s_2$. Let $t=s_1s_2s_1=s_2s_1s_2$. On the right of \cref{fig:S3} is the $c$-noncrossing partition lattice, which consists of the identity element $e$, the reflections $s_1,t,s_2$, and the Coxeter element $c$. The three maximal chains of this lattice are ${\color{Purple}\C_1}=\{e\lessdot_T s_1\lessdot_T c\}$, ${\color{Green}\C_2}=\{e\lessdot_T t\lessdot_T c\}$, and ${\color{red}\C_3}=\{e\lessdot_T s_2\lessdot_T c\}$. On the left of \cref{fig:S3} is the SBDW triangulation of the $\SSS_3$-permutahedron $\Perm_\yy$. This triangulation consists of the four simplices 
\[\conv(e{\color{Purple}\C_1}\yy),\quad \conv(e{\color{red}\C_3}\yy),\quad \conv(s_2{\color{Green}\C_2}\yy),\quad \conv(s_2{\color{Purple}\C_1}\yy).\] 

\medskip 

\begin{paracol}{2}
The Artin presentation of $\B_{\SSS_3}$ is \[\B_{\SSS_3}=\langle \mathbf{S}:{\color{Teal}\mathbf{s}_1\mathbf{s}_2\mathbf{s}_1}={\color{Orange}\mathbf{s}_2\mathbf{s}_1\mathbf{s}_2}\rangle.\] One can interpret the expression ${\color{Teal}\mathbf{s}_1\mathbf{s}_2\mathbf{s}_1}$ as the result of walking from $\yy$ to $\wo\yy$ along the left-hand edges of the $\SSS_3$-permutahedron $\Perm_\yy$ in \cref{fig:S3}. Similarly, the expression ${\color{Orange}\mathbf{s}_2\mathbf{s}_1\mathbf{s}_2}$ is the result of walking from $\yy$ to $\wo\yy$ along the right-hand edges of $\Perm_\yy$. The equality of these two expressions in the braid group corresponds to a $2$-cell, which is all of $\Perm_\yy$. 
\switchcolumn
Bessis's dual presentation of $\B_{\SSS_3}$ is \[\B_{\SSS_3}=\langle \mathbf{T}_c:{\color{Purple}\mathbf{s}_1\mathbf{s}_2}={\color{red}\mathbf{t}\mathbf{s}_1}={\color{Green}\mathbf{s}_2\mathbf{t}}\rangle.\] The expressions ${\color{Purple}\mathbf{s}_1\mathbf{s}_2}$, ${\color{red}\mathbf{t}\mathbf{s}_1}$, and ${\color{Green}\mathbf{s}_2\mathbf{t}}$ correspond to the chains in the lattice on the right-hand side of~\Cref{fig:S3}. One can also interpret each of these expressions as the result of walking along a simplex of the same color in \cref{fig:6BBW}. We will find it convenient to let ${\color{DarkBlue}\mathbf{c}}={\color{Purple}\mathbf{s}_1\mathbf{s}_2}={\color{red}\mathbf{t}\mathbf{s}_1}={\color{Green}\mathbf{s}_2\mathbf{t}}$.  
\end{paracol}

\medskip

Suppose we are given the dual presentation of the braid group. The SBDW triangulation allows us to deduce the Artin presentation by ``pushing through'' simplices as follows. Let us start with the expression ${\color{Teal}\mathbf{s}_1\mathbf{s}_2\mathbf{s}_1}$, which corresponds to the path $\yy\to s_1\yy\to s_1s_2\yy\to\wo\yy$ walking up the left-hand edges of $\Perm_\yy$. We can push through the lower-left simplex $\conv(e{\color{Purple}\C_1}\yy)$ so that we now traverse the path $\yy\to s_1s_2\yy\to\wo\yy$, which corresponds to the expression ${\color{DarkBlue}\mathbf{c}}{\color{Teal}\mathbf{s}_1}$. This pushing can be viewed as the chain of equalities 
\[{\color{Teal}\mathbf{s}_1\mathbf{s}_2\mathbf{s}_1}=({\color{Purple}\mathbf{s}_1\mathbf{s}_2}){\color{Teal}\mathbf{s}_1}={\color{DarkBlue}\mathbf{c}}{\color{Teal}\mathbf{s}_1}.\] We then push through the simplex $\conv(e{\color{red}\C_3}\yy)$ so that we traverse the path $\yy\to s_2\yy\to s_1s_2\yy\to\wo\yy$, which corresponds to the expression ${\color{Orange}\mathbf{s}_2}{\color{Brown}\mathbf{t}}{\color{Teal}\mathbf{s}_1}$. This pushing can be viewed as the chain of equalities 
\[{\color{DarkBlue}\mathbf{c}}{\color{Teal}\mathbf{s}_1}=({\color{red}\mathbf{s}_2\mathbf{t}}){\color{Teal}\mathbf{s}_1}={\color{Orange}\mathbf{s}_2}{\color{Brown}\mathbf{t}}{\color{Teal}\mathbf{s}_1}.\]
We then push through the simplex $\conv(s_2{\color{Green}\C_2}\yy)$ so that we traverse the path $\yy\to s_2\yy\to\wo\yy$, which corresponds to the expression ${\color{Orange}\mathbf{s}_2}{\color{DarkBlue}\mathbf{c}}$. This pushing can be viewed as the chain of equalities 
\[{\color{Orange}\mathbf{s}_2}{\color{Brown}\mathbf{t}}{\color{Teal}\mathbf{s}_1}={\color{Orange}\mathbf{s}_2}({\color{Green}\mathbf{t}\mathbf{s}_1})={\color{Orange}\mathbf{s}_2}{\color{DarkBlue}\mathbf{c}}.\] 
Finally, we push through the simplex $s_2{\color{Purple}\C_1}\yy$ so that we traverse the path $\yy\to s_2\yy\to s_2s_1\yy\to\wo\yy$, which corresponds to the expression ${\color{Orange}\mathbf{s}_2\mathbf{s}_1\mathbf{s}_2}$. This pushing can be viewed as the chain of equalities 
\[{\color{Orange}\mathbf{s}_2}{\color{DarkBlue}\mathbf{c}}={\color{Orange}\mathbf{s}_2}({\color{Purple}\mathbf{s}_1\mathbf{s}_2})={\color{Orange}\mathbf{s}_2\mathbf{s}_1\mathbf{s}_2}.\] 
In summary, pushing through simplices in the SBDW triangulation allows us to deduce the relation in the Artin presentation of the braid group via the following chain of equalities: 
\[{\color{Teal}\mathbf{s}_1\mathbf{s}_2\mathbf{s}_1}={\color{DarkBlue}\mathbf{c}}{\color{Teal}\mathbf{s}_1}={\color{Orange}\mathbf{s}_2}{\color{Brown}\mathbf{t}}{\color{Teal}\mathbf{s}_1}={\color{Orange}\mathbf{s}_2}{\color{DarkBlue}\mathbf{c}}={\color{Orange}\mathbf{s}_2\mathbf{s}_1\mathbf{s}_2}.\]
\end{example} 

The general strategy for deriving the Artin presentation from the Bessis dual presentation is similar to that presented in \cref{exam:S3}. Suppose we are given the Bessis dual presentation. Suppose in addition that we are given two reduced $S$-words $s_1\cdots s_N$ and $s_1'\cdots s_N'$ for $w_\circ$, and let $\mathbf{s}_1\cdots\mathbf{s}_N$ and $\mathbf{s}_1'\cdots\mathbf{s}_N'$ be their corresponding words over $\mathbf{S}$. These two words over $\mathbf{S}$ correspond to (oriented) paths from $\yy$ to $w_\circ\yy$ that use edges of $\Perm_\yy$. We can continuously deform the first path into the second by repeatedly pushing through simplices in $\SBDW_\yy(W,c)$. Each of these pushing moves corresponds to a relation in $[\mathrm{Red}_T(c)]$. Thus, this deformation is a topological manifestation of the fact that one can derive the equality of $\mathbf{s}_1\cdots\mathbf{s}_N$ and $\mathbf{s}_1'\cdots\mathbf{s}_N'$ in $\mathbf{B}_W$ from the relations in the Bessis dual presentation. As the words $\mathbf{s}_1\cdots\mathbf{s}_N$ and $\mathbf{s}_1'\cdots\mathbf{s}_N'$ were arbitrary, we can derive the entire Artin presentation from the Bessis dual presentation. 

\subsection{Dual lifts and the noncrossing Bruhat order}\label{subsec:bookkeeping} 
The SBDW triangulation leads naturally to the following new partial order on $W$. 

\begin{definition}\label{def:NCB}
The \dfn{$c$-noncrossing Bruhat order} is the partial order $\leq_{c\mathrm{NCB}}$ on $W$ defined so that there is a cover relation $u\lessdot_{c\mathrm{NCB}}v$ if and only if $u\lessdot_\hB v$ and a simplex in $\SBDW_\yy(W,c)$ contains an edge between $u\yy$ and $v\yy$.  
\end{definition} 

(Note that the $c$-noncrossing Bruhat order does not depend on the base point $\yy$.)   The noncrossing Bruhat order provides a ``bookkeeping'' mechanism for encoding the elements of $\dT$. More precisely, we will exhibit a labeling of the edges of the noncrossing Bruhat order with the elements of $\dT$.  

Let $\theta\colon\mathbf{B}_W\to W$ be the quotient map sending each generator $\mathbf{s}\in\mathbf{S}$ to its corresponding simple reflection $s\in S$. Let $\mathbf{R}=\theta^{-1}(T)$. Recall that $r$ is the rank of $W$. For $i\in[r-1]$, we can perform a \dfn{Hurwitz move} to a tuple $(\bt_1,\ldots,\bt_r)\in\mathbf{R}^r$ by replacing it with the tuple $(\bt_1,\ldots,\bt_{i-1},\bt_{i+1},\bt_{i+1}^{-1}\bt_i\bt_{i+1},\bt_{i+2},\ldots,\bt_{r})$, which is again in $\mathbf{R}^r$. The corresponding \dfn{inverse Hurwitz move} replaces $(\bt_1,\ldots,\bt_{i-1},\bt_{i+1},\bt_{i+1}^{-1}\bt_i\bt_{i+1},\bt_{i+2},\ldots,\bt_{r})$ with $(\bt_1,\ldots,\bt_r)$. The \dfn{Hurwitz orbit} of a tuple $(\bt_1,\ldots,\bt_r)\in\mathbf{R}^r$ is the set of tuples that can be obtained from $(\bt_1,\ldots,\bt_r)$ via a sequence of Hurwitz moves and inverse Hurwitz moves. 

Let $s_1\cdots s_r$ be a reduced $S$-word for $c$, and consider the tuple $(\bs_1,\ldots,\bs_r)\in\mathbf R^r$ (where $\bs_i\in\mathbf{S}$ corresponds to $s_i$). Let $t\in T$. Bessis \cite{bessis2003dualbraid} proved that the lift $\dt$ is the unique element of $\in\theta^{-1}(t)$ that appears in a tuple in the Hurwitz orbit of $(\bs_1,\ldots,\bs_r)$. 

As above, let us labeling each edge of $\Perm_\yy$ with the corresponding lift $\bs\in{\bf S}$. Consider a pair $(u,\bpi)\in\Omega(W,c)$ such that $w,w'\in w\bpi$. The $\preceq_c$-increasing chain \[\bpi^*=\{e<_Ts_1<_Ts_1s_2<_T\cdots<_Ts_1s_2\cdots s_r=c\}\] is concordant with every element of $W_c^+$, so in particular, it is concordant with $u$. The labels of the edges in the chain $u\bpi^*$ form the tuple $(\bs_1,\ldots,\bs_r)$. Using \cref{lem:basification}, we can transform the chain $u\bpi^*$ into $u\bpi$ via a sequence of moves, each of which only affects a diamond in the Bruhat order. We can associate a tuple in $\mathbf{R}^r$ to each chain appearing throughout this transofrmation. Indeed, each move changing a diamond in the Bruhat order corresponds to a composition of Hurwitz moves and inverse Hurwitz moves, so we can apply these moves to the tuple associated to one chain to obtain the tuple associated to the next chain. In the end, we obtain a tuple in $\mathbf{R}^r$ associated to $u\bpi$. This produces an element of $\mathbf R$ labeling each cover relation $w\lessdot_\hB w'$ with $w,w'\in u\bpi$. This label is a lift of the reflection $t=w^{-1}w'$. By the aforementioned result of Bessis, this label is the dual lift $\dt\in\dT$. In summary, we have labeled the edges of the noncrossing Bruhat order using the elements of $\dT$. 

\section{Future Work}\label{sec:future}

\subsection{Concordancy posets}
For $u\in W_c^+$, recall that $\Class_c(u)$ is the set of commutation classes of maximal chains in the $c$-noncrossing partition lattice that are concordant with $u$. Our combinatorial and geometric results outlined in \cref{subsec:properties-of-triang} motivate a natural partial order $\leqD$ on $\Class_c(u)$. For $\CL,\CL'\in\Class_c(u)$, we write $\CL\leqD\CL'$ if and only if $\Delta(\CL)\subseteq\Delta(\CL')$. It follows from \cref{thm:Cambrian,thm:spans} that the $\preceq_c$-increasing class $\mathcal I_c$ and the $\preceq_c$-decreasing class $\DD_c(u)$ are the maximum and minimum elements, respectively, of the poset $\Class_c(u)$. These posets appear to have rich combinatorial structure. 

\begin{conjecture}
For each element $u\in W_c^+$, the poset $\Class_c(u)$ is a semidistributive lattice. Moreover, the Hasse diagram of $\Class_c(u)$ is a regular graph of degree $r-1$. 
\end{conjecture} 

\begin{conjecture}
Two distinct classes $\CL,\CL'\in\Class_c(u)$ are adjacent in the Hasse diagram of the partial order $\leqD$ if and only if there is an SBDW simplex in the collection $\{\conv(u\bpi\yy):\bpi\in\CL\}$ that shares a facet with an SBDW simplex in the collection $\{\conv(u\bpi\yy):\bpi\in\CL'\}$. 
\end{conjecture} 

\subsection{Noncrossing weak order and noncrossing Bruhat order}\label{sec:other_orders}

Since the set $S$ of simple reflections is a subset of the set $T$ of all reflections, one might expect to be able to identify the Artin generating set $\mathbf{S}$ as a subset of the dual generating set $\dT$ (as elements of the braid group).  Doing so specifies an isomorphism $\BA\simeq \BD$, and we can leverage this inclusion to relate the two groups to the real geometry of the $W$-permutahedron.

\medskip

\begin{paracol}{2}
The \dfn{positive Artin monoid} $\BA^+$ is the submonoid of $\BA$ generated by $\mathbf{S}$; by Garside theory, this positive monoid injects into the braid group.  The Coxeter group $W$ lifts to $\BA^+$ by lifting each reduced $S$-word to the corresponding $\mathbf{S}$-word in Artin generators.  We denote this \defn{positive Artin lift} of $w \in W$ to $\BA^+$ by $\mathbf{w}$.

The map $w\mapsto \bw$ is an isomorphism from the right weak order on $W$ to the interval $[\mathbf{e},\mathbf{w}_\circ]_\mathbf{S}$ in $\BA^+$.
\switchcolumn
The \dfn{dual braid monoid} $\BD^+$ is the submonoid of $\BD$ generated by $\dT$; by Garside theory, this positive monoid injects into the braid group.  The $c$-noncrossing partitions lift to $\mathbf{B}_{W,c}^+$ by lifting each reduced $T$-word to the corresponding $\dT$-word in dual generators.  We denote this \defn{positive dual lift} of $\pi \in \NC(W,c)$ to $\mathbf{B}_{W,c}^+$ by $\boldsymbol{\pi}_{c}$. 

The map $t\mapsto\dt$ is an isomorphism from the $c$-noncrossing partition lattice to the interval $[\mathbf{e},\mathbf{c}_c]_{\dT}$ in $\mathbf{B}_{W,c}^+$.
\end{paracol} 

\medskip

Under the inclusion $\mathbf{S} \subseteq \dT$ identifying $\mathbf{s}$ with $\mathbf{s}_{c}$, 
the positive Artin monoid injects into the dual braid monoid.  We therefore have the inclusions \begin{equation}W=[e,w_\circ]_S \simeq [\mathbf{e},\mathbf{w}_\circ]_\mathbf{S} \hookrightarrow \BA^+ \hookrightarrow \BD^+,\label{eq:inclusion}\end{equation}  and it is natural to wonder what additional edges are introduced in $W$ and in $\BA^+$ when allowing the full set $\dT$ of dual generators---in other words, we consider the restriction of the Cayley graph of $\BD^+$ to the elements of $\BA^+$. 

\begin{definition}
The \dfn{$c$-noncrossing weak order} is the poset $\NCW(\BA^+,c)$ with underlying set $\BA^+$ defined by the cover relations $\mathbf{u}\lessdot_{c\mathrm{NCWeak}}\mathbf{v}$ whenever $\mathbf{v}=\mathbf{u}\dt$ for some $\dt\in\dT$.  We define the \defn{$c$-noncrossing weak order} on $W$ to be the poset  $\NCW(W,c)$ with underlying set $W$ defined so that the map $w\mapsto\mathbf{w}$ is an isomorphism from $\NCW(W,c)$ to the interval between $\mathbf{e}$ and $\mathbf{w}_\circ$ in $\NCW(\BA^+,c)$. 
\label{def:noncrossing_weak}
\end{definition}
\Cref{fig:nathan} illustrates the restriction of $\NCW(\B_{\SSS_3}^+,s_1s_2)$ to the interval from the identity to the full twist $\mathbf{w}_\circ^2$.  Restricting to the interval from $\mathbf{e}$ to $\mathbf{w}_\circ$ and replacing elements of $\B_{\SSS_3}$ with their image in $\SSS_3$ gives $\NCW(\SSS_3,s_1s_2)$.

\begin{conjecture}\label{conj:NCWeak_lattice}
The $c$-noncrossing weak order on $W$ is a lattice. 
\end{conjecture} 

Recall the noncrossing Bruhat order, which was introduced in \cref{def:NCB}. 

\begin{conjecture}\label{conj:NCB_NCWeak}
The $c$-noncrossing weak order on $W$ and the $c$-noncrossing Bruhat order on $W$ are isomorphic. 
\end{conjecture} 

\begin{conjecture}\label{conj:NCB_lattice}
The $c$-noncrossing Bruhat order on $W$ is a lattice. 
\end{conjecture}

Of course, if \cref{conj:NCB_NCWeak} holds, then \cref{conj:NCB_lattice,conj:NCWeak_lattice} are equivalent.

\begin{figure}[htbp]
\begin{center}
\includegraphics[height=9.500cm]{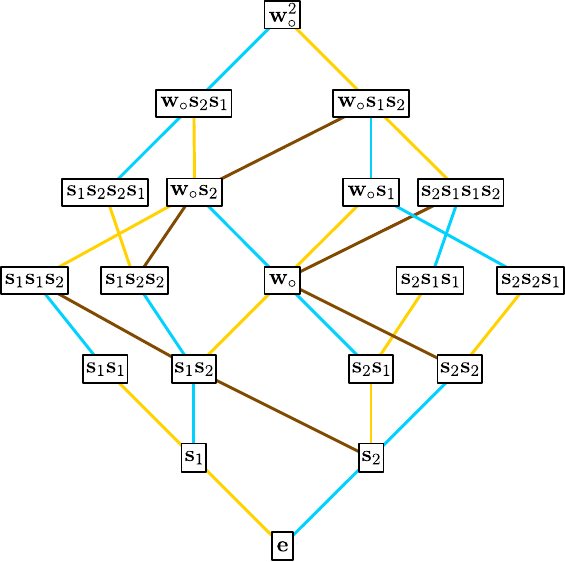}
\end{center}
\caption{The interval between the identity and the full twist in the noncrossing weak order $\NCW(\B_{\SSS_3}^+,c)$, where $c=s_1s_2$.  We have drawn $\mathbf{s}_1=(\mathbf{s}_1)_c$ edges in {\color{s1Yellow}yellow}, $\mathbf{s}_2=(\mathbf{s}_2)_c$ edges in {\color{s2Blue}blue}, and $\mathbf{t}_c=\mathbf{s}_1\mathbf{s}_2\mathbf{s}_1^{-1}=\mathbf{s}_2^{-1}\mathbf{s}_1\mathbf{s}_2$ edges in {\color{Brown}brown}.  
}\label{fig:nathan}
\end{figure}

\subsection{Pure braid group presentations}\label{subsec:pure} 

The pure braid group $\mathbf{P}_W$ is the kernel of the natural homomorphism $\theta\colon\mathbf{B}_W\to W$ defined in \cref{subsec:bookkeeping}. The center of $\mathbf{P}_W$ is generated by the \dfn{full twist} $\mathbb{c}:={\bf w}_\circ^2={\bf c}^h$, where ${\bf w}_\circ$ is the Artin lift of $w_\circ$ in $\mathbf{B}_W$ and ${\bf c}$ is the common lift of $c$ in $\mathbf{B}_W$ and $\mathbf{B}_{W,c}$.    For each $t\in T$, recall that $\dt$ denotes its lift in $\mathbf T_c$; let $\mathbbm{t}_c:={\mathbf t}_c^2$.   Let $\mathbb T_c:=\{\mathbbm{t}_c:t\in T\}$.

\cref{conj:pure} gives an explicit, elegant, type-uniform presentation of ${\bf P}_W$. To manipulate the relations in this presentation, it would be desirable to have a characterization of reduced $\mathbb{T}_{c}$-words for the full twist. This is the content of the next conjecture, which was formulated in unpublished notes by the third author.

\begin{conjecture}[{N.~Williams}]\label{conj:total_orders}
A total order $(t^{(1)},\ldots,t^{(N)})$ on the set of reflections $T$ is an EL-shelling order for $\NC(W,c)$ if and only if $\mathbbm{t}^{(1)}_c\cdots\mathbbm{t}^{(N)}_c$ is a reduced $\mathbb T_c$-word for the full twist $\mathbb{c}=\mathbf{w}_\circ^2=c^h$. 
\end{conjecture} 
The third author further conjectured a complete uniform characterization of the EL-shelling orders for $\NC(W,c)$, which was subsequently proven by Hugh Thomas in 2018 (up to a nontrivial finite check of the characterization up to Coxeter groups of rank 8; this check is easy to carry out for the classical types and low-rank exceptional types), generalizing work of Athanasiadis, Brady, and Watt~\cite{athanasiadis2007shellability}.  Note that there is always at least one EL shelling order: the inversion sequence of the $c$-sorting word for the long element.

We now give a speculative way to relate~\Cref{conj:total_orders} to the geometry of the braid group.  For any $\bpi \in \MCh(W,c)$, let $\mathbf{S}_{\bpi}$ be a formal copy of the reflections in $\rw(\bpi)$.  If we restrict the generators of the dual braid monoid $\mathbf{B}_{W,c}^+$ to the set $\mathbf{S}_{\bpi}$, we define a new monoid $\mathbf{B}_{W,\bpi}^+$; in slightly different language, the corresponding group relations have been studied in, for example,~\cite{palit2022dual,grant2017braid,barot2015reflection}.  Generalizing~\Cref{sec:other_orders} from the chain $\bpi^*$ for which $\mathbf{S}_{\bpi^*}=\mathbf{S}$ is the usual set of Artin generators, we can consider the interaction of some lift of $W$ to $\mathbf{B}_{W,\bpi}^+$ and the additional edges then induced by the full set of generators of $\mathbf{B}_{W,c}^+$.  One might hope that these additional edges allow us to construct a new simplicial complex from $\mathbf{S}_{\bpi}$ that is in some sense analogous to our triangulation of the Salvetti complex.  One might further expect for the combinatorics of this complex to then be governed by some EL-shelling order of the $c$-noncrossing partition lattice with basic reflections given by $\rw(\bpi)$, with decreasing chains indexing the ``innermost'' simplices.  The connection to~\Cref{conj:total_orders} is as follows: we see an obvious word in the braid group for the full twist $\mathbb{c}=c^h$ by walking in $\mathbf{S}_{\bpi}$ order around the outside of the complex; the difficult word in the pure braid group coming from the EL-shelling order will then be seen using the innermost simplices.

\begin{example}
For example, consider the interval $[\mathbf{e},\mathbf{w}_\circ]_\mathbf{S}$ for $\SSS_3$ with $\mathbf{S}_{\bpi}=\mathbf{S}$ (depicted in~\Cref{fig:nathan}).  The innermost edges and the triangulation allow us to convert the word $\mathbb{s}_2 \mathbb{t}_c \mathbb{s}_1$ using generators in $\mathbb{T}_c$ into an obvious word for the full twist using generators in $\mathbf{S}$:
\begin{align*}\mathbb{s}_2 \mathbb{t}_c \mathbb{s}_1=\mathbf{s}_2^2 \mathbf{t}_c^2 \mathbf{s}_1^2&= \mathbf{s}_2 (\mathbf{s}_2 \mathbf{t}_c) (\mathbf{t}_c \mathbf{s}_1) \mathbf{s}_1 = \mathbf{s}_2 (\mathbf{s}_1 \mathbf{s}_2) (\mathbf{s}_1 \mathbf{s}_2) \mathbf{s}_1 = \mathbb{c}.\end{align*}
\end{example} 

\subsection{Obstacles to an explicit homotopy}  \label{sec:obstacles}

Because the Salvetti complex $\Sal(W)$ and the Bessis--Brady--Watt complex $\Bessis(W,c)$ are each homotopy equivalent to the space $\Vreg$, it is natural to ask for an explicit homotopy between them. One of our original motivations for constructing the SBDW triangulation was to provide such a homotopy. However, we encountered several obstacles, which we now detail. 

The first issue we faced is that there are too many simplices used in our triangulation of $\Sal(W)$ to be directly comparable to the simplices in $\Bessis(W,c)$.  This problem is already evident for $\SSS_3$:
\begin{itemize}
\item for $\Sal(\SSS_3)$, each of the $|\SSS_3|=3!$ permutahedral hexagon faces of $\Sal(W)$ have been subdivided into four triangles, for a total of $24$ simplices;
\item on the other hand, each of the $|\SSS_3|=3!$ noncrossing partition lattices for $\Bessis(\SSS_3,c)$ contains only three triangles, for a total of $18$ simplices. 
\end{itemize}
The issue in this $\SSS_3$ example is that each of the simplices labeled by the $\preceq_c$-increasing chain appears in two permutahedral faces of $\Sal(\SSS_3)$.

More generally, in our triangulation of $\Sal(W)$, we wish to see only one copy of each simplex labeled by a particular factorization at each vertex $w \in W$.  Our triangulation gives us many duplicates---for example, our triangulation has $|W_c^+|$-many simplices labeled by the all-simples factorizations at each vertex $w \in W$.

When we tried to eliminate these duplicate simplices, we ran into a second issue.  The natural EL-shelling order of $\NC(W,c)$ coming from $\inv(\w_\circ(c))$ gives a natural order in which to glue together duplicate simplices.  The idea is that, starting with the all-simples factorization at any $w \in W$ that is glued on its two facets contained in the surface of $\Perm(W)$, we begin gluing together the offending duplicated simplices.  If every homology facet of $\NC(W,c)$---the $\preceq_c$-decreasing simplices---appeared only once in our triangulation of $\Perm(W)$, then it seems that this should work to give the desired homotopy.  Because of the existence of non-singleton postive $c^{-1}$-Cambrian classes (see~\Cref{thm:Cambrian}), some of these decreasing simplices can be repeated multiple times.  For example, in~\Cref{fig:TrianglesLinear}, the $\preceq_c$-decreasing chain whose corresponding reduced $T$-word is $(34)(14)(12)$ is concordant with both $u=s_2$ and $u=s_3s_2$.

The failure of these attempts suggests that our triangulation of $\Sal(W)$ should more accurately model the interaction of the positive Artin and dual braid monoids.  That is, we expect that we should replace our use of Bruhat order with the interaction from~\Cref{sec:other_orders} of weak order (and sortable elements) with the noncrossing partition lattice to redefine our triangulation of the permutahedron.  We plan to return to this in future work.

\section*{Acknowledgments}
We thank Mitchell Lee, Christian Stump, and Lauren Williams for extremely helpful conversations.  We especially thank Theo Douvropoulos for catching a serious oversight  we made in an earlier draft of this paper.  
Colin Defant was supported by the National Science Foundation under Award No.\ 2201907 and by a Benjamin Peirce Fellowship at Harvard University. Melissa Sherman-Bennett was partially supported by the National Science Foundation under Award No.~2444020.  Nathan Williams was partially supported by the National Science Foundation under Award No.~2246877 and would like to thank Theo once more here.

\appendix

\section{Upgrading triangulations to subdivisions} 

Our main result in this subsection is the following proposition, which allows us to continuously deform subdivisions under certain conditions. 

\begin{proposition}\label{prop:a-deform}
Let $(P_\tau)_{\tau\in[0,1]}$ be a collection of $d$-dimensional polytopes in $U$. Let $\JJ_1,\ldots,\JJ_m\subseteq\VV(P_0)$. Suppose that for each $\tau\in[0,1]$, there is a bijection $\varphi_\tau\colon\VV(P_0)\to\VV(P_\tau)$. Define $Q_\tau^i := \conv(\varphi_\tau(\JJ_i))$. Assume the following conditions hold: 
\begin{itemize}
\item $\{Q_0^i\}_{i=1}^m$ is a subdivision of $P_0$
\item The map $\varphi_0$ is the identity map on $\VV(P_0)$. 
\item For each $\vvv\in\VV(P_0)$, the map $\tau\mapsto\varphi_\tau(\vvv)$ is a continuous function from $[0,1]$ to $U$. 
\item For every $\tau\in[0,1]$, there is a poset isomorphism from $\Faces(P_0)$ to $\Faces(P_\tau)$ that agrees with $\varphi_\tau$ on $\VV(P_0)$. 
\item For every $\tau\in[0,1]$, there is a poset isomorphism from $\Faces(Q_0^i)$ to $\Faces(Q_\tau^i)$ that agrees with $\varphi_\tau$ on $\VV(Q_0^i)=\JJ_i$. 
\end{itemize} 
Then $\{Q_\tau^i\}_{i=1}^m$ is a subdivision of $P_\tau$ for every $\tau\in[0,1]$.  
\end{proposition}


Throughout this subsection, we preserve the hypotheses and notation from  \cref{prop:a-deform}. Before proving that proposition, we need the following lemmas.

\begin{lemma}\label{lem:a-common_facet}
Suppose $i,i'\in[m]$ are such that $Q_0^i$ and $Q_0^{i'}$ intersect along a common facet, which is necessarily $\conv(\JJ_i \cap \JJ_{i'})$. For every $\tau\in[0,1]$, the polytopes $Q_\tau^i$ and $Q_\tau^{i'}$ intersect along a common facet, which is $\conv(\varphi_\tau(\JJ_i\cap\JJ_{i'}))$. 
\end{lemma}
\begin{proof}
Let $\www \in \JJ_i\setminus\JJ_{i'}$ and $\www' \in \JJ_{i'}\setminus\JJ_i$. We note that by assumption, $\conv(\varphi_\tau(\JJ_i\cap\JJ_{i'}))$ is a facet of $Q_\tau^i$, and so in particular the affine hull of $\varphi_\tau(\JJ_i\cap\JJ_{i'})$ is a hyperplane $H_\tau$. Let $\mathcal T$ be the set of $\tau\in[0,1]$ such that $\varphi_\tau(\www)$ and $\varphi_\tau(\www')$ lie in opposite open half-spaces bounded by $H_\tau$. Because each map $\tau\mapsto\varphi_\tau(\vvv)$ for $\vvv\in\VV(P_0)$ is continuous, $\mathcal T$ is an open subset of $[0,1]$ that contains $0$. We wish to show that $\mathcal T=[0,1]$, so suppose by way of contradiction that this is not that case. Then $[0,1]\setminus\mathcal T$ contains its minimum, which is some number $\tau^\#$. Either $\varphi_{\tau^\#}(\www)$ or $\varphi_{\tau^\#}(\www')$ lies in the hyperplane $H_{\tau^{\#}}$. This contradicts the assumption that $Q_0^{i}$ and $Q_\tau^{i}$ (or $Q_0^{i'}$ and $Q_\tau^{i'}$) are combinatorially isomorphic. We conclude that $\mathcal T=[0,1]$, as desired. 

This shows that for all $\tau$, all elements of $\varphi_\tau(\JJ_i \setminus \JJ_{i'})$ lie in the opposite open half-space defined by $H_\tau$ from all elements of $\varphi_\tau(\JJ_{i'} \setminus \JJ_{i})$. So the intersection $Q_\tau^i \cap Q_\tau^{i'}$ is contained in $H_{\tau}$. On the other hand, $Q_\tau^i \cap H_\tau$ and $Q_\tau^{i'} \cap H_\tau$ are exactly the facet $\conv(\varphi_\tau(\JJ_i\cap\JJ_{i'}))$.
\end{proof} 

\begin{lemma}\label{lem:a-facet_or_boundary}
Choose $i \in [m]$, $\tau \in [0,1]$, and $F$ a facet of $Q_\tau^i$.
Either $F$ is contained in the boundary of $P_\tau$, or $F=Q_\tau^i\cap Q_\tau^j$ for some $j \neq i$. If \cref{prop:a-deform} holds for $(d-1)$-dimensional polytopes and $F$ is contained in the boundary of $P_\tau$, then for every $j\neq i$, the relative interior of $F$ does not intersect $Q_\tau^j$.
\end{lemma}

\begin{proof}

The conclusion of the lemma certainly holds if $\tau=0$. 

Now suppose $\tau>0$. By the fourth bulleted assumption of \cref{prop:a-deform}, there is a facet $F_0$ of $Q_0^i$ such that $\VV(F) = \varphi_\tau(\VV(F_0))$. If $F_0$ is not contained in the boundary of $P_0$, then no facet of $P_0$ contains $\VV(F_0)$. The third bulleted item of the hypotheses of \cref{prop:a-deform} imply that no facet of $P_\tau$ contains $\VV(F)$, so $F$ is not contained in the boundary of $P_\tau$ either. Also, since $\{Q_0^a\}_{a=1}^m$ is a subdivision, $F_0=Q_0^i \cap Q_0^j$ for some $j \neq i$. By \cref{lem:a-common_facet}, $Q_\tau^i \cap Q_\tau^j = \conv(\varphi_\tau(\JJ_i \cap \JJ_j)) = F$ as desired.

If $F_0$ is contained in the boundary of $P_0$, then some facet of $P_\tau$ contains $\VV(F_0)$. The third bulleted assumption of \cref{prop:a-deform} then implies that $F$ is contained in some facet of $P_\tau$.

For the final statement, suppose \cref{prop:a-deform} holds for $(d-1)$-dimensional polytopes. Choose a facet $R_0$ of $P_0$, and let $R_\tau$ be the facet of $P_\tau$ with $\VV(R_\tau)=\varphi_\tau(\VV(R))$. Note that $R_0$ and $R_\tau$ are combinatorially isomorphic by assumption. Let $H_\tau$ be the corresponding facet hyperplane. Let $I$ be the set of $i\in[m]$ such that $Q_0^i \cap H_0$ is a $(d-1)$-dimensional polytope. Then $\{Q_0^i \cap H_0\}_{i \in I}$ is a subdivision of $R_0$. The deformations $\{Q_\tau^i \cap H_\tau\}_{i \in I}$ satisfy the assumptions of \cref{prop:a-deform} using the same map $\varphi_\tau$ (restricted to $\VV(R_0)$). Thus, $\{Q_\tau^i \cap H_\tau\}_{i \in I}$ is a subdivision of $R_\tau$.

Suppose that $F$ is contained in $R_\tau$ and $F \cap Q_\tau^j\neq\emptyset$ for some $j\neq i$. This means $Q_\tau^j \cap H_\tau$ is nonempty and, since $H_\tau$ is a supporting hyperplane of $P_\tau \supset Q_\tau^j$, the polytope $Q_\tau^j \cap H_\tau$ is some face $F'$ of $Q_\tau^j$; $\VV(F')$ is exactly $\varphi_\tau(\JJ_j) \cap \VV(R_\tau)$. By the fourth bulleted assumption of \cref{prop:a-deform}, there is a face $F_0'$ of $Q_0^j$ whose vertices are $\varphi_\tau^{-1}(\VV(F'))= \JJ_j \cap R_0$. We have $F_0'=Q_0^j \cap H_0$. Since $\{Q_0^i\}_{i=1}^m$ is a subdivision, $F_0$ and $F_0'$ intersect in a common face contained in $H_0$. Since $\{Q_\tau^i \cap H_\tau\}_{i=1}^m$ is a subdivision, $F$ and $F'$ also intersect in a common face. Thus, the relative interior of $F$ does not intersect $Q_\tau^j$.





\end{proof}

\begin{proof}[Proof of \cref{prop:a-deform}]
We proceed by induction on dimension. The base case of $d=0$ is trivial. Now assume the proposition holds for dimension $d-1$.

Let $\mathcal T$ be the set of $\tau\in[0,1]$ such that $\{Q_\tau^i\}_{i=1}^m$ is not a subdivision of $P_\tau$; our goal is to show that $\mathcal T=\emptyset$. It follows from \cref{lem:union-covers-real} and the hypothesis of the proposition that $\bigcup_{i=1}^mQ_\tau^i=P_\tau$ for every $\tau\in[0,1]$. Therefore, $\mathcal T$ is equal to the set of $\tau\in[0,1]$ such that two simplices in $\Theta_\tau$ do not intersect along a common face. 

Let $\mathcal T'$ be the set of $\tau\in[0,1]$ such that there exist $Q_\tau^i, Q_\tau^{i'}$ that intersect in their interiors. Since interiors of polytopes are open sets, we know that $\mathcal T'$ is an open subset of $[0,1]$. We will show that $\mathcal T=\mathcal T'$.

We certainly have $\mathcal T'\subseteq\mathcal T$. Suppose that there exists $\tau\in\mathcal T\setminus\mathcal T'$; we will obtain a contradiction, which implies the desired equality. According to \cref{lem:facet-to-facet}, there exist $i,i'\in[m]$ such that $Q_\tau^i \cap Q_\tau^{i'}$ is $(d-1)$-dimensional but is not a common facet of $Q_\tau^i$ and $Q_\tau^{i'}$. Then $Q_\tau^i \cap Q_\tau^{i'}$ is properly contained in, and intersects the relative interior of, some facet $F$ of $Q_\tau^i$. Note that $F$ is not contained in the boundary of $P_\tau$. By \cref{lem:a-facet_or_boundary}, there exists $j\in[m]$ such that $F=Q_\tau^i\cap Q_\tau^j$. But then $Q_\tau^{i'}$ and $Q_\tau^j$ must intersect in their interior, which contradicts the assumption that $\tau\not\in\mathcal T'$.  

Now we would like to show that $\mathcal T'$ is empty, which would prove the proposition. Fix $\tau \in [0,1]$, and suppose for the sake of contradiction that some point $\ppp$ is in the interior of both $Q_\tau^i$ and $Q_\tau^j$. Let $X$ be the union of the $(d-2)$-dimensional faces of the polytopes in $\Theta=\{Q_\tau^i\}_{i=1}^m$. Choose a point $\ppp''$ in $U \setminus P_\tau$ generically, so that the line segment from $\ppp$ to $\ppp''$ does not intersect $X$, and let $\ppp'$ be the intersection of this line segment with the boundary of $P_\tau$. We claim that every point of the line segment $L$ from $\ppp$ to $\ppp'$ is contained in at least two polytopes in $\Theta$. This claim leads to a contradiction: if $\ppp' \in Q_\tau^a \cap Q_\tau^b$, then since $\ppp'$ is in the boundary of $P_\tau$ and is not in $X$, it is in the relative interior of a facet of $Q_\tau^a$ and a facet of $Q_\tau^b$. But this contradicts \cref{lem:a-facet_or_boundary}, as desired.

We now prove the claim. As you walk along $L$, you are contained in $Q_\tau^i \cap Q_\tau^j$ up until you arrive at a point $p$ where $L$ hits a facet $F$ of one of the polytopes, say $Q_\tau^i$, transversely. If $p = \ppp'$, we are done, so we assume $F$ is not on the boundary of $P_\tau$. \cref{lem:a-facet_or_boundary} guarantees that $F= Q_\tau^i \cap Q_\tau^a$ for some $a$. If somehow $p$ is also the point where $L$ hits a facet $F'$ of $Q_\tau^j$ transversely, again \cref{lem:a-facet_or_boundary} guarantees $F'= Q_\tau^j \cap Q_\tau^b$ for some $b$. It is impossible for $F$ to equal $F'$: in that case, the vertex sets agree, so the corresponding facets of $Q_0^i$ and $Q_0^j$ are equal, and since we have a subdivision for $\tau=0$, their intersection $Q_0^i \cap Q_0^j$ is their shared facet and \cref{lem:a-common_facet} implies that $Q_\tau^i \cap Q_\tau^j$ is a common facet, which is impossible as their interiors intersect.
 It is also impossible to have $a = b$ because $p \notin X$, so $p$ is in at most one facet of any polytope. In summary, if you go slightly past $p$ on $L$, you are still in two polytopes in $\Theta$. Repeating this argument as you proceed along $L$ proves the claim.
\end{proof}

\bibliographystyle{plain}
\bibliography{refs.bib}

\begin{thebibliography}{10}

\bibitem{apruzzese2020stability}
Paul~J Apruzzese and Kiyoshi Igusa.
\newblock Stability conditions for affine type ${A}$.
\newblock {\em Algebras and representation theory}, 23:2079--2111, 2020.

\bibitem{artin1947theory}
Emil Artin.
\newblock Theory of braids.
\newblock {\em Annals of Mathematics}, pages 101--126, 1947.

\bibitem{ASS}
Ibrahim Assem, Daniel Simson, and Andrzej Skowro\'nski.
\newblock {\em Elements of the representation theory of associative algebras. {V}ol. 1}, volume~65 of {\em London Mathematical Society Student Texts}.
\newblock Cambridge University Press, Cambridge, 2006.
\newblock Techniques of representation theory.

\bibitem{athanasiadis2007shellability}
Christos Athanasiadis, Thomas Brady, and Colum Watt.
\newblock Shellability of noncrossing partition lattices.
\newblock {\em Proceedings of the American Mathematical Society}, 135(4):939--949, 2007.

\bibitem{barot2015reflection}
Michael Barot and Bethany Marsh.
\newblock Reflection group presentations arising from cluster algebras.
\newblock {\em Transactions of the American Mathematical Society}, 367(3):1945--1967, 2015.

\bibitem{BGNTS2}
Nantel Bergeron, Lucas Gagnon, Philippe Nadeau, Hunter Spink, and Vasu Tewari.
\newblock personal communication.

\bibitem{BGNTS}
Nantel Bergeron, Lucas Gagnon, Philippe Nadeau, Hunter Spink, and Vasu Tewari.
\newblock The quasisymmetric flag variety: a toric complex on noncrossing partitions.
\newblock {\em arXiv preprint arXiv:2508.12171}, 2025.

\bibitem{bessis2003dualbraid}
David Bessis.
\newblock The dual braid monoid.
\newblock {\em Ann. Sci. \'Ecole Norm. Sup. (4)}, 36(5):647--683, 2003.

\bibitem{bessis2015finite}
David Bessis.
\newblock Finite complex reflection arrangements are ${K}(\pi,1)$.
\newblock {\em Annals of Mathematics}, pages 809--904, 2015.

\bibitem{bessisslides}
David Bessis.
\newblock Complex reflection groups are somehow real.
\newblock Finite {C}hevalley groups, reflection groups and braid groups - A conference in honour of {P}rofessor {J}ean {M}ichel, 21--23 September 2016.
\newblock EPFL, Lausanne, Switzerland.

\bibitem{BianeJosuatVerges}
Philippe Biane and Matthieu Josuat-Verg\`es.
\newblock Noncrossing partitions, {B}ruhat order and the cluster complex.
\newblock {\em Ann. Inst. Fourier (Grenoble)}, 69(5):2241--2289, 2019.

\bibitem{BjornerBrenti}
Anders Bj\"orner and Francesco Brenti.
\newblock {\em Combinatorics of {C}oxeter groups}, volume 231 of {\em Graduate Texts in Mathematics}.
\newblock Springer, New York, 2005.

\bibitem{Boretsky}
Jonathan Boretsky.
\newblock Totally nonnegative tropical flags and the totally nonnegative flag dressian.
\newblock {\em arXiv preprint arXiv:2208.09128}, 2023.

\bibitem{BEW}
Jonathan Boretsky, Christopher Eur, and Lauren Williams.
\newblock Polyhedral and tropical geometry of flag positroids.
\newblock {\em Algebra Number Theory}, 18(7):1333--1374, 2024.

\bibitem{slides}
T.~Brady, E.~Delucchi, and C.~Watt.
\newblock Slides: Triangulating the permutahedron.
\newblock \url{https://www.maths.tcd.ie/~vdots/HMI2017Brady.pdf}.
\newblock Accessed on 27 March, 2025.

\bibitem{brady2000artin}
Thomas Brady.
\newblock Artin groups of finite type with three generators.
\newblock {\em Michigan Mathematical Journal}, 47(2):313--324, 2000.

\bibitem{brady2001partial}
Thomas Brady.
\newblock A partial order on the symmetric group and new ${K}(\pi,1)$'s for the braid groups.
\newblock {\em Advances in Mathematics}, 161(1):20--40, 2001.

\bibitem{brady2018noncrossing}
Thomas Brady, Michael Falk, and Colum Watt.
\newblock Noncrossing partitions and {M}ilnor fibers.
\newblock {\em Algebraic \& Geometric Topology}, 18(7):3821--3838, 2018.

\bibitem{brady2002k}
Thomas Brady and Colum Watt.
\newblock ${K}(\pi 1)$'s for {A}rtin groups of finite type.
\newblock {\em Geometriae Dedicata}, 94:225--250, 2002.

\bibitem{BradyWatt2008}
Thomas Brady and Colum Watt.
\newblock Non-crossing partition lattices in finite real reflection groups.
\newblock {\em Trans. Amer. Math. Soc.}, 360(4):1983--2005, 2008.

\bibitem{brieskorn2006groupes}
Egbert Brieskorn.
\newblock Sur les groupes de tresses [d'apr{\`e}s {VI} {A}rnol'd].
\newblock In {\em S{\'e}minaire Bourbaki vol. 1971/72 Expos{\'e}s 400--417}, pages 21--44. Springer, 2006.

\bibitem{carter1972conjugacy}
Roger~W Carter.
\newblock Conjugacy classes in the {W}eyl group.
\newblock {\em Compositio Mathematica}, 25(1):1--59, 1972.

\bibitem{chang2024geometric}
Wen Chang, Yu~Qiu, and Xiaoting Zhang.
\newblock Geometric model for module categories of {D}ynkin quivers via hearts of total stability conditions.
\newblock {\em Journal of Algebra}, 638:57--89, 2024.

\bibitem{charney1995k}
Ruth Charney and Michael Davis.
\newblock The ${K}(\pi,1)$-problem for hyperplane complements associated to infinite reflection groups.
\newblock {\em Journal of the American Mathematical Society}, 8(3):597--627, 1995.

\bibitem{cohen2001monodromy}
Daniel Cohen.
\newblock Monodromy of fiber-type arrangements and orbit configuration spaces.
\newblock In {\em Forum Mathematicum}, volume~13, pages 505--530. Berlin; New York: De Gruyter, c1989, 2001.

\bibitem{DLRS-triang-book}
J.~De~Loera, J.~Rambau, and F.~Santos.
\newblock {\em Triangulations: Structures for algorithms and applications}.
\newblock Algorithms and Computation in Mathematics. Springer Berlin Heidelberg, 2010.

\bibitem{defant2022pop}
Colin Defant and Nathan Williams.
\newblock Pop, crackle, snap (and pow): some facets of shards.
\newblock {\em arXiv preprint arXiv:2209.05392}, 2022.

\bibitem{deligne1972immeubles}
Pierre Deligne.
\newblock Les immeubles des groupes de tresses g{\'e}n{\'e}ralis{\'e}s.
\newblock {\em Inventiones Mathematicae}, 17:273--302, 1972.

\bibitem{delucchi2024dual}
Emanuele Delucchi, Giovanni Paolini, and Mario Salvetti.
\newblock Dual structures on {C}oxeter and {A}rtin groups of rank three.
\newblock {\em Geometry \& Topology}, 28(9):4295--4336, 2024.

\bibitem{diaz2022total}
Yariana Diaz, Cody Gilbert, and Ryan Kinser.
\newblock Total stability and {A}uslander--{R}eiten theory for {D}ynkin quivers.
\newblock {\em arXiv preprint arXiv:2208.02445}, 2022.

\bibitem{digne2015pr}
Fran\c{c}ois Digne.
\newblock Pr{\'e}sentation des groupes de tresses purs et de certaines de leurs extensions.
\newblock {\em arXiv preprint arXiv:1511.08731}, 2015.

\bibitem{digne2001presentation}
Fran{\c{c}}ois Digne and V~Gomi.
\newblock Presentation of pure braid groups.
\newblock {\em Journal of Knot Theory and Its Ramifications}, 10(04):609--623, 2001.

\bibitem{DouvropoulosJosuatVerges1}
Theo Douvropoulos and Matthieu Josuat-Verg\`es.
\newblock The generalized cluster complex: refined enumeration of faces and related parking spaces.
\newblock {\em SIGMA Symmetry Integrability Geom. Methods Appl.}, 19:Paper No. 069, 40, 2023.

\bibitem{dyer1991bruhat}
Matthew Dyer.
\newblock On the ``{B}ruhat graph'' of a {C}oxeter system.
\newblock {\em Compositio Mathematica}, 78(2):185--191, 1991.

\bibitem{falk2000homotopy}
Michael Falk and Richard Randell.
\newblock On the homotopy theory of arrangements, {II}.
\newblock In {\em Arrangements--Tokyo 1998}, volume~27, pages 93--126. Mathematical Society of Japan, 2000.

\bibitem{garside1969braid}
Frank Garside.
\newblock The braid group and other groups.
\newblock {\em The Quarterly Journal of Mathematics}, 20(1):235--254, 1969.

\bibitem{grant2017braid}
Joseph Grant and Bethany~Rose Marsh.
\newblock Braid groups and quiver mutation.
\newblock {\em Pacific Journal of Mathematics}, 290(1):77--116, 2017.

\bibitem{facet-to-facet}
Peter~M. Gruber and Sergej~S. Ryškov.
\newblock Facet-to-facet implies face-to-face.
\newblock {\em European Journal of Combinatorics}, 10(1):83--84, 1989.

\bibitem{han2020spherical}
Huhe Han and Takashi Nishimura.
\newblock Spherical separation theorem.
\newblock {\em arXiv preprint arXiv:2002.06558}, 2020.

\bibitem{hillevideo}
Lutz Hille.
\newblock Stable representations for {D}ynkin quivers.
\newblock \url{https://www.birs.ca/events/2018/5-day-workshops/18w5178/videos/watch/201811011730-Hille.html}, 2018.
\newblock Video recorded at the Banff Intenational Research Station Workshop 18w5178: Stability Conditions and Representation Theory of Finite-Dimensional Algebras. Accessed 5 June 2025.

\bibitem{Hohlweg}
Christophe Hohlweg and Jean-Philippe Labb\'e.
\newblock On inversion sets and the weak order in {C}oxeter groups.
\newblock {\em European J. Combin.}, 55:1--19, 2016.

\bibitem{HuangHu}
Pengfei Huang and Zhi Hu.
\newblock Stability and indecomposability of the representations of quivers of {$A_n$}-type.
\newblock {\em Comm. Algebra}, 48(7):2905--2919, 2020.

\bibitem{Humphreys}
James~E. Humphreys.
\newblock {\em Reflection groups and {C}oxeter groups}, volume~29 of {\em Cambridge Studies in Advanced Mathematics}.
\newblock Cambridge University Press, Cambridge, 1990.

\bibitem{JosuatVerges}
Matthieu Josuat-Verg\`es.
\newblock Refined enumeration of noncrossing chains and {H}ook formulas.
\newblock {\em Ann. Comb.}, 19(3):443--460, 2015.

\bibitem{JLLO}
Michael Joswig, Georg Loho, Dante Luber, and Jorge~Alberto Olarte.
\newblock Generalized permutahedra and positive flag dressians.
\newblock {\em Int. Math. Res. Not. IMRN}, (19):16748--16777, 2023.

\bibitem{keller2011cluster}
Bernhard Keller.
\newblock On cluster theory and quantum dilogarithm identities.
\newblock {\em Representations of algebras and related topics}, 5:85, 2011.

\bibitem{kinser2022total}
Ryan Kinser.
\newblock Total stability functions for type ${A}$ quivers.
\newblock {\em Algebras and Representation Theory}, pages 1--11, 2022.

\bibitem{KSSB}
Allen Knutson, Mario Sanchez, and Melissa Sherman-Bennett.
\newblock Permutahedral subdivisions and class formulas from {C}oxeter elements.
\newblock In {\em Proceedings of the 37th Conference on Formal Power Series and Algebraic Combinatorics (FPSAC 2025)}, volume 93B of {\em S{\'e}minaire Lotharingien de Combinatoire}, 2025.

\bibitem{KodamaWilliams}
Yuji Kodama and Lauren Williams.
\newblock The full {K}ostant-{T}oda hierarchy on the positive flag variety.
\newblock {\em Comm. Math. Phys.}, 335(1):247--283, 2015.

\bibitem{Kostant}
Bertram Kostant.
\newblock Lie algebra cohomology and the generalized {B}orel-{W}eil theorem.
\newblock {\em Ann. of Math. (2)}, 74:329--387, 1961.

\bibitem{marczinzik2023total}
Ren{\'e} Marczinzik.
\newblock On total stability conditions for {D}ynkin quivers.
\newblock {\em Bulletin of the London Mathematical Society}, 55(2):886--891, 2023.

\bibitem{markushevich1991d}
Dmitri~G Markushevich.
\newblock The $d_n$ generalized pure braid group.
\newblock {\em Geometriae Dedicata}, 40(1):73--96, 1991.

\bibitem{mccammond2017artin}
Jon McCammond and Robert Sulway.
\newblock Artin groups of {E}uclidean type.
\newblock {\em Inventiones Mathematicae}, 210(1):231--282, 2017.

\bibitem{NTS}
Philippe Nadeau, Hunter Spink, and Vasu Tewari.
\newblock The geometry of quasisymmetric coinvariants.
\newblock {\em arXiv preprint arXiv:2410.12643}, 2024.

\bibitem{nakamura1983note}
Tokushi Nakamura.
\newblock A note of the ${K}(\pi,1)$-property of the orbit space of the unitary reflection group ${G}(m,l,n)$.
\newblock {\em Sci. Papers College of Arts and Sciences, Univ. Tokyo}, 33:1--6, 1983.

\bibitem{orlik1988discriminants}
Peter Orlik and Louis Solomon.
\newblock Discriminants in the invariant theory of reflection groups.
\newblock {\em Nagoya Mathematical Journal}, 109:23--45, 1988.

\bibitem{palit2022dual}
Priyojit Palit.
\newblock {\em Dual Braid Presentations and Cluster Algebras}.
\newblock The University of Texas at Dallas, 2022.

\bibitem{paris2014k}
Luis Paris.
\newblock ${K}(\pi,1)$ conjecture for {A}rtin groups.
\newblock In {\em Annales de la Facult{\'e} des sciences de Toulouse: Math{\'e}matiques}, volume~23, pages 361--415, 2014.

\bibitem{pilaud2015brick}
Vincent Pilaud and Christian Stump.
\newblock Brick polytopes of spherical subword complexes and generalized associahedra.
\newblock {\em Advances in Mathematics}, 276:1--61, 2015.

\bibitem{qiu2015stability}
Yu~Qiu.
\newblock Stability conditions and quantum dilogarithm identities for {D}ynkin quivers.
\newblock {\em Advances in Mathematics}, 269:220--264, 2015.

\bibitem{qiu2025geometric}
Yu~Qiu and Xiaoting Zhang.
\newblock Geometric classification of total stability conditions.
\newblock {\em Mathematische Zeitschrift}, 309(4):1--43, 2025.

\bibitem{ReadingCambrian}
Nathan Reading.
\newblock Cambrian lattices.
\newblock {\em Adv. Math.}, 205(2):313--353, 2006.

\bibitem{reading2007clusters}
Nathan Reading.
\newblock Clusters, {C}oxeter-sortable elements and noncrossing partitions.
\newblock {\em Transactions of the American Mathematical Society}, 359(12):5931--5958, 2007.

\bibitem{reading2007sortable}
Nathan Reading.
\newblock Sortable elements and {C}ambrian lattices.
\newblock {\em Algebra universalis}, 56(3):411--437, 2007.

\bibitem{reading2011noncrossing}
Nathan Reading.
\newblock Noncrossing partitions and the shard intersection order.
\newblock {\em Journal of Algebraic Combinatorics}, 33(4):483--530, 2011.

\bibitem{reading2011sortable}
Nathan Reading and David Speyer.
\newblock Sortable elements in infinite {C}oxeter groups.
\newblock {\em Transactions of the American Mathematical Society}, 363(2):699--761, 2011.

\bibitem{ReadingSpeyerFans}
Nathan Reading and David~E. Speyer.
\newblock Cambrian fans.
\newblock {\em J. Eur. Math. Soc. (JEMS)}, 11(2):407--447, 2009.

\bibitem{ReadingSpeyerInfinte}
Nathan Reading and David~E. Speyer.
\newblock Sortable elements in infinite {C}oxeter groups.
\newblock {\em Trans. Amer. Math. Soc.}, 363(2):699--761, 2011.

\bibitem{reineke2003harder}
Markus Reineke.
\newblock The {H}arder-{N}arasimhan system in quantum groups and cohomology of quiver moduli.
\newblock {\em Inventiones Mathematicae}, 152(2):349--368, 2003.

\bibitem{salvetti1987topology}
Mario Salvetti.
\newblock Topology of the complement of real hyperplanes in $\mathbb{C}^n$.
\newblock {\em Inventiones Mathematicae}, 88(3):603--618, 1987.

\bibitem{salvetti1994homotopy}
Mario Salvetti.
\newblock The homotopy type of {A}rtin groups.
\newblock {\em Mathematical Research Letters}, 1(5):565--577, 1994.

\bibitem{steinberg:FiniteReflectionGroups}
Robert Steinberg.
\newblock Finite reflection groups.
\newblock {\em Trans. Amer. Math. Soc.}, 91:493--504, 1959.

\bibitem{Stembridge}
John Stembridge.
\newblock Folding by automorphisms.
\newblock \url{https://dept.math.lsa.umich.edu/∼jrs/papers/folding.pdf}.

\bibitem{stump2025cataland}
Christian Stump, Hugh Thomas, and Nathan Williams.
\newblock {\em Cataland: {W}hy the {F}uss?}, volume 305.
\newblock American Mathematical Society, 2025.

\bibitem{TsukermanWilliams}
E.~Tsukerman and L.~Williams.
\newblock Bruhat interval polytopes.
\newblock {\em Adv. Math.}, 285:766--810, 2015.

\bibitem{lek1983homotopy}
Harm Van~der Lek.
\newblock {\em The homotopy type of complex hyperplane complements}.
\newblock PhD thesis, Nijmegen, 1983.

\bibitem{nathansite}
Nathan Williams and Google Gemini.
\newblock {B}essis and {S}alvetti.
\newblock \url{https://personal.utdallas.edu/~nxw170830/docs/bessalvetti.html}.
\newblock Accessed: August 25, 2025.

\end{thebibliography}

\end{document}